\def\coralreport{1}

\documentclass[11pt]{article}

\usepackage{ifthen}
\usepackage{array}
\newcolumntype{P}[1]{>{\centering\arraybackslash}p{#1}}
\usepackage{multirow}

\usepackage{caption}
\usepackage{subcaption}

\usepackage{lipsum}
\usepackage{amsfonts}
\usepackage{graphicx}
\usepackage{epstopdf}
\usepackage{amsmath,mathtools}
\usepackage{algorithm}
\usepackage{algpseudocode}
\usepackage{amsmath,amssymb,amsfonts,amsthm,bm}
\usepackage[inline]{enumitem}   
\usepackage{tcolorbox}
\usepackage{mathtools}
\newcommand{\st}{\text{s.t.}}

\newcommand{\RR}{\mathbb{R}}
\newcommand{\NN}{\mathbb{N}}

\newcommand{\XXX}{\mathcal{X}}

\newcommand{\tbar}{{\bar t}}
\newcommand{\that}{{\hat t}}

\newcommand{\Gammabar}{{\bar \Gamma}}

\newcommand{\Qbar}{{\bar Q}}

\newcommand{\Zbar}{{\bar Z}}

\newcommand{\fbar}{{\bar f}}
\newcommand{\cbar}{{\bar c}}
\newcommand{\gbar}{{\bar g}}
\newcommand{\Jbar}{{\bar J}}
\newcommand{\Wbar}{{\bar W}}
\newcommand{\Hbar}{{\bar H}}
\newcommand{\mbar}{{\bar m}}
\newcommand{\dbar}{{\bar d}}
\newcommand{\phibar}{{\bar \phi}}
\newcommand{\taubar}{{\bar \tau}}
\newcommand{\vbar}{{\bar v}}
\newcommand{\ubar}{{\bar u}}
\newcommand{\ybar}{{\bar y}}

\newcommand{\Pbar}{{\bar P}}

\newcommand{\dhat}{{\hat d}}
\newcommand{\alphahat}{{\hat \alpha}}
\newcommand{\alphabar}{{\bar \alpha}}

\newcommand{\epsktau}{\vareps_{\tau,k}}
\newcommand{\epskd}{\varepsilon_{d,k}}
\newcommand{\epsku}{\varepsilon_{u,k}}
\newcommand{\epskv}{\varepsilon_{v,k}}
\newcommand{\Ekmtwo}{E_k^{m, 2}}
\newcommand{\Ekmone}{E_k^{m, 1}}
\newcommand{\Ekm}{E_k^m}
\newcommand{\epskg}{\vareps_{g,k}}
\newcommand{\epskc}{\vareps_{c,k}}
\newcommand{\epskJ}{\vareps_{J,k}}

\newcommand{\calEkmone}{\mathcal{E}_k^{m, 1}}
\newcommand{\calEkmtwo}{\mathcal{E}_k^{m, 2}}
\newcommand{\calEkm}{\mathcal{E}_k^m}
\newcommand{\sigmabar}{{\bar \sigma}}

\newcommand{\emOneOne}{e_1^{m, 1}}
\newcommand{\emOneTwo}{e_2^{m, 1}}
\newcommand{\emTwoOne}{c_1^{m, 2}}
\newcommand{\emTwoTwo}{c_2^{m, 2}}

\newcommand{\xizero}[1]{\xi_2(#1)}
\newcommand{\xione}{\xi_{1, k}}
\newcommand{\xitwo}{\xi_5}
\newcommand{\xithree}{\xi_3}
\newcommand{\xifour}{\xi_4}
\newcommand{\xifive}{\xi_6}

\newcommand{\eps}{\epsilon}
\newcommand{\vareps}{\varepsilon}

\makeatletter

\ifpdf
  \DeclareGraphicsExtensions{.eps,.pdf,.png,.jpg}
\else
  \DeclareGraphicsExtensions{.eps}
\fi

\newtheorem{theorem}{Theorem}
\newtheorem{assumption}{Assumption}[section]
\newtheorem{corollary}{Corollary}[section]

\newtheorem{lemma}{Lemma}[section]
\newtheorem{remark}{Remark}[section]

\allowdisplaybreaks

\newcommand{\ftrue}{\fbar}
\newcommand{\gtrue}{\gbar}
\newcommand{\ctrue}{\cbar}
\newcommand{\Jtrue}{\Jbar}
\newcommand{\vtrue}{\vbar}
\newcommand{\utrue}{\ubar}
\newcommand{\dtrue}{\dbar}
\newcommand{\ytrue}{\ybar}
\newcommand{\phitrue}{\phibar}
\newcommand{\mtrue}{\mbar}
\newcommand{\tautrue}{\taubar}

\newcommand{\Wtrue}{\Wbar}
\newcommand{\Htrue}{\Hbar}
\newcommand{\Ztrue}{\Zbar}
\newcommand{\Ptrue}{\Pbar}
\newcommand{\Gammatrue}{\Gammabar}

\newcommand{\JCfrac}{{\tfrac{\|J_k^Tc_k\|}{\|\Jtrue_k^T\ctrue_k\|}}}

\usepackage{hyperref}
\hypersetup{colorlinks=true}

\ifthenelse{\coralreport = 1}{
\usepackage{isetechreport}
\def\reportyear{25T}
\def\reportno{002}
\def\revisionno{0}
\def\originaldate{February 16, 2025}
\coralfalse
\cvcrfalse

}{}

\begin{document}

\title{An Interior-Point Algorithm for Continuous Nonlinearly Constrained Optimization with Noisy Function and Derivative Evaluations}






\ifthenelse{\coralreport = 1}{

\author{Frank E.~Curtis\thanks{E-mail: \texttt{frank.e.curtis@lehigh.edu},
supported by ONR grant N00014-24-1-2703}}
\affil{Department of Industrial and Systems Engineering, Lehigh University}
\author{Shima Dezfulian\thanks{E-mail: \texttt{dezfulian@u.northwestern.edu},
supported by NSF Grant DMS-2012410}}
\author{Andreas Waechter\thanks{E-mail: \texttt{andreas.waechter@northwestern.edu},
supported by NSF Grant DMS-2012410}}
\affil{Department of Industrial Engineering and Management Sciences, Northwestern University}

\titlepage

}{

}

\maketitle

\begin{abstract}
  An algorithm based on the interior-point methodology for solving continuous nonlinearly constrained optimization problems is proposed, analyzed, and tested.  The distinguishing feature of the algorithm is that it presumes that only noisy values of the objective and constraint functions and their first-order derivatives are available.  The algorithm is based on a combination of a previously proposed interior-point algorithm that allows inexact subproblem solutions and recently proposed algorithms for solving bound- and equality-constrained optimization problems with only noisy function and derivative values.  It is shown that the new interior-point algorithm drives a stationarity measure below a threshold that depends on bounds on the noise in the function and derivative values.  The results of numerical experiments show that the algorithm is effective across a wide range of problems.
\ifthenelse{\coralreport = 0}{
\bigskip\noindent
{\bf Keywords:} BS  
}{}

\end{abstract}

\section{Introduction}

Interior-point methods have been studied extensively and have proved to be successful in practice for solving continuous nonlinearly constrained optimization problems. Two prominent categories of interior-point methods are those based on line-search 
\cite{CurtScheWaec10,wachter2005line,wachter2006implementation,yamashita1998globally} 
and trust-region 
\cite{byrd2000trust,byrd1999interior,yamashita2005globally} 
mechanisms.  Many state-of-the-art nonlinear optimization solvers, such as IPOPT \cite{wachter2006implementation} and KNITRO \cite{byrd2006k}, are based on interior-point methods.

An important feature of these contemporary interior-point methods is that they rely on exact values of the objective and constraint functions and their first-order derivatives.  On the other hand, there has been little work done on developing and analyzing such algorithms for cases when function evaluations are affected by noise.  In noisy settings, only estimates of the function and derivative values are available.  Generally speaking, noise can be stochastic or deterministic.  By stochastic noise, we refer to situations in which it is reasonable to model a function value through a probability distribution, in which case any realization of a value is a realization of a random variable.  By deterministic noise, on the other hand, we are referring to situations in which any request for a function or derivative value with respect to an input value always results in the same \emph{noisy} value that approximates some original \emph{true} value.  The original function of interest may be presumed to be, e.g., continuously differentiable, but the corresponding noisy function may be nonsmooth or even discontinuous. Computational noise---say, from a numerical simulation---is an example of deterministic noise \cite{shi2021methods}, where it might only be assumed that the noise is bounded over the domain of a given function.

Recently, a number of algorithms have been proposed and analyzed for solving unconstrained, bound-constrained, or equality-constrained optimization problems in the presence of stochastic or deterministic noise.  For further information, we refer the reader to \cite{berahas2019derivative,berahas2021global,berahas2021sequential,curtis2021inexact,curtis2023sequential,curtis2019stochastic,na2023inequality,oztoprak2023constrained,paquette2020stochastic,qiu2023sequential,sun2023trust,xie2020analysis}.  There has also been some recent work on the design of stochastic-gradient-based interior-point methods \cite{CurtJianWang24,CurtKungRobiWang23}.  However, to our knowledge, there have not yet been extensions of recent algorithms for the bound- or equality-constrained setting with deterministic noise to the \emph{nonlinear-inequality-constrained setting with deterministic noise}.  In this paper, we propose, analyze, and test such an algorithm.  Due to its practical success in noiseless settings, our algorithm is based on the interior-point methodology.  As a result, we expect our approach to yield good performance in practice, and expect that our proposed techniques can be incorporated readily into state-of-the-art software packages for solving continuous nonlinearly constrained optimization problems.

\subsection{Contributions}\label{sec.contributions}

This paper builds primarily upon the interior-point algorithm proposed in \cite{dezfulian2024convergence} for solving bound-constrained optimization problems in the presence of deterministic noise by extending the methodology and theoretical results to settings involving nonlinear inequality constraints.  The method proposed in \cite{dezfulian2024convergence} is a line-search interior-point algorithm.  Proposing such a method for solving nonlinearly constrained problems comes with significant additional challenges.  For example, as is well known even in settings when exact function and derivative values are available, line-search interior-point methods that compute the search directions through Newton-based techniques may fail to converge due to inconsistency between the step computation and enforcement of nonnegativity constraints on variables~\cite{wachter2000failure}.  In addition, if the Jacobian of the constraint function at a given point can be rank deficient, then a Newton-based approach may result in an ill-posed subproblem.  To address these issues, we follow the step-decomposition approach proposed for the interior-point method in \cite{CurtScheWaec10}.

Another challenge that we face as we transition from the bound-constrained case to the nonlinear-inequality-constrained setting (with deterministic noise) is assessing the progress of the algorithm in terms of  minimizing the objective function while aiming to satisfy the constraints.  For solving bound-constrained problems with the interior-point method in \cite{dezfulian2024convergence}, the barrier subproblem is unconstrained and satisfaction of the bounds is enforced through a fraction-to-the-boundary rule.  Thus, the barrier objective function measures the progress of the algorithm and the fraction-to-the-boundary rule ensures satisfaction of the constraints.  However, in the nonlinearly constrained case, it is typically inefficient to enforce feasibility of the constraints at every algorithm iterate, meaning that it is preferred to employ a so-called infeasible algorithm.  In the context of such an algorithm, one needs a measure of progress that balances the improvement in the objective function and constraint satisfaction.  In noiseless settings, this is typically accomplished by defining a merit function that is a weighted sum of the objective function and a measure of constraint violation.  The weight assigned to the objective function relative to the constraint satisfaction \cite{nocedal1999numerical} is called the merit parameter.  Typically, the merit parameter is updated in an adaptive manner in such a way that it guarantees that the original optimization problem is being solved.  Extending such an approach and its analysis to the noisy setting is the main contribution of this work.

We mention upfront that our proposed algorithm and analysis leave certain natural questions unanswered.  In particular, we focus our attention in this paper on an algorithm for solving a single barrier subproblem that arises in an interior-point framework.  This subproblem is equality-constrained, but unlike for an approach for the equality-constrained setting with deterministic noise such as that proposed in \cite{oztoprak2023constrained}, significant obstacles arise in our setting due to the fact that our method needs to enforce nonnegativity of slack variables that are introduced for our problem formulation.  Along with our concluding remarks, we suggest a heuristic strategy for decreasing the barrier parameter so that, in practice, a sequence of barrier subproblems may be solved, as is typical for an interior-point method.  However, our analysis does not cover convergence guarantees for solving the original constrained optimization problem as the barrier parameter vanishes.  On this note, one should recognize that such an analysis would be questionable in any case.  After all, even for a method for the equality-constrained setting such as that in \cite{oztoprak2023constrained}, the presence of deterministic noise means that one cannot guarantee convergence to a solution of the original problem of interest.  Such is the case for our method for solving each barrier subproblem as well, meaning that it may be impossible to guarantee that a point is reached that is sufficiently close to optimality (or even stationarity) of the noiseless barrier subproblem to warrant a decrease in the barrier parameter.  That said, in practice, a decrease of the barrier parameter may yield good behavior, so we suggest a heuristic in our concluding remarks.

\subsection{Notation}

Our optimization problem of interest is stated in terms of a decision variable $x \in \RR^n$; see \eqref{eq.main_problem} in Section~\ref{sec.problem}.  For any $k \in \NN$, the value of $x$ in the~$k$th iteration of our algorithm is denoted as~$x_k$.  Such subscripts are used for other variables and adaptive parameters that are employed.  For any $i \in [n] := \{1,\dots,n\}$, the $i$th component of a vector $x \in \RR^n$ is denoted $x^{(i)}$.  Generally, vectors are expressed using lowercase letters.  With respect to any such a vector, the corresponding uppercase letter denotes a diagonal matrix with the components of the vector on the diagonal, e.g., with respect to $s \in \RR^q$ we denote $S := \text{diag}(s) \in \RR^{q \times q}$.

Given a matrix $M \in \RR^{m\times n}$, its range space is denoted as $\mathrm{range}(M)$, its null space is denoted as $\mathrm{null}(M)$, its minimum singular value is denoted as $\sigma_{\min}(M)$, its maximum singular value is denoted as $\sigma_{\max}(M)$, and its condition number is denoted as $\kappa(M) := \sigma_{\max}(M)/\sigma_{\min}(M)$.  Given a square matrix $M \in \RR^{n \times n}$, a square root of $M$ is any matrix $M^{1/2}$ such that $M^{1/2}M^{1/2} = M$. Given a pair of vectors $(u,v) \in \RR^n \times \RR^n$, their inner product is denoted as $\langle u,v \rangle = u^Tv$.  We use the operator notation $\|\cdot\| := \|\cdot\|_2 = \sqrt{\langle \cdot, \cdot \rangle}$ to denote the 2-norm of a vector and similarly use $\|\cdot\|$ to denote the induced 2-norm for a matrix.  For a symmetric and positive definite matrix $M \in \RR^{n \times n}$ and a pair of vectors $(u,v) \in \RR^n \times \RR^n$, we write $\langle u, v \rangle_M = u^T M v$ and denote $\|\cdot\|_M = \sqrt{\langle \cdot, \cdot \rangle_M}$.

Given two sequences $\{u_k\}$ and $\{v_k\}$ with $u_k \in \RR$ and $v_k \in [0,\infty)$ for all $k \in \NN$, the ``big-O'' notation, namely, $u_k = O(v_k)$, indicates that there exists a real number $c \in (0,\infty)$---defined independently from $k$---such that for all sufficiently large $k \in \NN$ one has $|u_k| \leq c v_k$.

\subsection{Outline}

Our problem of interest and related optimization problem formulations for which our proposed algorithm is designed are stated in Section~\ref{sec.problem}.  Our proposed algorithm for solving a so-called barrier subproblem that arises in our discussions is presented in Section~\ref{sec.algorithm}.  Our analysis of the convergence properties of our proposed algorithm from Section~\ref{sec.algorithm} is provided in Section~\ref{sec.analysis}. In Section~\ref{sec.numerical_exp}, we present the results of numerical experiments that show the effectiveness of the proposed algorithm. A conclusion is provided in Section~\ref{sec.conclusion_4}.

\section{Problem Description}\label{sec.problem}

Our main algorithm, stated as Algorithm~\ref{alg.exact_alg} on page~\pageref{alg.exact_alg}, generates an iterate sequence $\{x_k\}$ with $x_k \in \RR^n$ for all generated $k \in \NN$.  Its aim is to solve (if only approximately) the continuous optimization problem
\begin{equation}\label{eq.main_problem}
  \min_{x \in \RR^n}\ \ftrue_0(x)\ \ \st \ \ \ctrue_I(x) \leq 0,
\end{equation}
where $\ftrue_0: \RR^n \to \RR$ and $\ctrue_I: \RR^n \to \RR^q$.  (Throughout the paper, a ``bar'' above a quantity is used to indicate a noiseless quantity, as opposed to noisy quantities that have no ``bar'' above them.)  The algorithm can be extended to settings in which equality constraints are also present, but for simplicity and since the main algorithmic components and corresponding analysis are nearly identical for that situation and this one, we restrict attention to the setting with only inequality constraints.  The problem functions $\ftrue_0$ and $\ctrue_I$ are assumed to satisfy the following loose assumption pertaining to the generated iterates.  The assumption includes that the functions are continuously differentiable so that the objective gradient function $\gtrue_0 := \nabla \ftrue_0 : \RR^n \to \RR^n$ and constraint Jacobian function $\Jtrue_I := \nabla \ctrue_I^T : \RR^n \to \RR^{q \times n}$ are well defined.  As previously mentioned, the main challenge of our setting is that the algorithm only has access to noisy values of the objective function~$\ftrue_0$, its gradient function $\gtrue_0$, the constraint function $\ctrue_I$, and its Jacobian function $\Jtrue_I$.  Our assumption about the noisy evaluations of these functions is included in the following assumption.

\begin{assumption}\label{assumption.err0}
  \textit{
  The iterate sequence $\{x_k\}$ generated by Algorithm~\ref{alg.exact_alg} is contained in an open convex set $\XXX \subseteq \RR^n$ over which $\ftrue_0$ and $\ctrue_I$ are continuously differentiable, $\ftrue_0$ is bounded below, $\ctrue_I$ is bounded in norm, and both functions $\gtrue_0 := \nabla \ftrue_0$ and $\Jtrue_I := \nabla \ctrue_I^T$ are Lipschitz continuous with values that are bounded uniformly in norm.  Moreover, for any $x \in \XXX$, a call for $\ftrue_0(x)$, $\gtrue_0(x)$, $\ctrue_I(x)$, or $\Jtrue_I(x)$ results in the approximate value $f_0(x)$, $g_0(x)$, $c_I(x)$, or $J_I(x)$, respectively, where for some known constants $(\eps_f,\eps_g,\eps_c,\eps_J) \in (0,\infty)^4$ one has
  \begin{equation}\label{eq.errors}
    \begin{aligned}
      |\ftrue_0(x) - f_0(x) | \leq \eps_f,\ \ \|\gtrue_0(x) - g_0(x) \| &\leq \eps_g, \\ \|\ctrue_I(x) - c_I(x) \| \leq \eps_c,\ \ \text{and}\ \ \|\Jtrue_I(x) - J_I(x)\| &\leq \eps_J.
    \end{aligned}
  \end{equation}
  }
\end{assumption}
Under Assumption~\ref{assumption.err0}, there exist constants $f_{\inf} \in \RR$ and $(c_{I,\sup},g_{0,\sup},J_{I,\sup}) \in (0,\infty)^3$ such that for all $x \in \XXX$ one has
\begin{equation*}
  \ftrue_0(x) \geq f_{\inf},\ \|\ctrue_I(x)\| \leq c_{I,\sup},\ \|\gtrue_0(x)\| \leq g_{0,\sup},\ \text{and}\ \|\Jtrue_I(x)\| \leq J_{I,\sup},
\end{equation*}
and there exist $(L_g,L_J) \in (0,\infty)^2$ such that for all $(x,\bar x) \in \XXX \times \XXX$ one has
\begin{equation*}
  \|\gtrue_0(x) - \gtrue_0(\bar x)\| \leq L_g \|x - \bar x\|\ \text{and}\ \|\Jtrue_I(x) - \Jtrue_I(\bar x)\| \leq L_J \|x - \bar x\|.
\end{equation*}

Under Assumption~\ref{assumption.err0}---specifically the assumption that the objective and constraint functions are continuously differentiable---and a constraint qualification, e.g., the Mangasarian-Fromovitz constraint qualification (MFCQ)~\cite{MangFrom67}, it follows that at a local minimizer $x \in \RR^n$ of \eqref{eq.main_problem} there exists $\ytrue_I \in \RR^q$ with
\begin{equation*}
  \gtrue_0(x) + \Jtrue_I(x)^T\ytrue_I = 0,\ \ \ctrue_I(x) \leq 0,\ \ \ytrue_I \geq 0,\ \ \text{and}\ \ \ctrue_I(x)^T\ytrue_I = 0.
\end{equation*}
However, more generally, it is possible that the constraints of problem~\eqref{eq.main_problem} are infeasible (at least locally near a point), or that at a local minimizer of~\eqref{eq.main_problem} the constraints are degenerate in the sense that a constraint qualification does not hold.  With respect to these situations, the aim of the algorithm is at least to solve the infeasibility-minimization problem
\begin{equation}\label{eq.feas_problem}
  \min_{x \in \RR^n} \tfrac{1}{2} \| \max\{\ctrue_I(x),0\} \|^2,
\end{equation}
where the max is defined component-wise.  The objective of this problem is continuously differentiable and first-order conditions for optimality for it are
\begin{equation}\label{eq.feas_problem_opt}
  \Jtrue_I(x)^T\max\{\ctrue_I(x),0\} = 0.
\end{equation}
We refer to any $x$ yielding \eqref{eq.feas_problem_opt} with $\ctrue_I(x) \not\leq 0$ as an infeasible stationary point.

Following the design of the algorithm in \cite{CurtScheWaec10}, our algorithm aims to solve (perhaps only approximately) the optimization problem~\eqref{eq.main_problem} or at least the infeasibility-minimization problem~\eqref{eq.feas_problem} by working with related formulations involving slack variables.  In particular, observe that~\eqref{eq.main_problem} is equivalent to
\begin{equation}\label{eq.main_problem_slack}
  \min_{x \in \RR^n}\ \ftrue_0(x)\ \ \st \ \ \ctrue_I(x) + s = 0\ \ \text{and}\ \ s \geq 0,
\end{equation}
and that if $s \geq 0$ and $\ctrue_I(x) + s \geq 0$, then~\eqref{eq.feas_problem_opt} is equivalent to
\begin{equation}\label{eq.feas_problem_opt_slack}
  \Jtrue_I(x)^T(\ctrue_I(x) + s) = 0\ \ \text{and}\ \ S (\ctrue_I(x) + s) = 0.
\end{equation}
The conditions $s \geq 0$ and $\ctrue_I(x) + s \geq 0$ motivate a \textit{slack reset} that is incorporated into the algorithm to ensure that these inequalities always hold at least with respect to the noisy constraint function evaluation; see Section~\ref{sec.algorithm}.

Our algorithmic strategy falls within the interior-point methodology, which means that it does not aim to solve~\eqref{eq.main_problem_slack} or \eqref{eq.feas_problem_opt_slack} directly; rather, it employs a log-barrier function of the slack variables to replace the inequality constraints.  The general idea of an interior-point method would be to solve to some accuracy a subproblem involving the barrier function on the slack variables for a given value of a barrier parameter $\mu \in (0,\infty)$, then decrease this parameter and solve the next subproblem in an iterative manner such that the overall algorithm converges to a solution of \eqref{eq.main_problem_slack} or at least \eqref{eq.feas_problem_opt_slack}.  The goal of our main algorithm and our corresponding analysis of it is to solve such a barrier subproblem approximately using only noisy function and derivative evaluations.  The barrier subproblem that we employ (with $\log(\cdot)$ denoting the natural logarithm) is
\begin{align}\label{eq.barrier}
  \min_{(x,s) \in \RR^n \times \RR^q}\ \ftrue_0(x) - \mu \sum_{i=1}^q \log(s^{(i)})\ \ \st \ \ \ctrue_I(x) + s = 0\ \ \text{and $s > 0$}.
\end{align}
Observe that, at any point $(x,s)$ with $s > 0$, the linear independence constraint qualification (LICQ) with respect to \eqref{eq.barrier} holds.  Hence, at any local minimizer $(x,s)$ of~\eqref{eq.barrier}, there exists a Lagrange multiplier $\ytrue_{I,\mu} \in \RR^q$ such that
\begin{equation}\label{eq.first_order_opt}
  \gtrue_0(x) + \Jtrue_I(x)^T \ytrue_{I,\mu} = 0,\ \ -\mu S^{-1} + \ytrue_{I,\mu} = 0,\ \ \ctrue_I(x) + s = 0,\ \ \text{and}\ \ s > 0.
\end{equation}
We remark that in a setting in which equality constraints are also present, the LICQ would not be guaranteed to hold, but a similar set of conditions would be guaranteed under a suitable constraint qualification, such as the MFCQ.

For notational convenience in our presentation and analysis of our proposed algorithm, let us introduce the combined primal iterate vector $z \in \RR^{n + q}$ as $z := [x^T\ s^T]^T$.  For a given barrier parameter $\mu$, the objective and constraint functions of~\eqref{eq.barrier} can be expressed as $\ftrue : \RR^{n+q} \to \RR$ and $\ctrue : \RR^{n+q} \to \RR^q$ where
\begin{equation*}
  \ftrue(z) = \ftrue_0(x) - \mu \sum_{i=1}^q \log(s^{(i)})\ \text{and}\ \ctrue(z) = \ctrue_I(x) + s\ \text{for all}\ z \in \RR^{n+q}.
\end{equation*}
Following \cite{CurtScheWaec10}, our algorithm employs \textit{scaled} first-order derivatives of these functions.  In particular, we define the scaled gradient function $\gtrue : \RR^{n+q} \to \RR$ and scaled Jacobian function $\Jtrue : \RR^{n+q} \to \RR^{q \times (n+q)}$ such that, for all $z \in \RR^{n+q}$,
\begin{equation*}
  \gtrue(z) = \begin{bmatrix} \gtrue_0(x) \\ -\mu e \end{bmatrix},\ \ \text{and}\ \ \Jtrue(z) = \begin{bmatrix} \Jtrue_I(x) & S \end{bmatrix},
\end{equation*}
where $e \in \RR^q$ is a vector with all entries equal to one.  (The imposed scaling is that the second block of the gradient of the objective and the second block of the constraint Jacobian are both multiplied by the diagonal matrix $S$.)  Observe that in terms of this scaled Jacobian, the conditions in \eqref{eq.feas_problem_opt_slack} are equivalent to
\begin{equation}\label{eq.feas_problem_opt_z}
  \Jtrue(z)^T\ctrue(z) = 0.
\end{equation}
Letting $f$, $g$, $c$, and $J$ denote noisy approximations of $\ftrue$, $\gtrue$, $\ctrue$, and $\Jtrue$, respectively, Assumption~\ref{assumption.err0} implies that for all $z = [x^T\ s^T]^T$ with $x \in \XXX$ one has
\begin{equation}\label{eq.errors2}
  \begin{aligned}
    |\ftrue(z) - f(z)| \leq \eps_f,\ \ \|\gtrue(z) - g(z) \| &\leq \eps_g, \\ \|\ctrue(z) - c(z)\| \leq \epsilon_c,\ \ \text{and}\ \ \|\Jtrue(z) - J(z)\| &\leq \epsilon_J.
  \end{aligned}
\end{equation}

\section{Algorithm Description, Fixed Barrier Parameter}\label{sec.algorithm}

Let us now present our main algorithm, which aims to produce at least an approximate solution of \eqref{eq.main_problem} by solving the barrier subproblem~\eqref{eq.barrier} for a fixed barrier parameter $\mu \in (0,\infty)$, or at least to yield convergence to satisfying the stationarity conditions \eqref{eq.feas_problem_opt_z}.  To describe the steps of our main algorithm, suppose that it has reached iteration $k \in \NN$, in which case the $k$th iteration proceeds as described in this section with the noisy function and derivative values $f_k := f(z_k) := f_0(x_k) - \mu \sum_{i=1}^q \log (s_k^{(i)})$, $c_k := c(z_k) := \mbox{$c_I(x_k) + s_k$}$, $g_k^T := g(z_k)^T := [g_0(x_k)^T\ -\mu e^T]^T$, and $J_k := J(z_k) := [J_I(x_k)\ S_k]$.  The corresponding noiseless values $\ftrue_k$, $\ctrue_k$, $\gtrue_k$, and $\Jtrue_k$ are defined similarly.  Our proposed algorithm is stated formally as Algorithm~\ref{alg.exact_alg} on page~\pageref{alg.exact_alg}. We remark that, as previously mentioned, the initial conditions and structure of the algorithm will ensure that $s_k \geq 0$ and $c(z_k) = c_I(x_k) + s_k \geq 0$ for all $k \in \NN$.  Let us also define, for all $k \in \NN$, the error values $\vareps_{f,k} \in \RR$, $\vareps_{g,k} \in \RR^n$, $\vareps_{c,k} \in \RR^{q}$, and $\vareps_{J,k} \in \RR^{q \times (n + q)}$ according to the equations
\begin{equation}\label{eq.errors3}
  f_k = \ftrue_k + \vareps_{f,k},\ g_k = \gtrue_k + \vareps_{g,k},\ c_k = \ctrue_k + \vareps_{c,k},\ \text{and}\ J_k = \Jtrue_k + \vareps_{J,k},
\end{equation}
where by \eqref{eq.errors2} one has $|\vareps_{f,k}| \leq \eps_f$, $\|\vareps_{g,k}\| \leq \eps_g$, $\|\vareps_{c,k}\| \leq \eps_c$, and $\|\vareps_{J,k}\| \leq \eps_J$.

The following occurs for each iteration index $k \in \NN$.  First, if $J_k^Tc_k = 0$ and $c_I(x_k) \not\leq 0$, then the algorithm has reached an infeasible stationary point (based on noisy evaluations) and it terminates; recall that the algorithm ensures $s_k \geq 0$ and $c_k = c_I(x_k) + s_k \geq 0$, so $J_k^Tc_k = 0$ (i.e., \eqref{eq.feas_problem_opt_z}) corresponds to \eqref{eq.feas_problem_opt_slack}.  Otherwise, the algorithm computes a normal step toward the satisfaction of a linear approximation of the constraint functions within a trust region.  Specifically, for a prescribed constant $\omega \in (0,\infty)$, the normal step $v_k$ is computed by solving the trust-region subproblem
\begin{equation}\label{eq.normal_tr_step}
  \min_{v \in \mathrm{range}(J_k^T)}\ \ \tfrac{1}{2} \| c_k + J_k v\|^2\ \ \st\ \ \|v\| \leq \omega \|J_k^T c_k\|.
\end{equation}
For later reference, let us define $\vtrue_k$ as the solution of \eqref{eq.normal_tr_step} (with respect to~$x_k$) if the true values $(\ctrue_k, \Jtrue_k)$ are used in place of the noisy values $(c_k, J_k)$.  Second, the algorithm computes a tangential step to minimize a local model of the objective function subject to maintaining the progress toward linearized feasibility that was attained by the normal step.  This computation employs symmetric
\begin{equation}\label{eq.W}
  W_k \gets \begin{bmatrix} H_k & 0 \\ 0 & \Sigma_k \end{bmatrix},
\end{equation}
where $H_k \in \RR^{n \times n}$ and $\Sigma_k \in \RR^{q \times q}$.  (See below for further discussion on the requirements of $W_k$ for each $k \in \NN$.)  The tangential step is defined as $u_k := d_k - v_k$, where the full step $d_k$ (i.e., normal step plus tangential step) along with a new Lagrange multiplier estimate $y_{k+1}$ is computed by solving
\begin{equation}\label{eq.perturbed_newton_sys}
  \begin{bmatrix} W_k & J_k^T \\ J_k & 0 \end{bmatrix} \begin{bmatrix} d_k \\ y_{k+1} \end{bmatrix} = \begin{bmatrix} -g_k \\ J_k v_k \end{bmatrix}.
\end{equation}
For later reference, let us define $(\utrue_k, \dtrue_k, \ytrue_{k+1})$ as the values that would have been computed (with respect to the same symmetric matrix $W_k$) if the true values $(\Jtrue_k, \gtrue_k, \vtrue_k)$ are used in place of the noisy values $(J_k, g_k, v_k)$.

A few comments on the choice of $\{W_k\}$ are in order.  First, we remark that the traditional choice in an interior-point method in the noiseless setting is for $H_k$ to be chosen as the Hessian of a Lagrangian at the current iterate, or an approximation to it.  However, for our noisy setting wherein we do not guarantee a fast rate of local convergence, we require $H_k$ to be a symmetric matrix, but do not require it to be a matrix of second-order derivatives.  (It can be constructed using noisy second-order derivative estimates, if they are available, but this is not necessary for our analysis.)  As for $\Sigma_k$, if it is chosen as $\Sigma_k = \mu I$, then this is said to be the primal approach, whereas if it is chosen as $\Sigma_k = S_k\hat Y_k$ for some multiplier vector $\hat y_k \in \RR^q$, then this is said to be the primal-dual approach.  That said, we do not restrict our analysis to any such choice of $\Sigma_k$, as long as it is diagonal and positive definite.  In addition to each element of the sequence $\{W_k\}$ defined by \eqref{eq.W} being symmetric, we also require that (a)~the linear system \eqref{eq.perturbed_newton_sys} has a unique solution for all $k \in \NN$, (b)~$W_k$ is sufficiently positive definite in the null space of $J_k$ for all $k \in \NN$, and (c)~$\{\|W_k\|\}$ is bounded.  We formalize these assumptions in Assumption~\ref{assumption.bounded} on page~\pageref{assumption.bounded}, where in fact (for simplicity) we go a bit further and assume that $W_k$ satisfies such properties also with respect to true function and derivative values.  To ensure these conditions with respect to noisy quantities in practice, during the $k$th iteration for any $k \in \NN$, the algorithm can start with initial symmetric $H_k$ and $\Sigma_k$ such that $W_k$ in \eqref{eq.W} satisfies a prescribed bound on its norm.  Then, the algorithm can determine if~\eqref{eq.perturbed_newton_sys} has a unique solution and $W_k$ is positive definite over $\mathrm{null}(J_k)$, say, by checking the inertia of the matrix in \eqref{eq.perturbed_newton_sys}; see, e.g., \cite{nocedal1999numerical}.  If not, then $H_k$ and/or $\Sigma_k$ can be replaced with a convex combination of its current value and a prescribed positive definite matrix, or it can be modified by adding a multiple of the identity matrix; see, e.g., \cite{wachter2006implementation}.  This process can be repeated iteratively until the desired conditions are satisfied.  Observe that modifying $W_k$ as discussed above guarantees that the system~\eqref{eq.perturbed_newton_sys} has a unique solution because $J_k$ has full row rank. 

Upon the computation of the normal step $v_k$, tangential step $u_k$, combined search direction $d_k$, and updated multiplier estimate $y_{k+1}$, the algorithm checks another termination condition.  If $J_k^Tc_k = 0$ (implying $c_I(x_k) \leq 0$ since the algorithm did not terminate previously) and $g_k + J_k^Ty_{k+1} = 0$, then the algorithm has reached a point at which no further progress can be made (with respect to noisy values).  This situation implies that the algorithm has reached an approximate stationary point for the barrier subproblem \eqref{eq.barrier}.  We discuss the meaning of this situation further in our analysis in Section~\ref{sec.analysis}.

If the algorithm proceeds in the $k$th iteration after checking this second termination condition and has not terminated, then it needs to determine the amount to move along the search direction, i.e., it needs to determine a positive step size.  This is done using a merit function---which balances the objective function and a constraint violation measure---whose aim is to measure progress of the algorithm toward a minimizer.  Before determining the step size, the algorithm first updates the merit parameter that weighs the contributions of the merit function terms.  For the merit function, our proposed algorithm employs $\phi : \RR^{n+q} \times (0,\infty) \to \RR$ defined by
\begin{equation*}
  \phi(z,\tau) = \tau f(z) + \| c(z) \|,
\end{equation*}
where $\tau \in (0,\infty)$ is the merit parameter that is updated dynamically.  For later reference, let us also introduce the corresponding true merit function $\phitrue : \RR^{n+q} \times (0,\infty) \to \RR$ with $\phitrue(z,\tau) = \tau \ftrue(z) + \|\ctrue(z)\|$.  To determine whether and by how much the merit parameter should be decreased, the algorithm uses a (noisy) model of $\phi$ at $x_k$, namely, $m_k : \RR^{n+q} \times (0,\infty) \to \RR$ with
\begin{equation*}
  m_k(d,\tau) = \tau (f_k + g_k^T d) + \|c_k + J_kd\|.
\end{equation*}
The change in this model corresponding to $(d_k,\tau_k) \in \RR^{n+q} \times (0,\infty)$ is measured by the model reduction function $\Delta m_k : \RR^{n+q} \times (0,\infty) \to \RR$ defined by
\begin{align}
   \Delta m_k(d_k, \tau_k) &= m_k(0,\tau_k) - m_k(d_k,\tau_k) \nonumber \\ 
   &= -\tau_k g_k^T d_k + \|c_k\| - \|c_k + J_k d_k\| \nonumber \\
   &= -\tau_k g_k^T d_k + \|c_k\| - \|c_k + J_k v_k\|, \label{eq.delta_m_tilde}
\end{align}
where the last equation follows by \eqref{eq.perturbed_newton_sys}.  For reference in our analysis, let us also note that the reduction in a true (i.e., noiseless) model of the true merit function~$\phitrue$ at $x_k$ can be defined as $\Delta \mtrue_k : \RR^{n+q} \times (0,\infty) \to \RR$ defined similarly.  This true model reduction is notable since having $\{(d_k,\tau_k)\}$ with $\tau_k \geq \tau$ for some $\tau \in (0,\infty)$ and $\{\Delta \mtrue_k(d_k,\tau_k)\} \to 0$ is central to analyses of an algorithm of our type in the noiseless setting.  We discuss this further in Section~\ref{sec.analysis}.

The goal of our update strategy of the merit parameter is to ensure that the chosen value $\tau_k \in (0,\tau_{k-1}]$ (where $\tau_{k-1} \in (0,\infty)$ is the value from the prior iteration) yields, for some prescribed $\sigma \in (0,1)$, the condition
\begin{equation}\label{eq.model_reduction_cond}
  \Delta m_k(d_k,\tau_k) \geq \tfrac12 \tau_k  u_k^T W_k u_k + \sigma (\|c_k\| - \|c_k + J_k v_k\|).
\end{equation}
(This is the same kind of approach as has been employed in the noiseless setting; see, e.g., \cite{CurtScheWaec10}.)  From \eqref{eq.normal_tr_step} and positive-definiteness of $W_k$ with respect to $\mathrm{null}(J_k)$, the right-hand side of this expression is always nonnegative, and in fact in our analysis we show that this quantity is strictly positive when the algorithm does not terminate in iteration $k$.  To determine such a value for the merit parameter $\tau_k$, the algorithm uses the trial merit parameter value
\begin{equation}\label{eq.pi_trial}
  \tau_k^{\text{trial}} \gets \begin{cases} \infty & \text{if $g_k^T d_k + \tfrac12 u_k^T W_k u_k \leq 0$} \\ \frac{(1 - \sigma)(\|c_k \| - \|c_k + J_k v_k \|)}{g_k^T d_k +  \tfrac12 u_k^T W_k u_k} & \text{otherwise}, \end{cases} 
\end{equation}
then, with prescribed $\delta_{\tau} \in (0,1)$, sets the merit parameter as
\begin{equation}\label{eq.tau}
  \tau_k = \begin{cases} \tau_{k-1} & \text{if $\tau_{k-1} \leq \tau_k^{\text{trial}}$} \\ \min\{(1-\delta_{\tau}) \tau_{k-1}, \tau_k^{\text{trial}}\} & \text{otherwise}. \end{cases}
\end{equation}
(In Lemma~\ref{lemma.well_defined}, we show that this ensures that \eqref{eq.model_reduction_cond} holds.)  For future reference, let us also define the sequence $\{\tautrue_k\}$ as the one that would have been generated in this manner if the algorithm generated the same iterate sequence $\{x_k\}$, but employed noiseless values for updating the merit parameter value for all $k \in \NN$.

Now that the value of the merit parameter has been determined, the algorithm proceeds to compute a step size to take along the primal search direction.  At this stage, it is necessary to characterize the primal search direction in terms of two components---corresponding to the variables $x$ and $s$, respectively---and the search direction corresponding to the unscaled derivatives as well.  The scaled and unscaled search directions, respectively, are
\begin{equation*}
  d_k =: \begin{bmatrix} d_k^x \\ d_k^s \end{bmatrix}\ \ \text{and}\ \ \dhat_k \gets \begin{bmatrix} d_k^x \\ S_k d_k^s \end{bmatrix},\ \ \text{where}\ \ (d_k^x,d_k^s) \in \RR^n \times \RR^q.
\end{equation*}
One rule for the step size is that it must ensure that the slack variables remain sufficiently positive.  This is done through a fraction-to-the-boundary rule.  The unscaled direction in the slack variables is the latter component of the direction~$\dhat_k$, namely, $S_k d_k^s$.  Therefore, the fraction-to-the-boundary rule involves computing the largest value of $\alpha_k^{\max}$ in $(0,1]$ such that
\begin{equation}\label{eq.foc}
  s_k + \alpha_k^{\max} S_k d_k^s \geq (1 - \eta_s) s_k,
\end{equation}
where $\eta_s \in (0, 1)$ is a prescribed parameter.  Once this maximum step size has been determined, the algorithm backtracks, if necessary, until a relaxed Armijo condition is satisfied 
\cite{berahas2019derivative,oztoprak2023constrained}.  
Defining $\vareps_k \in (0,\infty)$ for all $k \in \NN$ by $\vareps_k := \tau_k \eps_f + \eps_c$, and with $\eta_\phi \in (0,1)$ and $\zeta \in (0,\infty)$, the condition is
\begin{equation}\label{eq.armijo}
  \phi(z_k + \alpha_k  \dhat_k, \tau_k) \leq \phi(z_k, \tau_k) - \eta_\phi \alpha_k \Delta m_k (d_k, \tau_k) + (2 + \zeta) \vareps_k.
\end{equation}
(In practice, we expect one to set $\zeta$ relatively small, at least less than 1.  For example, in our numerical experiments, $\zeta \gets 0.1$.) Observe that \eqref{eq.armijo} is well defined since it is always satisfied for sufficiently small $\alpha_k \in (0,1]$ in the presence of positive $\eps_f$ and/or $\eps_c$.  Given $\alpha_k \in (0,1]$ obtained by backtracking (starting from $\alpha_{\max}$) that yields \eqref{eq.armijo}, the algorithm sets $x_{k+1} \gets x_k + \alpha_k d_k^x$ and $s_{k+1} \gets \max\{s_k + \alpha_k S_k d_k^s, -c_I(x_{k+1})\} \geq 0$, for which one finds that 
\begin{equation}\label{eq.slack_reset}
  \begin{aligned}
    c(z_{k+1}) &= c_I(x_{k+1}) + s_{k+1}\\
    &= c_I(x_{k+1}) + \max \left\{s_k + \alpha_k S_k d_k^s, - c_I(x_{k+1})\right\} \geq 0. 
  \end{aligned}
\end{equation}
It is important to observe that, after employing this update for the slack variables, which as in the literature we refer to as a slack reset, one finds that $\phi(z_{k+1},\tau_k) \leq \phi(z_k + \alpha_k  \dhat_k, \tau_k)$, meaning that \eqref{eq.armijo} holds with the left-hand side replaced by $\phi(z_{k+1},\tau_k)$.  This is clarified in the following remark.

\begin{remark}\label{remark.slackreset}
  Line~\ref{step.slack_reset} of Algorithm~\ref{alg.exact_alg} guarantees $\phi(z_{k+1},\tau_k) \leq \phi(z_k + \alpha_k  \dhat_k, \tau_k)$.  To see this, consider arbitrary $i \in \{1,\dots,q\}$.  If $[c_I(x_{k+1})]^{(i)} \geq -[s_k + \alpha_k S_k d_k^s]^{(i)}$, then it yields $s_{k+1}^{(i)} \gets [s_k + \alpha_k S_k d_k^s]^{(i)}$; otherwise, $s_{k+1}^{(i)} \gets [-c_I(x_{k+1})]^{(i)}$, in which case $[c(z_{k+1})]^{(i)} = 0$.  In any case, one finds $s_{k+1} \geq s_k + \alpha_k S_k d_k^s$, $f(z_{k+1}) \leq f\left(\begin{bmatrix} x_{k+1} \\ s_k + \alpha_k S_k d_k^s \end{bmatrix}\right)$, and $\|c(z_{k+1})\| \leq \left\|c\left(\begin{bmatrix} x_{k+1} \\ s_k + \alpha_k S_k d_k^s \end{bmatrix}\right) \right\|$, so $\phi(z_{k+1},\tau_k) \leq \phi(z_k + \alpha_k  \dhat_k, \tau_k)$, as claimed.
\end{remark}

\begin{algorithm}[ht]
  \caption{: Algorithm for Solving \eqref{eq.barrier} with Noisy Function Evaluations}
  \label{alg.exact_alg}
  \begin{algorithmic}[1]
  \Require $(x_0,s_0,y_0) \in \RR^n \times \RR^q \times \RR^q$ with $s_0 \geq 0$ and $c(z_0) = c_I(x_0) + s_0 \geq 0$; $(\eps_f,\eps_c) \in (0,\infty)^2$ from Assumption~\ref{assumption.err0}; $\tau_{-1} \in (0,\infty)$; $\omega \in (0,\infty)$; $\sigma \in (0,1)$; $\delta_{\tau} \in (0,1)$; $\eta_s \in (0,1)$; $\eta_\phi \in (0,1)$; and $\zeta \in (0,\infty)$
  \For{$k = 0, 1, \dots$}
    \State \textbf{if} $J_k^T c_k = 0$ and $c_I(x_k) \not\leq 0$ \textbf{then} \textbf{return} $(x_k,y_{k+1})$ \label{st.tm1}
    \State compute $v_k$ as the solution of \eqref{eq.normal_tr_step}
    \State choose $W_k$ satisfying (upcoming) Assumption~\ref{assumption.bounded}
    \State compute $(d_k,y_{k+1})$, where $d_k \equiv v_k + u_k \equiv \begin{bmatrix} d_k^x \\ d_k^s \end{bmatrix}$, by solving \eqref{eq.perturbed_newton_sys}
    \State \textbf{if} $J_k^T c_k = 0$ (so $c_I(x_k) \leq 0$) and $g_k + J_k^T y_{k+1} = 0$ \textbf{then} \textbf{return} $(x_k,y_{k+1})$ \label{st.tm2}
    \State set $\tau_k$ by \eqref{eq.tau} \label{step.tau}
    \State set $\alpha_k^{\max}$ as the largest value in $(0,1]$ such that \eqref{eq.foc} holds \label{step.alpha_max}
    \State set $\alpha_k \gets (\frac12)^j \alpha_k^{\max}$, where $j = \min \left\{j \in \NN: \text{\eqref{eq.armijo} holds with $\dhat_k \gets \begin{bmatrix} d_k^x \\ S_k d_k^s \end{bmatrix}$}\right\}$ \label{step.alpha}
    \State set $x_{k+1} \gets x_k + \alpha_k d_k^x$
    \State set $s_{k+1} \gets \max\{s_k + \alpha_k S_k d_k^s, -c_I(x_{k+1})\}$\label{step.slack_reset}
  \EndFor
  \end{algorithmic}
\end{algorithm}

\section{Analysis of Algorithm~\ref{alg.exact_alg}}\label{sec.analysis}

In this section, we analyze the behavior of Algorithm~\ref{alg.exact_alg} for solving \eqref{eq.barrier}.  Before starting our analysis, we discuss its main goals by remarking on the convergence guarantees that one aims to show in a noiseless setting for an algorithm such as ours.  Then, we discuss our main assumptions and show that the algorithm is well posed before proceeding to our main convergence analysis.

\subsection{Discussion of Aims of Convergence Analysis}\label{sec.discussion}

Given the fact that Algorithm~\ref{alg.exact_alg} aims to solve \eqref{eq.barrier}, but only has access to noisy function and derivative values, it is important to establish at the outset what one can hope to prove about the behavior of the algorithm.  First, we note that if for some $k \in \NN$ the algorithm terminates in line~\ref{st.tm1}, then at least with respect to noisy quantities it has reached an infeasible stationary point.  The point is not necessarily an infeasible stationary point with respect to noiseless quantities, but it is still reasonable for the algorithm to terminate since, by~\eqref{eq.errors3} and upcoming Lemma~\ref{lemma.violet_orange_olive} in our analysis, one respectively has
\begin{equation*}
  \|\ctrue_I(x_k) - c_I(x_k)\| \leq \epsilon_c\ \ \text{and}\ \ \|\Jtrue_k^T \ctrue_k - J_k^T c_k\| = O(\eps_c + \eps_J).
\end{equation*}
Similarly, if for some $k \in \NN$ the algorithm terminates in line~\ref{st.tm2}, then at least with respect to the noisy quantities it has reached a point that appears to be stationary for~\eqref{eq.barrier}.  The point is not necessarily stationary for~\eqref{eq.barrier} with respect to noiseless quantities, but it is still reasonable for the algorithm to terminate since, under the assumptions of our upcoming analysis, one can show that
\begin{equation*}
  \|\gtrue_k + \Jtrue_k^T \ytrue_{k+1} - (g_k + J_k^T y_{k+1})\| = O(\eps_c + \eps_J + \eps_g).
\end{equation*}
For the sake of brevity and since finite termination is unlikely in practice, we omit the details of this bound, but claim that it would follow under the assumptions and results in Section~\ref{sec.convergence} up through and including Corollary~\ref{lemma.ub_epskd}.

Assuming that Algorithm~\ref{alg.exact_alg} does not terminate finitely, the aims of our analysis are to prove results that are modeled on the noiseless setting.  In particular, we note the main convergence result in \cite{curtis2010matrix} and the main convergence result for a fixed barrier parameter in \cite{CurtScheWaec10}.  In both of these cases, the main aim is to show that, when the sequence of constraint Jacobians has singular values that are bounded away from zero, one has that the merit parameter remains bounded below by a positive real number and in the limit one finds
\begin{equation}\label{eq.aim}
  \{ \Delta \mtrue_k (\dtrue_k, \tautrue_k) \} \to 0,\ \text{which in turn yields}\ 
  \left\{ \left\| \begin{bmatrix} \gtrue_k + \Jtrue_k^T \ytrue_{k+1} \\ \ctrue_k \end{bmatrix} \right\| \right\} \to 0.
\end{equation}
(See, e.g., Lemma~3.18 and its proof in \cite{curtis2010matrix}.)  Otherwise (when the constraint Jacobians tend toward rank deficiency), one at least has that $\{\|\Jtrue_k^T\ctrue_k\|\} \to 0$.  In our noisy setting, the case of potential rank deficiency of the constraint Jacobians presents a practical hurdle, especially when the noisy Jacobians tend toward rank deficiency while the true Jacobians do not.  Our analysis thus focuses on the stronger results that can be proved when the constraint Jacobians do not tend toward rank deficiency and do not differ too much from their true values (see Assumptions~\ref{assumption.sigma_J} and \ref{ass.noise_level} in Section~\ref{sec.convergence}).  Overall, the main aim of our convergence analysis is to show that, under reasonable assumptions, Algorithm~\ref{alg.exact_alg} generates merit parameter values that remain bounded below by a positive number and generates iterates such that, at least for some iteration index $k \in \NN$, the value $\Delta \mtrue_k (\dtrue_k, \tautrue_k)$ (see \eqref{eq.aim}) is below a threshold that depends on the magnitude of the noise in the function and derivative values.

We also show under looser assumptions that the algorithm generates iterates such that, at least for some $k \in \NN$, the value $\|\Jtrue_k^T\ctrue_k\|^2$ is below a threshold that depends on the magnitude of the noise.  This result, which is arguably of less interest in practice, is relegated to Appendix~\ref{app.app}.

\subsection{Assumptions and Well-Posedness}\label{sec.assumptions}

Before proceeding to our analysis, let us introduce the assumptions that we make beyond Assumption~\ref{assumption.err0}.  We also make an additional assumption on the bounds on the noisy function and derivative evaluations; this is presented after Lemma~\ref{lemma.gamma_diff} since it relies on quantities introduced in that lemma.

First, we formalize our assumption pertaining to the sequence $\{W_k\}$ as follows.  Recall that this assumption with respect to the noisy quantities has already been justified in the paragraph following~\eqref{eq.perturbed_newton_sys}.  We strengthen it slightly to refer to noiseless values as well, since later on this will allow us to quantify, e.g., the difference between $d_k$ and $\dtrue_k$.  (Alternatively, we could introduce a ``true'' matrix $\Wtrue_k$ that satisfies similar properties with respect to noiseless quantities as $W_k$ does with respect to noisy ones.  However, this would add extra complication without adding significant value to our analysis.)

\begin{assumption} \label{assumption.bounded}
  \textit{
  The symmetric matrices $\{W_k\}$ defined by \eqref{eq.W} is chosen such that: (a) for all $k \in \NN$, \eqref{eq.perturbed_newton_sys} has a unique solution when either the noisy quantities $(J_k, g_k, v_k)$ or noiseless quantities $(\Jtrue_k, \gtrue_k, \vtrue_k)$ are used; (b)~for some $\sigma_W \in (0,\infty)$ and all $k \in \NN$, one has $p^T W_k p \geq \sigma_W \|p\|^2$ for all $p \in \mathrm{null}(J_k) \cup \mathrm{null}(\Jtrue_k)$; and (c) for some $\kappa_W \in (0,\infty)$ and all $k \in \NN$, one has $\|W_k\| \leq \kappa_W$.
  }
\end{assumption}

Second, while most of our analysis merely requires that the generated slack iterates have positive components, our later results require that the slack variables remain bounded.  Thus, we introduce Assumption~\ref{assumption.s_bounded}, below.  In the noiseless setting, boundedness of the slack variables can be proved by the behavior of the algorithm and boundedness of the function and derivative values, as in Assumption~\ref{assumption.err0}; see \cite{byrd2000trust,CurtScheWaec10}.  However, in our noisy setting, it happens that a similar conclusion cannot be drawn in the same way since the merit function values are not monotonically decreasing.  In any case, we contend that the following Assumption~\ref{assumption.s_bounded} is weak under Assumption~\ref{assumption.err0}.

\begin{assumption}\label{assumption.s_bounded}
  \textit{
  There exists $s_{\sup} \in (0,\infty)$ such that $\|s_k\| \leq s_{\sup}$ for all $k \in \NN$.  Thus, along with Assumption~\ref{assumption.err0}, $\{\phitrue(z_k,\tau_k) \}$ is bounded below.
  }
\end{assumption}

\noindent
Under Assumptions~\ref{assumption.err0} and \ref{assumption.s_bounded}, there exists $(g_{\sup}, J_{\sup}) \in (0, \infty)^2$ such that 
 \begin{equation}\label{eq.J_g_bnd}
   \|\gtrue_k\|\leq  g_{\sup}, \ \ \text{and} \ \  \|\Jtrue_k \|\leq  J_{\sup} \ \  \text{for all} \ \ k \in \NN.   
 \end{equation}

Third, many parts of our analysis merely require that $\Jtrue_k$ has full row rank for all $k \in \NN$, which in our setting is guaranteed since $s_k$ has positive components for all $k \in \NN$.  However, some of our results require that, at least eventually, the Jacobians have singular values that are bounded above the noise level for the Jacobian values.  Hence, let us introduce the following assumption.

\begin{assumption}\label{assumption.sigma_J}
  \textit{
  There exist $k_J \in \NN$ and $\gamma \in (\epsilon_J,\infty)$ such that, for all $k \in \NN$ with $k \geq k_J$, one has that $\sigma_{\min}(J_k) \geq \gamma$.
  }
\end{assumption}

\noindent
Under Assumptions~\ref{assumption.err0} and \ref{assumption.sigma_J}, it follows that for all $k \in \NN$ with $k \geq k_J$ one has $\sigma_{\min}(\Jtrue_k) \geq \gamma - \eps_J > 0$.  Thus, there is $(\sigmabar_J, \sigma_J) \in (0, \infty)^2$ such that, for such $k$,
\begin{equation}\label{eq.sigma_J}
   \| (\Jtrue_k \Jtrue_k^T)^{-1} \| \leq \sigmabar_J   \quad \text{and} \quad  \| (J_k J_k^T)^{-1} \| \leq \sigma_J.
\end{equation}

\noindent
It is worth emphasizing that both Assumption~\ref{assumption.sigma_J} and upcoming Assumption~\ref{ass.noise_level} impose implicit bounds on the noise level for the constraint Jacobian matrices.

Let us close this subsection by showing that, under our stated assumptions, the algorithm is well posed.  For this purpose, it is worthwhile to note that the normal step $v_k$, which is the exact solution of \eqref{eq.normal_tr_step}, satisfies the Cauchy decrease condition~\cite{nocedal1999numerical}, i.e., for some prescribed $\delta_v \in (0,1]$, one has
\begin{equation}\label{eq.cauchy_dec}
  \|c_k \| - \|c_k + J_k v_k \| \geq \delta_v (\|c_k\| - \|c_k + \alphahat_k J_k \hat v_k\|) \geq 0,
\end{equation}
where $\hat v_k := - J_k^T c_k$ and $\alphahat_k$ is a solution of the one-dimensional problem
\begin{equation}\label{eq.bar_alpha}
  \min_{\alphahat}\ \tfrac{1}{2} \|c_k + \alphahat J_k \hat v_k \|^2\ \ \st \ \ \alphahat \in [0,\omega].
\end{equation}
A Cauchy decrease condition also holds for the noiseless normal step $\vtrue_k$ with $(\ctrue_k,\Jtrue_k)$ in place of $(c_k,J_k)$ and with $\alphahat_k$ redefined accordingly.

The lemma below summarizes prior remarks to show that, by its construction, Algorithm~\ref{alg.exact_alg} is well defined in the sense that if it reaches iteration $k \in \NN$, then each step involves finite computation and the algorithm will either terminate in iteration $k$ or proceed to iteration $k+1$.  The lemma also shows that, as claimed, one finds $s_k \geq 0$ and $c(z_k) \geq 0$ for all generated $k \in \NN$.
\begin{lemma}
  Suppose that Assumptions~\ref{assumption.err0} and \ref{assumption.bounded} hold.  Then, if Algorithm~\ref{alg.exact_alg} reaches iteration $k \in \NN$ such that $s_k \geq 0$ and $c(z_k) \geq 0$, then it either terminates finitely or proceeds to iteration $k+1$ such that $s_{k+1} \geq 0$ and $c(z_{k+1}) \geq 0$.  Thus, the algorithm either terminates finitely or generates an infinite sequence of iterates, where in either case $s_k \geq 0$ and $c(z_k) \geq 0$ for all generated $k \in \NN$. 
\end{lemma}
\begin{proof}
  Suppose that Algorithm~\ref{alg.exact_alg} reaches iteration $k \in \NN$ such that $s_k \geq 0$ and $c(z_k) \geq 0$.  First, if it terminates in line~\ref{st.tm1}, then there is nothing left to prove.  Hence, we may proceed under the assumption that it continues, in which case one finds that \eqref{eq.normal_tr_step} is feasible and there exists $v_k \in \text{range}(J_k^T)$ that is feasible for \eqref{eq.normal_tr_step} and satisfies \eqref{eq.cauchy_dec}, meaning that the computation of $v_k$ is well defined.  It then follows under Assumption~\ref{assumption.bounded} that the computation of $u_k$ and $(d_k,y_{k+1})$ is well defined.  Next, either the algorithm terminates in line~\ref{st.tm2} or it reaches line~\ref{step.tau}, in which case it is guaranteed to reach line~\ref{step.alpha_max}.  Since $\eta_s \in (0,1)$, line~\ref{step.alpha_max} is well defined.  Similarly, since $\vareps_k \in (0,\infty)$, it follows that \eqref{eq.armijo} holds for sufficiently small~$\alpha_k$, meaning that line~\ref{step.alpha} is well defined.  It follows that line~\ref{step.slack_reset} is reached, which guarantees that $s_{k+1} \geq 0$ and $c(z_{k+1}) \geq 0$.
\end{proof}

\subsection{Convergence Analysis}\label{sec.convergence}

Let us now proceed to the main part of our analysis.  Since we have already discussed the consequences of finite terminate of the algorithm in Section~\ref{sec.discussion}, our analysis in this section presumes that the algorithm does not terminate.

Let us define for all $k \in \NN$ the errors
\begin{equation}\label{eq.errors_k}
  \epskv := v_k - \vtrue_k,\ \ \epsku := u_k - \utrue_k,\ \ \epskd := d_k - \dtrue_k,\ \ \text{and}\ \ \epsktau := \tau_k - \tautrue_k.
\end{equation}
Our preliminary analysis, prior to our main results, focuses on providing bounds on these errors in terms of the bounds on the errors from \eqref{eq.errors3}.  Observe that, since $\{\tau_k\}$ and $\{\tautrue_k\}$ are monotonically nonincreasing and bounded below by zero, it follows trivially that $\epsktau \leq \tau_0$ for all $k \in \NN$.

Our next few lemmas focus on properties of the constraint function value, constraint Jacobian value, and normal step.  The lemmas involve properties of these quantities as well as differences between noisy and noiseless values.

\begin{lemma}\label{lemma.violet_orange_olive}
 Suppose Assumptions~\ref{assumption.err0} and \ref{assumption.bounded} hold. Then, for all $k \in \NN$,
 \begin{subequations}
 \begin{align}
    \|J_k^T J_k - \Jtrue_k^T \Jtrue_k \| &\leq \eps_J ( \eps_J + 2 \|\Jtrue_k\| ), \label{eq.green0} \\
   \|J_k J_k^T - \Jtrue_k \Jtrue_k^T \| &\leq \eps_J ( \eps_J + 2 \|\Jtrue_k\| ), \label{eq.violet0} \\
   \|(J_k J_k^T)^{1/2} - (\Jtrue_k \Jtrue_k^T)^{1/2} \| &\leq (\eps_J ( \eps_J + 2 \|\Jtrue_k\| ) )^{1/2}, \label{eq.orange0} \\
   \text{and}\ \ \|\Jtrue_k^T \ctrue_k - J_k^T c_k\| &\leq \|\Jtrue_k\| \eps_c + (\|\ctrue_k\| + \eps_c) \eps_J. \label{eq.olive0}
  \end{align}
  \end{subequations}
\end{lemma}
\begin{proof}
  Consider arbitrary $k \in \NN$.  By Assumption~\ref{assumption.err0}, the triangle inequality, and submultiplicity of the matrix 2-norm, one finds that
   \begin{align*}
     &\ \left\|J_k^T J_k - \Jtrue_k^T \Jtrue_k \right\| = \left\| J_k^T J_k - J_k^T \Jtrue_k + J_k^T \Jtrue_k - \Jtrue_k^T \Jtrue_k \right\| \\
     \leq&\ \|J_k^T\| \eps_J + \|\Jtrue_k\| \eps_J \leq (\eps_J + \|\Jtrue_k^T\|) \eps_J + \|\Jtrue_k\| \eps_J = \eps_J ( \eps_J + 2 \|\Jtrue_k\| ),
  \end{align*} 
which gives~\eqref{eq.green0}.  The proof of inequality \eqref{eq.violet0} is similar to that of~\eqref{eq.green0}.  Next, by $\tfrac12$-H\"older continuity of the square root operator for symmetric positive definite matrices and the fact that the matrix 2-norm is unitarily invariant, inequality~\eqref{eq.orange0} follows from inequality~\eqref{eq.green0}.  Lastly, by \eqref{eq.errors3}, one finds
  \begin{equation*}
    \|\Jtrue_k^T \ctrue_k - J_k^T c_k \| \leq \|\Jtrue_k^T \ctrue_k - \Jtrue_k^T c_k + \Jtrue_k^T c_k - J_k^T c_k \| \leq \|\Jtrue_k\| \eps_c + (\|\ctrue_k\| + \eps_c) \eps_J, 
  \end{equation*} 
  which gives \eqref{eq.olive0}.
\end{proof}

The following useful lemma is well known.  Its proof can be found in \cite{byrd2000trust}.

\begin{lemma}\label{eq.prelim_alpha}
  Given $a \in \RR$, $b \in [0,\infty)$, and $\omega \in (0,\infty)$, the optimal value $\Phi^*$ of
  \begin{equation*}
    \min_{z \in [0,\omega)}\ \Phi(z) \equiv \tfrac{1}{2} z^2 a - z b\ \ \text{satisfies}\ \ \Phi^* \leq -\tfrac{b}{2} \min \{\tfrac{b}{|a|}, \omega\}.
  \end{equation*}
\end{lemma}

Lemma~\ref{eq.prelim_alpha} can be used to prove our next lemma.  A proof of a similar result, with less detail on the form of the lower bound that is proved, can be found in \cite{curtis2010matrix}.  We provide a proof here to show the precise form of the lower bound.

\begin{lemma}\label{lemma.prelim_cauchy_det}
  Suppose Assumptions~\ref{assumption.err0} and \ref{assumption.bounded} hold.  Then, for all $k \in \NN$,
 \begin{equation*}
   \|\ctrue_k\| (\|\ctrue_k \| - \|\ctrue_k + \Jtrue_k \vtrue_k \|) \geq \xione\|\Jtrue_k^T \ctrue_k\|^2,\ \text{where}\ \xione= \tfrac{1}{2} \delta_v \min \left\{\frac{1}{\|\Jtrue_k^T \Jtrue_k\|},  \omega \right\}.
  \end{equation*}
\end{lemma}
\begin{proof}
  Consider arbitrary $k \in \NN$ and recall that $\vtrue_k$ is the solution of \eqref{eq.normal_tr_step} when $(c_k,J_k)$ is replaced by $(\ctrue_k,\Jtrue_k)$. If $\|\Jtrue_k^T \ctrue_k\| = 0$, then $\vtrue_k = 0$, so the desired conclusion follows trivially.  Hence, we may proceed under the assumption that $\|\Jtrue_k^T \ctrue_k\| > 0$.  By \eqref{eq.bar_alpha} with respect to noiseless quantities, for which the steepest-descent direction is $\hat v_k = -\Jtrue_k^T\ctrue_k$, Cauchy decrease is attained with
  \begin{equation*}
    \hat\alpha_k \in \arg\min_{\hat\alpha \in [0,\omega]}\ \tfrac{1}{2} \hat\alpha^2 \|\Jtrue_k \Jtrue_k^T \ctrue_k\|^2 - \hat\alpha \|\Jtrue_k^T \ctrue_k\|^2.
  \end{equation*}
  Hence, by Lemma~\ref{eq.prelim_alpha} and $\|\Jtrue_k^T \ctrue_k\|^2 \|\Jtrue_k^T \Jtrue_k\| \geq {\|\Jtrue_k \Jtrue_k^T \ctrue_k\|^2}$, it follows that
  \begin{align}
    \tfrac{1}{2} ( \|\ctrue_k + \hat\alpha_k \Jtrue_k \hat{v}_k \|^2 - \|\ctrue_k\|^2) &= \tfrac{1}{2} \hat\alpha_k^2 \|\Jtrue_k \Jtrue_k^T \ctrue_k\|^2 - \hat\alpha_k \|\Jtrue_k^T \ctrue_k \|^2 \nonumber \\
    &\leq -\tfrac{1}{2} \|\Jtrue_k^T \ctrue_k\|^2 \min \left\{\frac{\|\Jtrue_k^T \ctrue_k\|^2}{\|\Jtrue_k \Jtrue_k^T \ctrue_k\|^2},  \omega \right\} \nonumber \\
    &\leq -\tfrac{1}{2} \|\Jtrue_k^T \ctrue_k\|^2 \min \left\{\frac{1}{\|\Jtrue_k^T \Jtrue_k\|},  \omega \right\}. \label{eq.cauchy_dec_lb}
  \end{align}
  Since $\vtrue_k$ satisfies Cauchy decrease (see \eqref{eq.cauchy_dec}) with respect to noiseless quantities and since $2 a (a - b) \geq a^2 - b^2$ for arbitrary $(a,b) \in \RR \times \RR$, one finds that
  \begin{align*}
    \|\ctrue_k\| ( \|\ctrue_k \| - \|\ctrue_k + \Jtrue_k \vtrue_k \|) 
    &\geq \delta_v \|\ctrue_k\| ( \|\ctrue_k \| - \|\ctrue_k + \hat\alpha_k \Jtrue_k \hat{v}_k \|) \\
    &\geq  \tfrac{1}{2} \delta_v(\|\ctrue_k\|^2 - \|\ctrue_k + \hat\alpha_k \Jtrue_k \hat{v}_k\|^2 ) \\
    &\geq  \tfrac{1}{2}  \delta_v\|\Jtrue_k^T \ctrue_k \|^2 \min\left\{ \frac{1}{\|\Jtrue_k^T \Jtrue_k\|}, \omega\right\},
  \end{align*}
  which is the desired conclusion.
\end{proof}

We now show that, under all of our aforementioned assumptions, the difference between the noisy and noiseless normal step is bounded uniformly for all sufficiently large $k \in \NN$.  The proof follows a similar argument as in~\cite{BeraShiZhou}.

\begin{lemma}\label{lemma.ub_epskv}
  Suppose Assumptions~\ref{assumption.err0}, \ref{assumption.bounded}, \ref{assumption.s_bounded}, and \ref{assumption.sigma_J} hold.  Then, there exists $\eps_v \in [0, \infty)$ such that for all $k \in \NN$ with $k \geq k_J$ one has
  \begin{equation}\label{eq.ub_epskv}
    \|\vtrue_k - v_k\| \leq \eps_v.
  \end{equation}
  In addition, \eqref{eq.ub_epskv} holds with $\eps_v \to 0$ as $(\eps_c,\eps_J) \searrow (0,0)$.
\end{lemma}
\begin{proof}
  Consider arbitrary $k \in \NN$ with $k \geq k_J$.  Since \eqref{eq.normal_tr_step} requires $v \in \mathrm{range}(J_k^T)$, one can express the solution of \eqref{eq.normal_tr_step} as $v_k = J_k^Tt_k$, where $t_k$ solves
  \begin{equation}\label{eq.normal_tr_step_cov_1}
    \min_{t \in \RR^q}\ \ \tfrac{1}{2} \| c_k + J_k J_k^T t  \|^2\ \ \st\ \ \|J_k^T t \| \leq \omega \|J_k^T c_k\|.
  \end{equation}
  Let us now define, as allowed under Assumption~\ref{assumption.sigma_J}, the change of variables
  \begin{equation}\label{eq.tbar_that}
   \that := \JCfrac (J_k J_k^T)^{-1/2} (\Jtrue_k \Jtrue_k^T)^{1/2} \tbar.
  \end{equation}
  Since $\|(J_k J_k^T)^{1/2} \that\| = \|J_k^T \that\|$ and $\|(\Jtrue_k \Jtrue_k^T)^{1/2} \tbar\| = \|\Jtrue_k^T \tbar\|$, it follows that $\|J_k^T \that\| = \JCfrac \| \Jtrue_k^T \tbar\|$.  Now considering \eqref{eq.normal_tr_step_cov_1} with respect to the noiseless quantities, but replacing the variable $\tbar$ with $\tfrac{\|\Jtrue_k^T \ctrue_k\|}{\|J_k^T c_k\|}  (\Jtrue_k \Jtrue_k^T)^{-1/2} (J_k J_k^T)^{1/2} \that$, one obtains
  \begin{equation}\label{eq.normal_tr_step_det_cov_2}
    \min_{\that \in \RR^q}\ \ \tfrac{1}{2} \| \ctrue_k + \tfrac{\|\Jtrue_k^T \ctrue_k\|}{\|J_k^T c_k\|} (\Jtrue_k \Jtrue_k^T)^{1/2} (J_k J_k^T)^{1/2} \that \|^2\ \ \st\ \ \|J_k^T \that \| \leq \omega \|J_k^T c_k\|.
  \end{equation}
  The feasible regions of~\eqref{eq.normal_tr_step_cov_1} and \eqref{eq.normal_tr_step_det_cov_2} are identical.  Thus, let us denote the feasible region of these problems as the nonempty, closed, and convex set $\Pi := \{t \in \RR^q: \|J_k^T t\| \leq \omega \|J_k^T c_k\|\}$.  Problems~\eqref{eq.normal_tr_step_cov_1} and \eqref{eq.normal_tr_step_det_cov_2} can now be written respectively as 
  \begin{align}
    &\min_{t \in \Pi}\ \ \tfrac{1}{2} \| (J_k J_k^T)^{-1} c_k +  t  \|_{(J_k J_k^T)^2}^2, \label{eq.normal_tr_step_proj} \ \ \text{and} \\
    &\min_{\that \in \Pi}\ \ \tfrac{1}{2} \left\| \JCfrac (J_k J_k^T)^{-1/2} (\Jtrue_k \Jtrue_k^T)^{-1/2} \ctrue_k +   \that   \right\|_{(J_k J_k^T)^{1/2} (\Jtrue_k \Jtrue_k^T) (J_k J_k^T)^{1/2}}^2. \label{eq.normal_tr_step_det_proj}
  \end{align}
  Let $(t^*, \that^*) \in \Pi \times \Pi$ be the optimal solutions of~\eqref{eq.normal_tr_step_proj} and \eqref{eq.normal_tr_step_det_proj}, respectively.  (Since their objective functions are strongly convex, these solutions are unique.)  By first-order optimality, respectively, for all $t \in \Pi$ one has
  \begin{align*}
    \langle - (J_k J_k^T)^{-1} c_k - t^*, t - t^* \rangle_{(J_k J_k^T)^2} &\leq 0 \\
    \langle \JCfrac (J_k J_k^T)^{-1/2} (\Jtrue_k \Jtrue_k^T)^{-1/2} \ctrue_k + \that^*,  \that^* - t \rangle_{(J_k J_k^T)^{1/2} \Jtrue_k \Jtrue_k^T (J_k J_k^T)^{1/2}} &\leq 0.
  \end{align*}
  Substituting $\that^*$ for $t$ in the first inequality, substituting $t^*$ for $t$ in the second inequality, and summing the results, one finds that $I \leq 0$, where
  \begin{samepage}
  \begin{align*}
    &\ I := - ((J_k J_k^T)^{-1} c_k + t^*)^T (J_k J_k^T)^2 (\that^* - t^*) \\
    &\ + (\JCfrac (J_k J_k^T)^{-1/2} (\Jtrue_k \Jtrue_k^T)^{-1/2}\ctrue_k +  \that^*)^T (J_k J_k^T)^{1/2}  \Jtrue_k \Jtrue_k^T(J_k J_k^T)^{1/2} (\that^* - t^*).
  \end{align*}
  \end{samepage}
  By expanding the inner products in $I$, one may write $I = I_1 + I_2$, where 
  \begin{align*}
    I_1 &:= - (t^*)^T (J_k J_k^T)^2   (\that^* - t^*) +  (\that^*)^T  (J_k J_k^T)^{1/2}  \Jtrue_k \Jtrue_k^T (J_k J_k^T)^{1/2} (\that^* - t^*) \\ \text{and}\ 
    I_2 &:= \JCfrac \ctrue_k^T (\Jtrue_k \Jtrue_k^T)^{1/2} (J_k J_k^T)^{1/2} (\that^* - t^*)  - c_k^T J_k J_k^T (\that^* - t^*).
  \end{align*}
  By adding and subtracting $(\that^*)^T (J_k J_k^T)^2 (\that^* - t^*)$, $I_1$ can be written as 
  \begin{align*}
  I_1
    =&\ (\that^* - t^*)^T (J_k J_k^T)^2  (\that^* - t^*) \nonumber \\
     &\ + (\that^*)^T (J_k J_k^T)^{1/2} \left( \Jtrue_k \Jtrue_k^T - J_k J_k^T \right) (J_k J_k^T)^{1/2} (\that^* - t^*) \nonumber \\ 
    =&\ \| \that^* - t^* \|_{(J_k J_k^T)^2}^2 +  (\that^*)^T (J_k J_k^T)^{1/2} \left( \Jtrue_k \Jtrue_k^T - J_k J_k^T \right) (J_k J_k^T)^{1/2} (\that^* - t^*),
  \end{align*}
  In addition, by adding and subtracting both $\ctrue_k^T (\Jtrue_k \Jtrue_k^T)^{1/2} (J_k J_k^T)^{1/2} (\that^* - t^*)$ and $c_k^T (\Jtrue_k \Jtrue_k^T)^{1/2} (J_k J_k^T)^{1/2} (\that^* - t^*)$, $I_2$ can be written as
  \begin{align*}
    &\ I_2 = (-(1 - \JCfrac) \ctrue_k^T (\Jtrue_k \Jtrue_k^T)^{1/2} \nonumber \\
    &\ + (\ctrue_k - c_k)^T (\Jtrue_k \Jtrue_k^T)^{1/2} + c_k^T ( (\Jtrue_k \Jtrue_k^T)^{1/2} - (J_k J_k^T)^{1/2} ) ) (J_k J_k^T)^{1/2} (\that^* - t^*).
  \end{align*}
  Thus, $I = I_1 + I_2 \leq 0$ along with these inequalities yields
  \begin{align*}
    &\ \| \that^* - t^* \|_{(J_k J_k^T)^2}^2 \\
    \leq&\ ( (\that^*)^T (J_k J_k^T)^{1/2} ( J_k J_k^T - \Jtrue_k \Jtrue_k^T ) + (1 - \JCfrac ) \ctrue_k^T (\Jtrue_k \Jtrue_k^T)^{1/2}  \\
        &\ + (c_k - \ctrue_k)^T (\Jtrue_k \Jtrue_k^T)^{1/2} + c_k^T ((J_k J_k^T)^{1/2} - (\Jtrue_k \Jtrue_k^T)^{1/2}) ) (J_k J_k^T)^{1/2} (\that^* - t^*).  
  \end{align*}
  Thus, by submultiplicity of the matrix 2-norm, the facts that $\|(J_k J_k^T)^{1/2} \that^*\| = \|J_k^T \that^*\|$ and $\|(\Jtrue_k \Jtrue_k^T)^{1/2} \ctrue_k\| = \|\Jtrue_k^T \ctrue_k\|$, the definition of the feasible set $\Pi$, and the fact that for any pair of vectors $(a, b) \in \RR^n \times \RR^n$ one has $| \|a\| - \|b\| | \leq \|a - b\|$, it follows that
  \begin{align*}
    &\ \|\that^* - t^*\|_{(J_k J_k^T)^2}^2 \nonumber \\
    \leq&\ ( \|J_k^T \that^*\| \| J_k J_k^T - \Jtrue_k \Jtrue_k^T \| + |1 - \JCfrac| \|\Jtrue_k^T \ctrue_k\| + \|c_k - \ctrue_k \| \|(\Jtrue_k \Jtrue_k^T)^{1/2}\| \nonumber \\
     & \ \ + \|c_k\| \|(J_k J_k^T)^{1/2} - (\Jtrue_k \Jtrue_k^T)^{1/2} \| ) \|(J_k J_k^T)^{1/2} (\that^* - t^*) \| \nonumber \\
    \leq&\ (\omega \| J_k^T c_k \| \| J_k J_k^T - \Jtrue_k \Jtrue_k^T \| + \| \Jtrue_k^T \ctrue_k - J_k^T c_k \| + \|c_k - \ctrue_k \| \|\Jtrue_k\| \nonumber \\
  & \ \ +  \|c_k\| \|(J_k J_k^T)^{1/2} - (\Jtrue_k \Jtrue_k^T)^{1/2} \| ) \|(J_k J_k^T)^{1/2} \| \| \that^* - t^* \|.
  \end{align*}
  This bound, submultiplicity of the matrix 2-norm, and Lemma~\ref{lemma.violet_orange_olive} yield
  \begin{align*}
    &\ \| \that^* - t^* \|_{(J_k J_k^T)^2}^2 \\
    \leq&\ (\omega (\|\Jtrue_k\| + \eps_J) (\|\ctrue_k \| + \eps_c) \eps_J (\eps_J + 2 \|\Jtrue_k\| ) + 2 \|\Jtrue_k\| \eps_c + (\|\ctrue_k\| + \eps_c) \eps_J\\
    &\ + (\|\ctrue_k\| + \eps_c) (\eps_J ( \eps_J + 2 \|\Jtrue_k\| ))^{1/2} ) (\|\Jtrue_k\| + \eps_J) \| \that^* - t^* \|. 
  \end{align*}
  Hence, along with $\| \that^* - t^* \|^2 \leq \|(J_k J_k^T)^{-2} \| \| \that^* - t^* \|_{(J_k J_k^T)^2}^2$ and \eqref{eq.sigma_J} one has  
   \begin{align}
     &\ \| \that^* - t^* \| \nonumber \\
     \leq&\ \sigma_J^2 (\|\Jtrue_k\| + \eps_J) (\omega (\|\Jtrue_k\| + \eps_J) (\|\ctrue_k \| + \eps_c) \eps_J ( \eps_J + 2 \|\Jtrue_k\| ) \nonumber \\
     &\ + 2 \|\Jtrue_k\| \eps_c + (\|\ctrue_k\| + \eps_c) \eps_J + (\|\ctrue_k\| + \eps_c) (\eps_J ( \eps_J + 2 \|\Jtrue_k\| ))^{1/2} ) \nonumber \\
     = :&\ \eps_{t_1,k}. \label{eq.that_ttilde_diff}
  \end{align}
  
  Now let $\tbar^* \in \Pi$ be the unique optimal solution of~\eqref{eq.normal_tr_step_cov_1} with respect to the noiseless quantities.  Our next aim is to bound $\|\tbar^* - t^* \|$.  That being said, since $\|\tbar^* - t^* \| \leq \| \tbar^* - \that^* \| + \| \that^* -  t^* \|$, where $\| \that^* -  t^* \|$ has the bound given in \eqref{eq.that_ttilde_diff}, we only need to bound $\| \that^* - \tbar^* \|$.  By \eqref{eq.sigma_J}, \eqref{eq.tbar_that}, and Lemma~\ref{lemma.violet_orange_olive},
  \begin{align*}
    &\ \| \tbar^* - \that^* \| \\
    =&\ \| ( I - \JCfrac (J_k J_k^T)^{-1/2} (\Jtrue_k \Jtrue_k^T)^{1/2} ) \tbar^* \| \\
    =&\ \| ( \JCfrac I - \JCfrac I + I - \JCfrac (J_k J_k^T)^{-1/2} (\Jtrue_k \Jtrue_k^T)^{1/2} ) \tbar^* \| \\
    \leq&\ \JCfrac ( |\tfrac{\|\Jtrue_k^T \ctrue_k\|}{\|J_k^T c_k\|} - 1 | + \|I - (J_k J_k^T)^{-1/2} (\Jtrue_k \Jtrue_k^T)^{1/2} \| ) \\
    &\ \cdot \| (\Jtrue_k \Jtrue_k^T)^{-1/2} \| \| (\Jtrue_k \Jtrue_k^T)^{1/2} \tbar^* \| \\
    \leq&\ \omega \|J_k^T c_k\| \| (\Jtrue_k \Jtrue_k^T)^{-1/2} \| \\
    &\ \cdot ( | 1 - \tfrac{\|\Jtrue_k^T \ctrue_k\|}{\|J_k^T c_k\|} | + \| (J_k J_k^T)^{-1/2} \| \| (J_k J_k^T)^{1/2} - (\Jtrue_k \Jtrue_k^T)^{1/2} \| ) \\
    =&\ \omega \| (\Jtrue_k \Jtrue_k^T)^{-1/2} \| ( |  \|J_k^T c_k\| - \|\Jtrue_k^T \ctrue_k\|| \\
    &\ +\|J_k^T c_k\|  \| (J_k J_k^T)^{-1/2} \| \| (J_k J_k^T)^{1/2} - (\Jtrue_k \Jtrue_k^T)^{1/2} \| ) \\
    \leq&\ \omega \sqrt{\sigmabar_J} ( \|\Jtrue_k\| \eps_c + (\|\ctrue_k\| + \eps_c) \eps_J \\
    &\ + ( \|\Jtrue_k^T \ctrue_k\| + \|\Jtrue_k\| \eps_c + \eps_J (\|\ctrue_k\| +  \eps_c)) \sqrt{\sigma_J} (\eps_J (\eps_J + 2 \|\Jtrue_k\| ) )^{1/2} ) =: \eps_{t_2,k}.
  \end{align*}
  Thus, with \eqref{eq.sigma_J} and \eqref{eq.that_ttilde_diff}, one has $\|\hat t^* - t^* \| \leq \eps_{t_1,k} + \eps_{t_2,k} =: \eps_{t,k}$, so
  \begin{align*}\label{eq.eps_v}
    \|\vtrue_k - v_k \|
      &= \| \Jtrue_k^T \hat t^* - J_k^T \hat t^* + J_k^T \hat t^* - J_k^T t^* \| \\
      &\leq \| \Jtrue_k - J_k \| \| (\Jtrue_k \Jtrue_k^T)^{-1/2} \| \| (\Jtrue_k \Jtrue_k^T)^{1/2} \hat t^* \| + \| J_k \| \| \hat t^* - t^* \| \\
      &\leq \eps_J \omega \sqrt{\sigmabar_J}  \| \Jtrue_k^T \ctrue_k\| + (\|\Jtrue_k\| + \eps_J) \eps_{t,k}.
  \end{align*}
  Thus, under Assumptions~\ref{assumption.err0} and \ref{assumption.s_bounded}, the fact that $\eps_{t,k} := \eps_{t_1,k} + \eps_{t_2,k}$, and the definitions of $\eps_{t_1,k}$ and $\eps_{t_2,k}$, the desired conclusions follow.
\end{proof}

We now present a few lemmas that focus on properties of the tangential steps.  Our goal in these lemmas is to prove a result similar to Lemma~\ref{lemma.ub_epskv}, except in terms of the tangential steps.  Our derivation of this result relies on expressing each tangential step as the result of an oblique projection onto the null space of a constraint Jacobian.  Hence, our first result pertains to oblique projection matrices.  The result is known, but we present it for the sake of completeness.  For a matrix $A \in \RR^{n \times n}$, the $M$-adjoint of $A$ is a matrix $A^* \in \RR^{n \times n}$ such that $\langle u, Av \rangle_{M} = \langle A^* u, v \rangle_{M}$ for all $(u, v) \in \RR^n \times \RR^n$. For a symmetric positive definite matrix $M$, the $M$-adjoint of $A$ equals $A^* = M^{-1} A^T M$. In addition, a matrix $P$ is called an $M$-orthogonal projection matrix if and only if it is idempotent ($P^2 = P$) and $M$-self-adjoint ($P = P^*$).

\begin{lemma}\label{lemma.prelim_tang_1}
  Suppose that the matrix $J \in \RR^{q \times (n+q)}$ has full row rank, the columns of the matrix $Z \in \RR^{(n+q) \times n}$ form an orthonormal basis for $\mathrm{null}(J)$, and $\Gamma \in \RR^{(n+q) \times (n+q)}$ is symmetric and positive definite.  Then,
  \begin{equation*}
    P := Z (Z^T \Gamma Z)^{-1} Z^T \Gamma = I - \Gamma^{-1} J^T (J \Gamma^{-1} J^T)^{-1} J \in \RR^{(n+q) \times (n+q)}
  \end{equation*}
  is a $\Gamma$-orthogonal projection matrix that projects $u \in \RR^{n+q}$ onto $\mathrm{null}(J)$ and
  \begin{equation*}
    Q := \Gamma^{-1} J^T (J \Gamma^{-1} J^T)^{-1} J \in \RR^{(n+q) \times (n+q)}
  \end{equation*}
  is a $\Gamma$-orthogonal projection matrix that projects $v \in \RR^{n+q}$ onto $\mathrm{range}(\Gamma^{-1}J^T)$, i.e., the space that is $\Gamma$-orthogonal to $\mathrm{null}(J)$.
\end{lemma}
\begin{proof}
  It is straightforward to verify that the stated matrices for $P$ are equal and that both $P$ and $Q$ are idempotent and $\Gamma$-self-adjoint.  Furthermore, given any $u \in \RR^{n+q}$, one finds that $JPu = 0$, which shows that $Pu$ lies in $\mathrm{null}(J)$, and given any $v \in \RR^{n+q}$, one finds that $Qv \in \mathrm{range}(\Gamma^{-1}J^T)$, as claimed.
\end{proof}

Let us now introduce the following result, which shows that the noisy and noiseless tangential steps can be expressed in terms of solutions of linear systems involving certain positive definite matrices that, under our assumptions, have eigenvalues that are uniformly contained within a positive interval.

\begin{lemma}\label{lemma.equiv_u_lambda}
  Suppose that Assumptions~\ref{assumption.err0}, \ref{assumption.bounded}, \ref{assumption.s_bounded}, and \ref{assumption.sigma_J} hold.  Then, there exists $(\lambda,\sigma_\Gamma) \in [0,\infty) \times (0,\infty)$ such that, for all $k \in \NN$, one can express
  \begin{equation}\label{eq.equiv_u_lambda}
    \begin{aligned}
      \begin{bmatrix} W_k + \lambda J_k^T J_k & J_k^T \\ J_k & 0 \end{bmatrix} \begin{bmatrix} u_k \\ y_{k+1} \end{bmatrix} &= -\begin{bmatrix} g_k + W v_k \\ 0 \end{bmatrix} \\ \text{and}\ \ \begin{bmatrix} W_k + \lambda \Jtrue_k^T \Jtrue_k  & \Jtrue_k^T \\ \Jtrue_k & 0 \end{bmatrix} \begin{bmatrix} \utrue_k \\ \bar y_{k+1} \end{bmatrix} &= -\begin{bmatrix} \gtrue_k + W \vtrue_k \\ 0 \end{bmatrix},
    \end{aligned}
  \end{equation}
  where $\Gamma_k := W_k + \lambda J_k^TJ_k$ and $\Gammatrue_k := W_k + \lambda \Jtrue_k^T\Jtrue_k$ satisfy
  \begin{equation*}
    \min\{\sigma_{\min}(\Gamma_k),\sigma_{\min}(\Gammatrue_k)\} \geq \sigma_\Gamma.
  \end{equation*}
\end{lemma}
\begin{proof}
  Consider arbitrary $k \in \NN$.  By the second block equation in \eqref{eq.perturbed_newton_sys}, one finds that $u_k \in \mathrm{null}(J_k)$ and $\utrue_k \in \mathrm{null}(\Jtrue_k)$.  Therefore, by \eqref{eq.perturbed_newton_sys}, the pairs $(u_k,y_{k+1})$ and $(\utrue_k,\ytrue_{k+1})$ satisfy \eqref{eq.equiv_u_lambda} for any $\lambda \in [0,\infty)$.  The remainder of the proof follows by known results; see, e.g., Section~5.4.2 of \cite{gill2019practical}.
\end{proof}

\begin{lemma}\label{lemma.gamma_diff}
  Suppose that Assumptions~\ref{assumption.err0}, \ref{assumption.bounded}, \ref{assumption.s_bounded}, and \ref{assumption.sigma_J} hold.  Then, for all $k \in \NN$ with $\Gamma_k$ and $\Gammatrue_k$ defined as in Lemma~\ref{lemma.equiv_u_lambda}, one has that
  \begin{equation}\label{eq.gamma_diff}
    \| \Gamma_k^{-1} - \Gammatrue_k^{-1} \| \leq \lambda \sigma_{\Gamma}^{-2} \eps_J (\eps_J + 2 \| \Jtrue_k\| ).
  \end{equation}
\end{lemma}
\begin{proof}
  Consider arbitrary $k \in \NN$.  By Lemma~\ref{lemma.equiv_u_lambda}, submultiplicity of the matrix 2-norm, and Lemma~\ref{lemma.violet_orange_olive}, it follows that
  \begin{align*}
    \| \Gamma_k^{-1} - \Gammatrue_k^{-1} \|
    &= \| \Gamma_k^{-1} ( \Gammatrue_k - \Gamma_k ) \Gammatrue_k^{-1} \| \\
    &\leq \| \Gamma_k^{-1} \| \| \Gammatrue_k^{-1} \| \| \Gamma_k - \Gammatrue_k \| \\
    &\leq \lambda \sigma_{\Gamma}^{-2} \|J_k J_k^T - \Jtrue_k \Jtrue_k^T\| \leq \lambda \sigma_{\Gamma}^{-2} \eps_J (\eps_J + 2 \| \Jtrue_k\| ),
  \end{align*}
  which gives the desired conclusion.
\end{proof}

We are now prepared to prove a uniform bound on the norm of the differences between the noisy and noiseless tangential steps.  The bound that we prove, stated in Lemma~\ref{lemma.ub_epsku} below, requires an additional assumption pertaining to the noise in the constraint Jacobians, which we state next.  That the assumption pertains to the noise in the constraint Jacobians specifically can be seen by observing the manner in which the matrices $\{\Gamma_k\}$ and $\{\Gammatrue_k\}$ are defined through Lemma~\ref{lemma.equiv_u_lambda} (i.e., with the same matrix $W_k$ for both the noisy and noiseless linear systems in \eqref{eq.equiv_u_lambda}).  Whereas Assumption~\ref{assumption.sigma_J} requires that, for large $k$, the noise in each constraint Jacobian estimate is small relative to singular values, the following assumption essentially requires that the noise in the constraint Jacobian estimates is small enough such that the null space defined by the noisy and noiseless values does not differ too much.

\begin{assumption}\label{ass.noise_level}
  \textit{
  Let $k_J \in \NN$ be defined as in Assumption~\ref{assumption.sigma_J}, and for all $k \in \NN$ let $\Gamma_k$ and $\Gammatrue_k$ be defined as in Lemma~\ref{lemma.equiv_u_lambda}.  Then, for all $k \in \NN$ with $k \geq k_J$,
  \begin{multline*}
    \|(\Jtrue_k \Gammatrue_k^{-1} \Jtrue_k)^{-1}\| (\|\Jtrue_k\| \|\Gamma_k^{-1} J_k^T - \Gammatrue_k^{-1} \Jtrue_k^T\| + \| \Gammatrue_k^{-1} \Jtrue_k^T \| \epsilon_J + \| \Gamma_k^{-1} J_k^T - \Gammatrue_k^{-1} \Jtrue_k^T \| \epsilon_J) < 1.
  \end{multline*}
  }
\end{assumption}

Observe that if $\eps_J = 0$, then $J_k = \Jtrue_k$ and $\Gamma_k = \Gammatrue_k$, in which case the expression on the left-hand side of the inequality in Assumption~\ref{ass.noise_level} is 0, so the assumption holds.  It follows that the assumption also holds for $\eps_J$ sufficiently small relative to the magnitudes of the elements of $\Jtrue_k$ and $W_k$ for all $k \in \NN$.

\begin{lemma}\label{lemma.ub_epsku}
  Suppose that Assumptions~\ref{assumption.err0}, \ref{assumption.bounded}, \ref{assumption.s_bounded}, \ref{assumption.sigma_J}, and \ref{ass.noise_level} hold.  Then, there exists $\eps_u \in [0,\infty)$ such that for all $k \in \NN$ with $k \geq k_J$ one has
  \begin{equation}\label{eq.ub_epsku}
    \|u_k - \utrue_k \| \leq \eps_u.
  \end{equation}
  In addition, \eqref{eq.ub_epsku} holds with $\eps_u \to 0$ as $(\eps_c,\eps_J,\eps_g) \searrow (0,0,0)$.
\end{lemma}
\begin{proof}
  Consider arbitrary $k \in \NN$ with $k \geq k_J$.  Let $Z_k$ and $\Ztrue_k$ have columns that form orthonormal bases for $\mathrm{null}(J_k)$ and $\mathrm{null}(\Jtrue_k)$, respectively. From~\eqref{eq.equiv_u_lambda},
  \begin{align*}
    u_k &= - Z_k (Z_k^T \Gamma_k Z_k )^{-1} Z_k^T \Gamma_k (\Gamma_k^{-1} g_k + \Gamma_k^{-1} W_k v_k), \\ \text{and}\ \ 
    \utrue_k &= - \Ztrue_k (\Ztrue_k^T \Gammatrue_k \Ztrue_k )^{-1} \Ztrue_k^T \Gammatrue_k (\Gammatrue_k^{-1}\gtrue_k + \Gammatrue_k^{-1} W_k \vtrue_k),
  \end{align*}
  where $\Gamma_k$ and $\Gammatrue_k$ are defined as in Lemma~\ref{lemma.equiv_u_lambda}.  By Lemma~\ref{lemma.prelim_tang_1}, the matrices
  \begin{equation*}
    P_k := Z_k (Z_k^T \Gamma_k Z_k )^{-1} Z_k^T \Gamma_k\ \ \text{and}\ \ \Ptrue_k := \Ztrue_k (\Ztrue_k^T \Gammatrue_k \Ztrue_k)^{-1} \Ztrue_k^T \Gammatrue_k
  \end{equation*}
  are a $\Gamma_k$-orthogonal projection matrix on $\mathrm{null}(J_k)$ and a $\Gammatrue_k$-orthogonal projection matrix on $\mathrm{null}(\Jtrue_k)$, respectively.  One finds that
  \begin{align}
    \| u_k - \utrue_k \|
      =&\ \|P_k (\Gamma_k^{-1} g_k + \Gamma_k^{-1} W_k v_k) - \Ptrue_k (\Gamma_k^{-1} g_k + \Gamma_k^{-1} W_k v_k) \nonumber \\
      &\ + \Ptrue_k ( \Gamma_k^{-1}g_k+  \Gamma_k^{-1} W_k v_k) - \Ptrue_k (\Gammatrue_k^{-1} \gtrue_k + \Gammatrue_k^{-1} W_k \vtrue_k)\| \nonumber \\
      \leq&\ \|P_k - \Ptrue_k \| \| \Gamma_k^{-1} g_k +  \Gamma_k^{-1} W_k v_k \| \label{eq.u_diff0} \\
      &\ + \| \Ptrue_k \| \|\Gamma_k^{-1}(g_k+   W_k v_k) -  \Gammatrue_k^{-1} (\gtrue_k +  W_k \vtrue_k) \|. \nonumber
  \end{align}
  Our aim now is to establish upper bounds for the terms on the right-hand side of \eqref{eq.u_diff0}.  To bound $\|P_k - \Ptrue_k \| $, we define $Q_k = \Gamma_k^{-1} J_k^T (J \Gamma_k^{-1}J_k^T)^{-1} J_k $ and $\Qbar_k = \Gammatrue_k^{-1} \Jtrue_k^T (\Jtrue_k \Gammatrue_k^{-1}\Jtrue_k^T)^{-1} \Jtrue_k$.  By Lemma~\ref{lemma.prelim_tang_1}, Theorem~3.4 in \cite{stewart2011numerical} regarding the properties of the oblique projection matrices (under Assumption~\ref{ass.noise_level}), one finds
  \begin{align}
          &\ \|P_k - \Ptrue_k \| = \| Q_k -  \Qbar_k \| \nonumber \\
      \leq&\ \|\Qbar_k \|(1 + \|\Qbar_k \|) \min \{ \kappa(\Gammatrue_k^{-1} \Jtrue_k^T),  \kappa ((\Jtrue_k \Gammatrue_k^{-1}\Jtrue_k^T)^{-1} \Jtrue_k) \} \frac{\|\Gamma_k^{-1} J_k^T - \Gammatrue_k^{-1} \Jtrue_k^T\|}{\|\Gammatrue_k^{-1} \Jtrue_k^T\|} \nonumber \\
          &\ +  \|\Qbar_k \|(1 + \|\Qbar_k \|)  \min \{\kappa(\Jtrue_k^T), \kappa (\Gammatrue_k^{-1} \Jtrue_k^T(\Jtrue_k \Gammatrue_k^{-1}\Jtrue_k^T)^{-1} ) \} \frac{\eps_J}{\|\Jtrue_k^T\|} \nonumber \\
          &\ + O( \|\Qbar_k \| \max\{\|\Gamma_k^{-1} J_k^T - \Gammatrue_k^{-1} \Jtrue_k^T\|, \eps_J \}^2). \label{eq.Delta_P}
  \end{align}
  By Lemmas~\ref{lemma.violet_orange_olive}, \ref{lemma.equiv_u_lambda}, and  \ref{lemma.gamma_diff} along with Assumption~\ref{assumption.err0}, one has
  \begin{align*}
    \|\Gamma_k^{-1} J_k^T - \Gammatrue_k^{-1} \Jtrue_k^T\| 
    &=  \|\Gamma_k^{-1} J_k^T - \Gamma_k^{-1} \Jtrue_k^T + \Gamma_k^{-1} \Jtrue_k^T - \Gammatrue_k^{-1} \Jtrue_k^T\| \\
    &\leq \|\Gamma_k^{-1}\| \eps_J + \|\Gamma_k^{-1} - \Gammatrue_k^{-1}\| \|\Jtrue_k\|  \\
    &\leq \sigma_{\Gamma}^{-1} (\eps_J + \lambda \sigma_{\Gamma}^{-1} \eps_J (\eps_J + 2 \|\Jtrue_k\|) \|\Jtrue_k\| ).
  \end{align*}
  Similarly, under the assumptions of the lemma, one can establish uniform upper bounds on the remaining terms on the right-hand side of \eqref{eq.Delta_P} in terms of $\epsilon_J$, $\lambda$, $\sigma_\Gamma$, $J_{I,\sup}$, and $s_{\sup}$ in such a manner that they all vanish as $\epsilon_J \searrow 0$.
  
  Let us now turn to the next term on the right-hand side in \eqref{eq.u_diff0}.  By Lemma~\ref{lemma.equiv_u_lambda}, \eqref{eq.normal_tr_step}, and the assumptions of the lemma, one finds that
  \begin{equation*}
    \| \Gamma_k^{-1} g_k +  \Gamma_k^{-1} W_k v_k \|\leq \sigma_{\Gamma}^{-1} (\|\gtrue_k\| + \eps_g + \kappa_W \omega (\|\Jtrue_k \| + \eps_J) (\|\ctrue_k\| + \eps_c) ).
  \end{equation*}
  Consequently, one has a uniform upper bound on this expression in terms of the error bounds $(\epsilon_c,\epsilon_J,\epsilon_g)$ that vanishes as $(\epsilon_c,\epsilon_J,\epsilon_g) \searrow (0,0,0)$.
  
  Next, with respect to $\|\Ptrue_k\|$, one finds that
  \begin{equation*}
    \|\Ptrue_k\| = \|I - \Gammatrue_k^{-1} \Jtrue_k^T (\Jtrue_k \Gammatrue_k^{-1}\Jtrue_k^T)^{-1} \Jtrue_k\| \leq 1 + \| \Gammatrue_k^{-1} \Jtrue_k^T  \| \|(\Jtrue_k \Gammatrue_k^{-1}\Jtrue_k^T)^{-1} \Jtrue_k\|.
  \end{equation*}
  Following a similar approach as for $\|P_k - \Ptrue_k\|$, one can derive a uniform upper bound for this expression that vanishes as $\epsilon_J \searrow 0$.
  
  Finally, with respect to the last term on the right-hand side in \eqref{eq.u_diff0}, one has along with \eqref{eq.normal_tr_step} (i.e., $\|v_k\| \leq \omega \|J_k^T c_k\|$), Lemma~\ref{lemma.equiv_u_lambda}, and Lemma~\ref{lemma.gamma_diff} that
  \begin{align*}
     &\ \|\Gamma_k^{-1}(g_k + W_k v_k) - \Gammatrue_k^{-1} (\gtrue_k +  W_k \vtrue_k) \|\\
    =&\ \|\Gamma_k^{-1}(g_k + W_k v_k) - \Gammatrue_k^{-1} (g_k + W_k v_k) \\
    &\ \ + \Gammatrue_k^{-1}(g_k + W_k v_k) - \Gammatrue_k^{-1} (\gtrue_k +  W_k \vtrue_k) \| \\
    \leq&\ \|\Gamma_k^{-1}  - \Gammatrue_k^{-1} \| \|g_k+   W_k v_k\| + \|\Gammatrue_k^{-1} \| \|(g_k - \gtrue_k) +   W_k (v_k - \vtrue_k)\| \\
    \leq&\ \sigma_{\Gamma}^{-1} ( \lambda \sigma_{\Gamma}^{-1} \eps_J (\eps_J + 2 \|\Jtrue_k\|)  ( \|\gtrue_k\| + \eps_g + \kappa_W \omega (\|\Jtrue_k\| + \eps_J) (\|\ctrue_k\| + \eps_c) ) \\
    &\ \ \ \ \ \ \ + \eps_g + \kappa_W \eps_v).
  \end{align*}
  Consequently, under the assumptions of the lemma, one can establish a uniform upper bound for this term in terms of $\epsilon_J$, $\lambda$, $\sigma_J$, $J_{I,\sup}$, $s_{\sup}$, $g_{0,\sup}$, $\mu$, $\epsilon_g$, $\kappa_W$, $\omega$, $c_{I,\sup}$, $\eps_c$, and $\eps_v$ in such a manner that the upper bound vanishes as $(\epsilon_c,\epsilon_J,\epsilon_g) \searrow (0,0,0)$.  Combined with prior results, the proof is complete.
\end{proof}

The following corollary now follows with $d_k = u_k + v_k$ and $\dtrue_k = \utrue_k + \vtrue_k$.

\begin{corollary}\label{lemma.ub_epskd}
  Suppose that Assumptions~\ref{assumption.err0}, \ref{assumption.bounded}, \ref{assumption.s_bounded}, \ref{assumption.sigma_J}, and \ref{ass.noise_level} hold.  Then, there exists $\eps_d \in [0,\infty)$ such that for all $k \in \NN$ with $k \geq k_J$ one has
  \begin{equation}\label{eq.ub_epskd}
    \|d_k - \dtrue_k \| \leq \eps_d.
  \end{equation}
  In addition, \eqref{eq.ub_epskd} holds with $\eps_d \to 0$ as $(\eps_c,\eps_J,\eps_g) \searrow (0,0,0)$.
\end{corollary}
\begin{proof}
  The proof follows from Lemmas~\ref{lemma.ub_epskv} and \ref{lemma.ub_epsku} with $\eps_d := \eps_v + \eps_u$.
\end{proof}

Observe that we have shown upper bounds for the quantities in \eqref{eq.errors_k}, as desired.  For the next phase of our analysis, we focus on properties of the computed search directions, where in particular we are interested in showing that they yield positive reductions in the model of the merit function, which in turn are reductions that (generally) vanish as the algorithm progresses.  As a brief aside, we present the following lemma, which shows that in terms of noiseless quantities the model reduction is a valid stationarity measure.  This lemma provides further detail to our claim in \eqref{eq.aim} that reducing the noiseless model reduction corresponds to improvement toward stationarity of the barrier subproblem (namely, \eqref{eq.barrier}) that the algorithm aims to solve.

\begin{lemma}\label{lemma.delta_m_stat_meas}
  Suppose that Assumptions~\ref{assumption.err0} and \ref{assumption.bounded} hold.  Then, for all $k \in \NN$ such that $\tautrue_{k-1} > 0$, one has that either
  \begin{enumerate}
    \item[(a)] $\Jtrue_k^T\ctrue_k = 0$ and $\ctrue_I(x_k) \not\leq 0$;
    \item[(b)] $\Jtrue_k^T\ctrue_k = 0$, $\ctrue_I(x_k) \leq 0$, and $\gtrue_k + \Jtrue_k^T\ytrue_{k+1} = 0$; or
    \item[(c)] $\tautrue_k > 0$ and $\Delta \mtrue_k(\dtrue_k, \tautrue_k) \geq \tfrac12 \tautrue_k \utrue_k^T W_k \utrue_k + \sigma (\|\ctrue_k\| - \|\ctrue_k + \Jtrue_k \vtrue_k\|) > 0$.
  \end{enumerate}
  Thus, if $x_k$ is not an infeasible stationary point, then $\Delta \mtrue_k(\dtrue_k, \tautrue_k) = 0$ if and only if $(x_k,y_{k+1})$ satisfies the first-order conditions for \eqref{eq.barrier}, i.e., \eqref{eq.first_order_opt}.
\end{lemma}
\begin{proof}
  Consider arbitrary $k \in \NN$ such that $\tautrue_{k-1} > 0$.  If (a) or (b) holds, then there is nothing left to prove.  Thus, one may proceed under the assumption that $\Jtrue_k^T\ctrue_k \neq 0$ and/or $\gtrue_k + \Jtrue_k^T\ytrue_{k+1} \neq 0$.  If $\Jtrue_k^T\ctrue_k \neq 0$, then it follows from well-known properites of \eqref{eq.normal_tr_step} (e.g., see \cite{nocedal1999numerical}) that $\vtrue_k \neq 0$ and $\|\ctrue_k\| - \|\ctrue_k + \Jtrue_k\vtrue_k\| > 0$, which in turn shows by \eqref{eq.delta_m_tilde}, \eqref{eq.model_reduction_cond}, and \eqref{eq.pi_trial} that (c) holds.  On the other hand, if $\Jtrue_k^T\ctrue_k = 0$, but $\gtrue_k + \Jtrue_k^T\ytrue_{k+1} \neq 0$, then from \eqref{eq.normal_tr_step} and \eqref{eq.perturbed_newton_sys} one has $\vtrue_k = 0$ and
  \begin{equation*}
    W_k \utrue_k + \Jtrue_k^T \ytrue_{k+1} = -\gtrue_k \implies \utrue_k \neq 0\ \text{and}\ 0 < \utrue_k^TW_k\utrue_k = -\gtrue_k^T\utrue_k.
  \end{equation*}
  Thus, $\tautrue_k \gets \tautrue_{k-1} > 0$ and from \eqref{eq.delta_m_tilde}--\eqref{eq.model_reduction_cond} the conditions of part (c) hold.
\end{proof} 

Following a similar line of argument as in the proof of Lemma~\ref{lemma.delta_m_stat_meas}, but with respect to the noisy quantities, the following result can be established due to our previous assumption that the algorithm does not terminate finitely.

\begin{lemma}\label{lemma.well_defined}
  Suppose that Assumptions~\ref{assumption.err0} and \ref{assumption.bounded} hold.  Then, for all $k \in \NN$, one has that $\tau_k > 0$ and $(d_k,\tau_k)$ together satisfy~\eqref{eq.model_reduction_cond}.
\end{lemma}
\begin{proof}
  Since the algorithm does not terminate, for all $k \in \NN$ either $J_k^Tc_k \neq 0$ or $g_k + J_k^Ty_{k+1} \neq 0$.  Thus, the proof follows like that of Lemma~\ref{lemma.delta_m_stat_meas}.
\end{proof}

The aim of our next three results is to establish a bound on the difference between $\Delta m_k(d_k,\tau_k)$ and $\Delta \mtrue_k(\dtrue_k,\tautrue_k)$ for all $k \in \NN$.  Ultimately, the purpose of this bound is to show that, as the algorithm is designed to drive $\Delta m_k(d_k,\tau_k)$ to small values, this in turn means that $\Delta \mtrue_k(\dtrue_k,\tautrue_k)$ is driven to be small as well.  We establish the bound in two steps, first by proving a bound on the difference between $\Delta m_k$ and $\Delta \mtrue_k$ with respect to $(d_k,\tau_k)$, then by proving a bound on the difference between $\Delta \mtrue_k$ with respect to $(d_k,\tau_k)$ and $(\dtrue_k,\tautrue_k)$.

\begin{lemma}\label{lemma.m1}
  Suppose that Assumptions~\ref{assumption.err0} and \ref{assumption.bounded} hold.  Then, for all $k \in \NN$, one has that $\Delta m_k(d_k, \tau_k) = \Delta \mtrue_k(d_k, \tau_k) + \Ekmone$, where 
  \begin{align*}
    & \Ekmone= \\
    & - (\tautrue_k +  \epsktau) \epskg^T( \dtrue_k + \epskd) + \|\ctrue_k + \epskc \| - \|\ctrue_k\|\\
    &-\left(\|\ctrue_k + \Jtrue_k \dtrue_k + \Jtrue_k \epskd  + \epskc + \epskJ \dtrue_k + \epskJ \epskd\| - \| \ctrue_k + \Jtrue_k \dtrue_k + \Jtrue_k \epskd \| \right).
  \end{align*}
  Moreover, supposing that Assumptions~\ref{assumption.s_bounded}, \ref{assumption.sigma_J}, and \ref{ass.noise_level} also hold, then for all $k \in \NN$ with $k \geq k_J$ and $\eps_d \in [0,\infty)$ defined in Corollary~\ref{lemma.ub_epskd}, one has
  \begin{equation*}
    |\Ekmone| \leq (\tautrue_k+   \epsktau) \eps_g \eps_d +  2 \eps_c + \eps_J \eps_d + \left( (\tautrue_k  +  \epsktau) \eps_g + \eps_J\right) \| \dtrue_k\| := \calEkmone.
  \end{equation*}
\end{lemma}
\begin{proof}
  For any $k \in \NN$, one finds from \eqref{eq.errors3}, \eqref{eq.delta_m_tilde}, and \eqref{eq.errors_k} that
  \begin{align*}
    &\ \Delta m_k(d_k, \tau_k) - \Delta \mtrue_k(d_k, \tau_k)  \\
    =&\ - \tau_k \epskg^T d_k+ \|c_k\| - \|\ctrue_k\| - (\|c_k + J_k d_k\| - \| \ctrue_k + \Jtrue_k d_k \| )  \\
    =&\ - (\tautrue_k +  \epsktau) \epskg^T( \dtrue_k + \epskd)  + \|\ctrue_k + \epskc \| - \|\ctrue_k\|\\
     &\ - (\|\ctrue_k + \epskc + \Jtrue_k \dtrue_k + \Jtrue_k \epskd + \epskJ \dtrue_k + \epskJ \epskd\| - \| \ctrue_k + \Jtrue_k \dtrue_k + \Jtrue_k \epskd \| ),
  \end{align*}
  as claimed.  Now by the additional assumptions of the lemma, Corollary~\ref{lemma.ub_epskd}, and the triangle inequality (applied twice), one finds that
  \begin{align*}
    |\Ekmone|
    \leq&\ (\tautrue_k + \epsktau) |\epskg^T \epskd | + \|\epskc\| \\
        &\ + \|\epskc +  \epskJ \epskd\| + ( (\tautrue_k + \epsktau) \eps_g + \eps_J ) \| \dtrue_k\| \\
    \leq&\ (\tautrue_k+   \epsktau) \eps_g \eps_d +  2 \eps_c + \eps_J \eps_d + ( (\tautrue_k +  \epsktau) \eps_g + \eps_J ) \| \dtrue_k\|,
  \end{align*}
  which completes the proof.
\end{proof}

Our next lemma provides the aforementioned complementary result.

\begin{lemma}\label{lemma.m2}
  Suppose that Assumptions~\ref{assumption.err0} and \ref{assumption.bounded} hold.  Then, for all $k \in \NN$, one has that $\Delta \mtrue_k (d_k, \tau_k) = \Delta \mtrue_k (\dtrue_k, \tautrue_k) + \Ekmtwo$, where
  \begin{equation*}
    \Ekmtwo = - \tautrue_k \gtrue_k^T \epskd - \epsktau \gtrue_k^T \dtrue_k - \epsktau \gtrue_k^T \epskd + \|\ctrue_k + \Jtrue_k \dtrue_k\| - \|\ctrue_k + \Jtrue_k \dtrue_k + \Jtrue_k \epskd\|.
  \end{equation*}
  Moreover, supposing that Assumptions~\ref{assumption.s_bounded}, \ref{assumption.sigma_J}, and \ref{ass.noise_level} also hold, then for all $k \in \NN$ with $k \geq k_J$ and $\eps_d \in [0,\infty)$ defined in Corollary~\ref{lemma.ub_epskd}, one has
  \begin{equation*}
    |\Ekmtwo| \leq (\tautrue_k + \epsktau) \|\gtrue_k\| \eps_d + |\epsktau| \|\gtrue_k\| \|\dtrue_k\| + \| \Jtrue_k\| \eps_d =: \calEkmtwo.
  \end{equation*}
\end{lemma}
\begin{proof}
  For any $k \in \NN$, one finds from \eqref{eq.errors3}, \eqref{eq.delta_m_tilde}, and \eqref{eq.errors_k} that
  \begin{align*}
    &\ \Delta \mtrue_k(d_k, \tau_k) \\
    =&\ - \tau_k \gtrue_k^T d_k + \|\ctrue_k\| - \|\ctrue_k + \Jtrue_k d_k\| \\
    =&\ - \tautrue_k \gtrue_k^T \dtrue_k + \|\ctrue_k\| -  \|\ctrue_k + \Jtrue_k \dtrue_k \| \\
     &\ - \tautrue_k \gtrue_k^T \epskd - \epsktau \gtrue_k^T \dtrue_k - \epsktau \gtrue_k^T \epskd + \|\ctrue_k + \Jtrue_k \dtrue_k\| - \|\ctrue_k + \Jtrue_k \dtrue_k + \Jtrue_k \epskd\| \\
    =&\ \Delta \mtrue_k(\dtrue_k, \tautrue_k) + \Ekmtwo,
  \end{align*}
  as claimed.  Now by the additional assumptions of the lemma, Corollary~\ref{lemma.ub_epskd}, and the triangle inequality, one finds that
  \begin{align*}
    |\Ekmtwo| 
    &\leq (\tautrue_k + \epsktau) |\gtrue_k^T \epskd | + |\epsktau| \|\gtrue_k\| \|\dtrue_k\| + \| \Jtrue_k \epskd\| \\
    &\leq (\tautrue_k + \epsktau) \|\gtrue_k\| \eps_d + |\epsktau| \|\gtrue_k\| \|\dtrue_k\| + \| \Jtrue_k\| \eps_d,
  \end{align*}
  which completes the proof.
\end{proof}

\begin{remark}\label{rem.ugly_but_it_is_what_it_is}
  If, for some $k \in \NN$, one finds $\vareps_{f,k} = \|\vareps_{c,k}\| = \|\vareps_{g,k}\| = \|\vareps_{J,k}\| = 0$, then $\Ekmone = 0$.  However, these errors being zero does not imply that $\epsktau$ is zero since one might find $\epsktau \neq 0$ due to updates of the merit parameter in prior iterations.  Thus, one finds that noise in any iteration can cause a nonzero difference between $\Delta \mtrue_k (d_k, \tau_k)$ and $\Delta \mtrue_k (\dtrue_k, \tautrue_k)$ for any subsequent index $k$.
\end{remark}

The following corollary follows easily from the previous two lemmas.  Thus, we state the corollary without a proof, since it is straightforward.

\begin{corollary}\label{cor.m}
  Suppose that Assumptions~\ref{assumption.err0} and \ref{assumption.bounded} hold.  Then, for all $k \in \NN$, one has that $\Delta m_k(d_k, \tau_k) = \Delta \mtrue_k(\dtrue_k, \tautrue_k) + \Ekm$, where $\Ekm := \Ekmone + \Ekmtwo$ with $\Ekmone$ and $\Ekmtwo$ defined in Lemma~\ref{lemma.m1} and Lemma~\ref{lemma.m2}, respectively.  Moreover, supposing that Assumptions~\ref{assumption.s_bounded}, \ref{assumption.sigma_J}, and \ref{ass.noise_level} also hold, then for all $k \in \NN$ with $k \geq k_J$ and $\eps_d \in [0,\infty)$ defined in Corollary~\ref{lemma.ub_epskd}, one has
  \begin{equation}\label{eq.m_bound}
    |\Ekm| \leq \calEkm := \calEkmone  + \calEkmtwo 
  \end{equation}
  where $\calEkmone$ and $\calEkmtwo$ are defined in Lemmas~\ref{lemma.m1} and \ref{lemma.m2}, respectively.
\end{corollary}

We have by now established bounds on differences between the noisy and noiseless model reductions.  For our next two lemmas, we explore the effect of steps on changes in the \emph{noiseless} merit function values.  By construction of the algorithm, it follows that each step yields an upper bound on the change in the \emph{noisy} merit function value given by \eqref{eq.armijo}.  However, to establish that the algorithm also ensures certain progress with respect to the noiseless quantities, the next two lemmas provide bounds with respect to the noiseless merit function $\phitrue$.  The proof of our next lemma is similar to that of \cite[Lemma 3.5]{CurtScheWaec10}.

\begin{lemma}\label{lemma.dec_phi_delta_m}
  Suppose that Assumptions~\ref{assumption.err0}, \ref{assumption.bounded}, and \ref{assumption.s_bounded}  hold.  Then, for all $k \in \NN$, one finds for all $\alpha \in (0, 1]$ with $\alpha d_k^s \geq -\eta_s$ that
  \begin{equation*}
    \phitrue(z_k+\alpha \dhat_k, \tau_k) - \phitrue(z_k, \tau_k) \leq - \alpha \Delta \mtrue_k(d_k, \tau_k) + \xizero{\tau_k} \alpha^2 \|d_k\|^2,
  \end{equation*}
  where $\xizero{\tau_k} := \max \{\tfrac{1}{2}(\tau_k L_g + L_J), \frac{\tau_k \mu}{1-\eta_s} \}$.
\end{lemma}
\begin{proof}
  For any real numbers $\xi$ and $\xi'$ with $\xi > 0$ and $\xi' \geq -\eta_s \xi$, one has
  \begin{align*}
    \left|\log(\xi + \xi') - \log \xi - \frac{\xi'}{\xi} \right| \leq \sup_{\xi'' \in [\xi,\xi+\xi']}\left| \frac{\xi'}{\xi''} - \frac{\xi'}{\xi} \right| \leq \frac{1}{1 - \eta_s} \left(\frac{\xi'}{\xi}\right)^2.
  \end{align*}
  Thus, by Taylor's theorem and Assumptions~\ref{assumption.err0}, \ref{assumption.bounded}, and \ref{assumption.s_bounded}, for any $k \in \NN$
  \begin{align*}
    &\  \phitrue(z_k+\alpha \dhat_k, \tau_k) - \phitrue(z_k, \tau_k) \\
    =&\ \tau_k \left(\ftrue_0(x_k + \alpha d_k^x) - \ftrue_0(x_k) - \mu \sum_{i=1}^q \log((s_k + \alpha S_k d_k^s)^{(i)}) + \mu \sum_{i=1}^q \log(s_k^{(i)}) \right) \\
    &\ + \| \ctrue_I(x_k + \alpha d_k^x) + (s_k + \alpha S_k d_k^s) \| - \| \ctrue_I(x_k) + s_k \| \\
    \leq&\ \tau_k \alpha \gtrue_0(x_k)^T d_k^x + \tfrac{1}{2}\tau_k \alpha^2 L_g \|d_k^x\|^2 - \tau_k\alpha \mu e^T d_k^s + \frac{\tau_k \mu}{1 - \eta_s} \alpha^2 \|d_k^s\|^2 \\
    &+ \|\ctrue_I(x_k) + \alpha \nabla \ctrue_I(x_k)^T d_k^x + (s_k + \alpha S_k d_k^s) \| - \|\ctrue_I(x_k) + s_k\| + \tfrac{1}{2} \alpha^2 L_J \|d_k^x\|^2 \\
    =&\ \tau_k \alpha \gtrue_k^T d_k + \tfrac{1}{2} \alpha^2 (\tau_k  L_g + L_J ) \|d_k^x\|^2  + \frac{\tau_k \mu}{1 - \eta_s} \alpha^2 \|d_k^s\|^2 \\
    &\ + \|\ctrue(z_k) + \alpha \Jtrue_k d_k \| - \|\ctrue(z_k)\| \\ 
    \leq&\ -\alpha (-\tau_k  \gtrue_k^T d_k + \|\ctrue(z_k) \| - \|\ctrue(z_k) + \Jtrue_k d_k\| ) + \xizero{\tau_k} \alpha^2 \|d_k\|^2\\
    =&\ -\alpha \Delta \mtrue_k(d_k, \tau_k) + \xizero{\tau_k} \alpha^2 \|d_k\|^2,
  \end{align*}
  as desired.
\end{proof}

Our next lemma extends the result of Lemma~\ref{lemma.dec_phi_delta_m} by expressing the right-hand side of the inequality in terms of the stationarity measure $\Delta \mtrue_k(\dtrue_k, \tautrue_k)$.

\begin{lemma}\label{lemma.dec_phi_delta_m_2}
  Suppose that Assumptions~\ref{assumption.err0}, \ref{assumption.bounded}, and \ref{assumption.s_bounded}  hold.  Then, for all $k \in \NN$, one finds for all $\alpha \in (0, 1]$ with $\alpha d_k^s \geq -\eta_s$ that
  \begin{align*}
    \phitrue(z_k+\alpha \dhat_k, \tau_k) - \phitrue(z_k, \tau_k) & \leq - \alpha \Delta \mtrue_k(\dtrue_k, \tautrue_k)  - \alpha \Ekmtwo + \xizero{\tau_k} \alpha^2 \|\dtrue_k + \epskd\|^2,
  \end{align*}
  where $\Ekmtwo$ and $\xizero{\tau_k}$ are defined as in Lemmas~\ref{lemma.m2} and \ref{lemma.dec_phi_delta_m}, respectively.
\end{lemma}
\begin{proof}
  The result follows directly from \eqref{eq.errors_k} along with Lemmas~\ref{lemma.m2} and \ref{lemma.dec_phi_delta_m}.
\end{proof}

Following the result of Lemma~\ref{lemma.dec_phi_delta_m_2}, our next aim is to establish, for all $k \in \NN$, a relationship between the noiseless model reduction $\Delta \mtrue_k(\dtrue_k, \tautrue_k)$ and the norm of the noiseless search direction $\dtrue_k$.  This relationship is established by tying both quantities to a combination of $\|\utrue_k\|^2$ and $\|\Jtrue_k^T \ctrue_k\|^2$.

\begin{lemma}\label{lemma.delta_m_1}
  Suppose that Assumptions~\ref{assumption.err0}, \ref{assumption.bounded}, and \ref{assumption.s_bounded} hold.  Then, there exist $\xithree \in (0, \infty)$ and $\xifour \in (0,\infty)$ such that, for all $k \in \NN$, one has
  \begin{align*}
    \Delta \mtrue_k(\dtrue_k, \tautrue_k) &\geq \xithree (\tautrue_k \|\utrue_k\|^2 + \|\Jtrue_k^T \ctrue_k\|^2 ) \\ \text{and}\ \ 
    \|\dtrue_k\|^2 &\leq \xifour (\|\utrue_k\|^2 + \|\Jtrue_k^T \ctrue_k\|^2 ).
  \end{align*}
\end{lemma}
\begin{proof}
  Consider arbitrary $k \in \NN$.  Following a similar argument as in Lemma~\ref{lemma.delta_m_stat_meas}, one has $\Delta \mtrue_k(\dtrue_k, \tautrue_k) \geq \tfrac12 \tautrue_k \utrue_k^T W_k \utrue_k  + \sigma (\|\ctrue_k\| - \|\ctrue_k + \Jtrue_k \vtrue_k\|)$, regardless of whether $x_k$ is an infeasible stationary point or a first-order stationary point for~\eqref{eq.barrier}.  Hence, along with Lemma~\ref{lemma.prelim_cauchy_det}, one has under the given conditions that
  \begin{align*}
    \Delta \mtrue_k(\dtrue_k, \tautrue_k) 
    &\geq \tfrac12 \tautrue_k \utrue_k^T  W_k \utrue_k  + \sigma (\|\ctrue_k\| - \|\ctrue_k + \Jtrue_k \vtrue_k\|) \\
    &\geq  \tfrac12 \tautrue_k \sigma_W  \|\utrue_k\|^2 + \sigma \xione \tfrac{\|\Jtrue_k^T \ctrue_k\|^2}{\|\ctrue_k\|}.
  \end{align*}
  The existence of $\xithree$, as claimed, now follows since $\|\ctrue_k\|$ is bounded under the conditions of the lemma.  On the other hand, by the triangle inequality and the definition of the normal subproblem in~\eqref{eq.normal_tr_step}, one finds that
  \begin{align*}
    \|\dtrue_k\|^2 & \leq  2 (\|\utrue_k\|^2 + \|\vtrue_k\|^2)  \leq  2 (\|\utrue_k\|^2 + \omega^2 \|\Jtrue_k^T \ctrue_k\|^2 ),
  \end{align*}
  from which the desired conclusion holds with $\xifour := 2 \max\{1, \omega^2\}$.
\end{proof}

Following Lemma~\ref{lemma.delta_m_1}, we are able to prove a stronger relationship between $\Delta \mtrue_k(\dtrue_k, \tautrue_k)$ and $\|\dtrue_k\|^2$ if there exist a positive lower bound for the elements of the noiseless merit parameter sequence $\{\tautrue_k\}$.  This may occur more generally, but let us at least establish that this is guaranteed to occur under Assumptions~\ref{assumption.err0}, \ref{assumption.bounded}, \ref{assumption.s_bounded}, and \ref{assumption.sigma_J}.  We do this through our next few lemmas.

\begin{lemma}\label{lemma.v_bnd}
  Suppose that Assumptions~\ref{assumption.err0}, \ref{assumption.bounded}, and \ref{assumption.s_bounded} hold.  Then, for all $k \in \NN$, one has $\|\vtrue_k \| \leq \omega J_{\sup} (c_{I,\sup} + s_{\sup})$.  Moreover, supposing Assumption~\ref{assumption.sigma_J} also holds, there exists $\xitwo \in (0,\infty)$ such that for all $k \in \NN$ with $k \geq k_J$ one finds $$\|\vtrue_k \| \leq  \xitwo (\| \ctrue_k\| - \|\ctrue_k + \Jtrue_k \vtrue_k\|).$$
\end{lemma}
\begin{proof}
  Under the initially stated conditions, the first desired conclusion follows from \eqref{eq.normal_tr_step} and \eqref{eq.J_g_bnd}.  Now suppose Assumption~\ref{assumption.sigma_J} also holds and consider arbitrary $k \in \NN$ with $k \geq k_J$.  Then, $\|\Jtrue_k^T \ctrue_k \| \geq \gamma \|\ctrue_k\|$. Hence, by Lemma~\ref{lemma.prelim_cauchy_det},
  \begin{align*}
    \|\ctrue_k \| - \|\ctrue_k + \Jtrue_k \vtrue_k \| \geq  \xione \tfrac{\|\Jtrue_k^T \ctrue_k\|^2}{\|\ctrue_k\|} 
    \geq    \xione \gamma^2 \|\ctrue_k\|.
  \end{align*}
  Thus, by the conditions of the lemma and \eqref{eq.normal_tr_step}, there exists $\xitwo \in (0,\infty)$ with
  \begin{align*}
    \|\vtrue_k\| \leq \omega \|\Jtrue_k^T\| \|\ctrue_k\| 
    &\leq \xione^{-1} \gamma^{-2} \omega \|\Jtrue_k^T\| (\|\ctrue_k\| - \|\ctrue_k + \Jtrue_k \vtrue_k \|) \\
    &\leq  \xitwo (\|\ctrue_k\| - \|\ctrue_k + \Jtrue_k \vtrue_k\|),
  \end{align*}
  which completes the proof.
\end{proof}

\begin{lemma}\label{lemma.u_bnd}
  Suppose that Assumptions~\ref{assumption.err0}, \ref{assumption.bounded}, and \ref{assumption.s_bounded} hold.  Then, for all $k \in \NN$, one has that $\|\utrue_k \| \leq \sigma_W^{-1} (g_{\sup} + \kappa_W \omega  J_{\sup} (c_{I,\sup} + s_{\sup}) )$.
\end{lemma}
\begin{proof}
  Consider arbitrary $k \in \NN$.  Let $\Ztrue_k \in \RR^{(n+q) \times n}$ be a matrix whose columns form an orthonormal basis for $\mathrm{null}(\Jtrue_k)$. As in the proof of Lemma~\ref{lemma.ub_epsku}, one has $\utrue_k = - \Ztrue_k (\Ztrue_k^T \Gammatrue_k \Ztrue_k )^{-1} \Ztrue_k^T (\gtrue_k + W_k \vtrue_k)$.  Thus, under the conditions of the lemma, one finds with $\|\Ztrue_k\| \leq 1$ and Lemmas~\ref{lemma.equiv_u_lambda} and \ref{lemma.v_bnd} that
  \begin{align*}
    \|\utrue_k \| &\leq \| \Ztrue_k (\Ztrue_k^T \Gamma_k \Ztrue_k)^{-1} \Ztrue_k^T \| \| \gtrue_k + W_k \vtrue_k \| \\
    &=  \| \Ztrue_k (\Ztrue_k^T W_k \Ztrue_k)^{-1} \Ztrue_k^T \| \| \gtrue_k + W_k \vtrue_k \| \\
    &\leq \sigma_W^{-1} (g_{\sup} + \kappa_W \omega J_{\sup} (c_{I,\sup} + s_{\sup}) ),
  \end{align*}
  which completes the proof.
\end{proof}

We now show the aforementioned fact that the noiseless merit parameter values can be guaranteed to be bounded below by a positive real number.  The proof of the following lemma is similar to that of Lemma~3.16 in \cite{curtis2010matrix}.

\begin{lemma}\label{lemma.bounded_tau}
  Suppose that Assumptions~\ref{assumption.err0}, \ref{assumption.bounded}, \ref{assumption.s_bounded}, and \ref{assumption.sigma_J} hold.  Then, there exist $k_\tautrue \in \NN$ and $\tautrue \in (0,\infty)$ such that $\tautrue_k = \tautrue$ for all $k \in \NN$ with $k \geq k_{\tautrue}$.
\end{lemma}
\begin{proof}
  By construction, $\{\tautrue_k\}$ is monotonically nonincreasing.  Thus, by \eqref{eq.tau}, one can reach the desired conclusion by showing that the corresponding sequence of trial merit parameter values is bounded below.  Consider arbitrary $k \in \NN$ with $k \geq k_J$.  By \eqref{eq.perturbed_newton_sys}, it follows that $W_k \dtrue_k  + \gtrue_k + \Jtrue_k^T \ytrue_{k+1} = 0$.  This, along with $\utrue_k \in \mathrm{null}(\Jtrue_k)$, yields $\utrue_k^T W_k \utrue_k + \utrue_k^T W_k \vtrue_k + \utrue_k^T \gtrue_k = 0$. Thus, along with Lemmas~\ref{lemma.v_bnd} and \ref{lemma.u_bnd}, there exists $\xifive \in (0,\infty)$ such that  
  \begin{align*}
    \gtrue_k^T \dtrue_k + \tfrac{1}{2} \utrue_k^T W_k \utrue_k &= \gtrue_k^T \utrue_k + \gtrue_k^T \vtrue_k + \tfrac{1}{2} \utrue_k^T W_k \utrue_k\\
&\leq    - \utrue_k^T W _k \vtrue_k  + \gtrue_k^T \vtrue_k - \tfrac{1}{2} \utrue_k^T W_k \utrue_k\\
    &\leq  - \tfrac{1}{2} \sigma_W \|\utrue_k\|^2 + \|\gtrue_k\|\|\vtrue_k\|  - \utrue_k^T W_k \vtrue_k  \\
    & \leq  \|\gtrue_k\|\|\vtrue_k\| + \|\utrue_k\| \|W_k\|\|\vtrue_k\| \\
    & \leq  \xifive \|\vtrue_k \| \\
    & \leq  \xifive \xitwo (\|\ctrue_k\| - \|\ctrue_k + \Jtrue_k \vtrue_k\|).
  \end{align*}
  Hence, $\tfrac{(1 - \sigma)(\|\ctrue_k\| - \|\ctrue_k + \Jtrue_k \vtrue_k\|)}{\gtrue_k^T \dtrue_k + \tfrac{1}{2}\utrue_k^T W_k \utrue_k} \geq \tfrac{1 - \sigma}{\xifive \xitwo}$, which, as stated, completes the proof.
\end{proof}

Our next lemma now builds upon the result of Lemma~\ref{lemma.delta_m_1}.

\begin{lemma}\label{lemma.d_bounded_by_delta_m}
  Suppose that Assumptions~\ref{assumption.err0}, \ref{assumption.bounded}, \ref{assumption.s_bounded}, and \ref{assumption.sigma_J} hold. Then, there exists $\xi_d \in (0,\infty)$ such that, for all $k \in \NN$, one finds
  \begin{equation*}
    \xi_d \|\dtrue_k\|^2 \leq \Delta \mtrue_k(\dtrue_k, \tautrue_k).
  \end{equation*}
\end{lemma}
\begin{proof}
  Consider arbitrary $k \in \NN$.  By Lemmas~\ref{lemma.delta_m_1} and \ref{lemma.bounded_tau}, one finds that
  \begin{align*}
    \Delta \mtrue_k(\dtrue_k, \tautrue_k)
    &\geq \xithree (\tautrue \|\utrue_k\|^2 + \|\Jtrue_k^T \ctrue_k\|^2) \\
    &\geq \xithree \min\{\tautrue,1\} (\|\utrue_k\|^2 + \|\Jtrue_k^T \ctrue_k\|^2 ) \geq \xi_d \|\dtrue_k \|^2,
  \end{align*}
  where $\xi_d := \xithree \xifour^{-1} \min\{\tautrue,1\}$, which completes the proof.
\end{proof}

We are almost prepared to present our main theorem about the behavior of our proposed algorithm.  The theorem considers two situations: one when only Assumptions~\ref{assumption.err0}, \ref{assumption.bounded}, and \ref{assumption.s_bounded} hold, and another when Assumptions~\ref{assumption.sigma_J} and \ref{ass.noise_level} also hold.  For the latter situation, we provide the following corollary.

\begin{corollary}\label{cor.m1_m2_bound}
  Suppose Assumptions~\ref{assumption.err0}, \ref{assumption.bounded}, \ref{assumption.s_bounded}, \ref{assumption.sigma_J}, and \ref{ass.noise_level} hold.  Then, for all $k \in \NN$,
  \begin{equation}\label{eq.m1_m2_bound}
    \calEkmone \leq \emOneOne + \emOneTwo \|\dtrue_k\|\ \ \text{and} \ \ \calEkmtwo \leq \emTwoOne + \emTwoTwo \|\dtrue_k\|,
  \end{equation}
  where $\calEkmone$ and $\calEkmtwo$ are defined in Lemmas~\ref{lemma.m1} and \ref{lemma.m2}, respectively, and
  \begin{align*}
    \emOneOne := (\tau_0 \eps_g + \eps_J) \eps_d + 2 \eps_c,\ \ 
    \emOneTwo &:= \tau_0 \eps_g + \eps_J, \\
    \emTwoOne := (\tau_0 g_{\sup} + J_{\sup} ) \eps_d,\ \ \text{and} \ \  
    \emTwoTwo &:= \tau_0 g_{\sup},
  \end{align*} 
  with $g_{\sup}$ and $J_{\sup}$ defined in \eqref{eq.J_g_bnd}.
\end{corollary}
\begin{proof}
  Under the stated assumptions, the desired conclusions follow from Lemmas~\ref{lemma.m1} and \ref{lemma.m2}, \eqref{eq.J_g_bnd}, and the facts that $\tautrue_k + \epsktau = \tau_k \leq \tau_0$ and $|\epsktau| \leq \tau_0$.
\end{proof}

Our last lemma prior to our main theorem, stated next, shows situations in which the relaxed line search condition is guaranteed to hold.

\begin{lemma}\label{lemma.lb_on_alpha_1}
  Suppose that Assumptions~\ref{assumption.err0}, \ref{assumption.bounded}, and \ref{assumption.s_bounded} hold, and for all $k \in \NN$ with $\alpha_k^{\max}$ computed to satisfy the fraction-to-the-boundary rule, $\epskd$ defined in~\eqref{eq.errors_k}, and $\xizero{\tau_k}$ defined in Lemma~\ref{lemma.dec_phi_delta_m}, let
  \begin{align}\label{eq.hat_alpha_k}
    \alphahat_k = \min \left \{\frac{(1 - \eta_\phi) \Delta \mtrue_k(\dtrue_k, \tautrue_k)}{2 \xizero{\tau_k} \|\dtrue_k + \epskd\|^2}, \alpha_k^{\max} \right\}.
  \end{align}
  Then, with $(\Ekmone, \Ekmtwo)$ from Lemmas~\ref{lemma.m1}--\ref{lemma.m2}, one has for any $k \in \NN$ that
  \begin{align}\label{eq.err_cond_0}
    \eta_\phi \Ekmone - (1 - \eta_\phi) \Ekmtwo \leq \tfrac{1}{2} (1 - \eta_\phi) \Delta \mtrue_k(\dtrue_k, \tautrue_k)
  \end{align}
  implies that \eqref{eq.armijo} holds for all $\alpha \in [0,\alphahat_k]$.  Now suppose that, in addition, Assumptions~\ref{assumption.sigma_J} and \ref{ass.noise_level} hold, and let $k_{\tautrue}$ be defined in Lemma~\ref{lemma.bounded_tau}, $\xi_d$ be defined in Lemma~\ref{lemma.d_bounded_by_delta_m}, $(\emOneOne, \emOneTwo, \emTwoOne, \emTwoTwo)$ be defined in Corollary~\ref{cor.m1_m2_bound}, and
  \begin{align}\label{eq.bar_alpha_k}
    \alphabar_k := \min \left \{ \frac{(1 - \eta_\phi) {\xi_d} \Delta \mtrue_k(\dtrue_k, \tautrue_k)}{ 4 \xizero{\tau_k} (\Delta \mtrue_k(\dtrue_k, \tautrue_k) + \xi_d\eps_d^2) }, \alpha_k^{\max} \right\}\ \ \text{for all}\ \ k \geq k_{\tautrue}.
  \end{align}
  Then, one has for any $k \geq k_{\tautrue}$ that
  \begin{align}
    &\ (\eta_{\phi} \emOneOne + (1 - \eta_{\phi}) \emTwoOne) + \xi_d^{-1/2} (\eta_{\phi} \emOneTwo + (1 - \eta_{\phi}) \emTwoTwo ) {\Delta \mtrue_k(\dtrue_k, \tautrue_k) }^{1/2} \nonumber \\
    \leq&\ \tfrac{1}{2} (1 - \eta_\phi) \Delta \mtrue_k(\dtrue_k, \tautrue_k) \label{eq.err_cond_rho1}
 \end{align}
  implies that \eqref{eq.armijo} holds for all $\alpha \in [0,\alphabar_k]$.
\end{lemma}
\begin{proof}
  Consider arbitrary $k \in \NN$.  Our first aim is to show that if \eqref{eq.err_cond_0} holds, then \eqref{eq.armijo} holds for all $\alpha \in [0,\alphahat_k]$.  Toward this end, observe that by \eqref{eq.errors2} and Lemma~\ref{lemma.dec_phi_delta_m_2}, one finds for all $\alpha \in (0, \alpha_k^{\max}]$ that
  \begin{align}
    &\ \phi(z_k+\alpha \dhat_k, \tau_k)  - \phi(z_k, \tau_k) \nonumber \\
    \leq&\ \phitrue(z_k+\alpha \dhat_k, \tau_k) -  \phitrue(z_k, \tau_k) + 2 \vareps_k \nonumber \\  
    \leq&\ -\alpha \Delta \mtrue_k(\dtrue_k, \tautrue_k) - \alpha \Ekmtwo + \alpha^2 \xizero{\tau_k} \|\dtrue_k + \epskd\|^2 + 2 \vareps_k \nonumber \\
    =&\ -\alpha \eta_\phi \Delta m_k( d_k, \tau_k) + 2 \vareps_k - \alpha (1-\eta_\phi) \Delta \mtrue_k(\dtrue_k, \tautrue_k) \nonumber \\
    &+ \alpha \eta_\phi \Ekmone - \alpha (1 - \eta_\phi) \Ekmtwo + \alpha^2 \xizero{\tau_k} \|\dtrue_k + \epskd\|^2. \label{eq.tmp_alpha_armijo_0}
  \end{align}
  Hence, along with \eqref{eq.hat_alpha_k} and \eqref{eq.err_cond_0}, one finds that
  \begin{align*}
    &\ \phi(z_k+\alpha \dhat_k, \tau_k) - \phi(z_k, \tau_k) \\
    \leq&\ -\alpha \eta_\phi \Delta m_k( d_k, \tau_k) + 2 \eps_k - \tfrac12 \alpha (1 - \eta_\phi) \Delta \mtrue_k(\dtrue_k, \tautrue_k) + \alpha^2 \xizero{\tau_k} \|\dtrue_k + \epskd\|^2 \\
    \leq&\ -\alpha \eta_\phi \Delta m_k( d_k, \tau_k) + 2 \vareps_k,
  \end{align*}
  which implies that \eqref{eq.armijo} is satisfied, as desired.
  
  Now let us prove the second part of the lemma with the addition of Assumptions~\ref{assumption.sigma_J} and \ref{ass.noise_level}.  Consider arbitrary $k \in \NN$ with $k \geq k_{\tautrue}$.  By \eqref{eq.tmp_alpha_armijo_0}, Lemmas~\ref{lemma.m1} and \ref{lemma.m2}, and Corollary~\ref{cor.m1_m2_bound}, one finds for all $\alpha \in (0, \alpha_k^{\max}]$ that
  \begin{align*}
    &\ \phi(z_k+\alpha \dhat_k, \tau_k)  - \phi(z_k, \tau_k) \\
    \leq&\ -\alpha \eta_\phi \Delta m_k( d_k, \tau_k) + 2 \vareps_k - \alpha (1-\eta_\phi) \Delta \mtrue_k(\dtrue_k, \tautrue_k) \\
    &\ + \alpha \eta_\phi \mathcal{E}_k^{m, 1} + \alpha(1 - \eta_\phi) \mathcal{E}_k^{m, 2} + 2\alpha^2 \xizero{\tau_k} (\|\dtrue_k \|^2+ \eps_d^2) \\
    \leq&\ -\alpha \eta_\phi \Delta m_k( d_k, \tau_k) + 2 \vareps_k - \alpha (1-\eta_\phi) \Delta \mtrue_k(\dtrue_k, \tautrue_k) \\
    &\ + \alpha \eta_\phi (\emOneOne + \emOneTwo \|\dtrue_k\|) \\
    &\ + \alpha(1 - \eta_\phi) (\emTwoOne + \emTwoTwo \|\dtrue_k\|) + 2\alpha^2 \xizero{\tau_k} (\|\dtrue_k \|^2+ \eps_d^2).
  \end{align*}
  Now by Lemmas~\ref{lemma.bounded_tau} and Lemma~\ref{lemma.d_bounded_by_delta_m}, one has 
  \begin{align*}
    &\ \phi(z_k+\alpha \dhat_k, \tau_k)  - \phi(z_k, \tau_k) \\
    \leq&\ -\alpha \eta_\phi \Delta m_k( d_k, \tau_k) + 2 \vareps_k  -\alpha (1-\eta_\phi) \Delta \mtrue_k(\dtrue_k, \tautrue_k)  \\ 
    &\ + \alpha (\eta_{\phi} \emOneOne + (1 - \eta_{\phi}) \emTwoOne ) \\
    &\ + \alpha \xi_d^{-1/2} (\eta_{\phi} \emOneTwo + (1 - \eta_{\phi}) \emTwoTwo ) \Delta \mtrue_k(\dtrue_k, \tautrue_k)^{1/2} \\
    &\ + 2 \alpha^2 \xizero{\tau_k} \xi_d^{-1} \Delta \mtrue_k(\dtrue_k, \tautrue_k) +  2 \alpha^2 \xizero{\tau_k} \eps_d^2. 
  \end{align*}
  Hence, by \eqref{eq.bar_alpha_k} and \eqref{eq.err_cond_rho1}, one finds that
  \begin{align*}
    &\ \phi(z_k + \alpha \dhat_k, \tau_k) - \phi(z_k, \tau_k) \\
    \leq&\ - \alpha \eta_\phi \Delta m_k( d_k, \tau_k) + 2 \vareps_k - \tfrac{1}{2} \alpha  (1 - \eta_\phi) \Delta \mtrue_k(\dtrue_k, \tautrue_k) \\
    &\ + 2 \alpha^2 \xizero{\tau_k} \xi_d^{-1} \Delta \mtrue_k(\dtrue_k, \tautrue_k) +  2 \alpha^2 \xizero{\tau_k} \eps_d^2 \\
    \leq&\ - \alpha \eta_\phi \Delta m_k( d_k, \tau_k) + 2 \vareps_k,
  \end{align*}
  which implies that \eqref{eq.armijo} is satisfied, as desired.
\end{proof}

We are now ready to prove our main theorem.  We following the proof of the theorem with commentary on the theorem's consequences.

\begin{theorem}\label{thm.main1}
  Suppose Assumptions~\ref{assumption.err0}, \ref{assumption.bounded}, and \ref{assumption.s_bounded} hold.  Then, with $\Ekmone$ defined in Lemma~\ref{lemma.m1}, $\Ekmtwo$ defined in Lemma~\ref{lemma.m2}, $\Ekm$ defined in Corollary~\ref{cor.m}, and $\alphahat_k$ defined in Lemma~\ref{lemma.lb_on_alpha_1}, there exists $k \in \NN$ such that
  \begin{equation}\label{eq.stat_cond_0}
    \Delta \mtrue_k(\dtrue_k, \tautrue_k) \leq \max\left\{\frac{2(\eta_\phi \Ekmone - (1 - \eta_\phi) \Ekmtwo)}{(1 - \eta_\phi)}, - 2 \Ekm, \frac{4(4+2\zeta) \vareps_k}{\alphahat_k \eta_\phi }  \right\}.
  \end{equation}
  Now suppose that, in addition, Assumptions~\ref{assumption.sigma_J} and \ref{ass.noise_level} hold, and let $k_{\tautrue}$ be defined in Lemma~\ref{lemma.bounded_tau}, $\xi_d$ be defined in Lemma~\ref{lemma.d_bounded_by_delta_m}, and $(\emOneOne, \emOneTwo, \emTwoOne, \emTwoTwo)$ be defined in Corollary~\ref{cor.m1_m2_bound}.  Then, there exists $k \in \NN$ with $k \geq k_{\tautrue}$ such that
  \begin{equation}\label{eq.stat_cond_0_v3}
    \Delta \mtrue_k(\dtrue_k, \tautrue) \leq \max\{\rho_{k,1}, \rho_{k,2}, \rho_{k,3} \},  
  \end{equation}
  where
  \begin{align*}
    \rho_{k,1} &= 2 \left( \tfrac{\eta_{\phi}}{1 - \eta_{\phi}} \emOneOne +  \emTwoOne + \xi_d^{-1/2} (\tfrac{\eta_{\phi}}{1 - \eta_{\phi}} \emOneTwo + \emTwoTwo )  {\Delta \mtrue_k(\dtrue_k, \tautrue) }^{1/2} \right), \\
    \rho_{k,2} &= 2 \left( \emOneOne + \emTwoOne + \xi_d^{-1/2} (\emOneTwo + \emTwoTwo) \Delta  \mtrue_k(\dtrue_k, \tautrue)^{1/2} \right), \\ \text{and}\ \ 
    \rho_{k,3} &= \tfrac{4(4 + 2\zeta) \vareps_k}{\alphabar_k \eta_\phi} = \tfrac{4(4 + 2\zeta) \vareps_k}{\eta_\phi} \max \left \{ \tfrac{4 \xizero{\tautrue} (\Delta \mtrue_k(\dtrue_k, \tautrue) + \xi_d\eps_d^2)}{(1 - \eta_\phi) \xi_d \Delta \mtrue_k(\dtrue_k, \tautrue)}, \tfrac{1}{\alpha_k^{\max}} \right\}.     
  \end{align*}
\end{theorem}
\begin{proof}
  To derive a contradiction to the first desired conclusion, suppose~\eqref{eq.stat_cond_0} does not hold for all $k \in \NN$.  Under this supposition, consider arbitrary $k \in \NN$.  Since \eqref{eq.stat_cond_0} does not hold, it follows that \eqref{eq.err_cond_0} holds, which in turn means by Lemma~\ref{lemma.lb_on_alpha_1} that \eqref{eq.armijo} holds for some step size $\alpha \geq \hat\alpha_k/2$.  Thus, with Corollary~\ref{cor.m},
  \begin{align}
    &\ \phitrue(z_k + \alpha \dhat_k, \tau_k) -  \phitrue(z_k, \tau_k) \nonumber \\
    \leq&\ \phi(z_k + \alpha \dhat_k, \tau_k) - \phi(z_k, \tau_k) + 2 \vareps_k \nonumber \\
    \leq&\ -\tfrac{1}{2} \alphahat_k \eta_\phi \Delta m_k(d_k, \tau_k) + 2 \vareps_k + (2 + \zeta) \vareps_k \nonumber \\
    =&\ -\tfrac{1}{2} \alphahat_k \eta_\phi \Delta \mtrue_k(\dtrue_k, \tautrue_k) - \tfrac{1}{2} \alphahat_k \eta_\phi \Ekm + (4 + \zeta) \vareps_k. \label{eq.tmp_thm_delta_m_0}
  \end{align}
  This, again with the fact that \eqref{eq.stat_cond_0} does not hold, implies
  \begin{equation*}
    \phitrue(z_k + \alpha \dhat_k, \tau_k) -  \phitrue(z_k, \tau_k) \leq - \tfrac14 \hat\alpha_k \eta_\phi \Delta  \mtrue_k(\dtrue_k, \tautrue_k) + (4 + \zeta) \vareps_k \leq -\zeta \vareps_k.
  \end{equation*}
  However, with Remark~\ref{remark.slackreset} this implies that $\phitrue(z_{k+1}, \tau_k) -  \phitrue(z_k, \tau_k) \leq - \zeta \vareps_k$, which contradicts Assumption~\ref{assumption.s_bounded}, specifically the fact that $\{\phitrue(z_k,\tau_k)\}$ is bounded below.  Consequently, \eqref{eq.stat_cond_0} must hold for some $k \in \NN$, as desired.

  Now consider the second desired conclusion under the addition of Assumptions~\ref{assumption.sigma_J} and \ref{ass.noise_level}.  To derive a contradiction, suppose that \eqref{eq.stat_cond_0_v3} does not hold for all $k \in \NN$ with $k \geq k_{\tautrue}$.  Under this supposition, consider arbitrary such $k$.  Since \eqref{eq.stat_cond_0_v3} does not hold, it follows that \eqref{eq.err_cond_rho1} holds, which in turn means by Lemma~\ref{lemma.lb_on_alpha_1} that \eqref{eq.armijo} holds for some step size $\alpha \geq \bar\alpha_k/2$.  Thus, by a similar derivation that led to \eqref{eq.tmp_thm_delta_m_0}, Corollary~\ref{cor.m}, Corollary~\ref{cor.m1_m2_bound}, and Lemma~\ref{lemma.d_bounded_by_delta_m},
  \begin{align*}
    &\ \phitrue(z_k + \alpha \dhat_k, \tau_k) -  \phitrue(z_k, \tau_k) \\
    \leq&\ -\tfrac{1}{2} \alphabar_k \eta_\phi \Delta \mtrue_k(\dtrue_k, \tautrue_k) + \tfrac{1}{2} \alphabar_k \eta_\phi \calEkm + (4 + \zeta) \vareps_k \\
    \leq&\ -\tfrac{1}{2} \alphabar_k \eta_\phi \Delta \mtrue_k(\dtrue_k, \tautrue_k) + (4 + \zeta) \vareps_k \\
    &\ + \tfrac{1}{2} \alphabar_k \eta_\phi (\emOneOne + \emTwoOne + (\emOneTwo + \emTwoTwo) \|\dtrue_k\| ) \\
    \leq&\ -\tfrac{1}{2} \alphabar_k \eta_\phi \Delta \mtrue_k(\dtrue_k, \tautrue_k) + (4 + \zeta) \vareps_k \\
    &\ + \tfrac{1}{2} \alphabar_k \eta_\phi (\emOneOne + \emTwoOne + \xi_d^{-1/2} (\emOneTwo + \emTwoTwo) \Delta  \mtrue_k(\dtrue_k, \tautrue_k)^{1/2} ).
  \end{align*}
  This, again with the fact that \eqref{eq.stat_cond_0_v3} does not hold, implies
  \begin{equation*}
    \phitrue(z_k + \alpha \dhat_k, \tau_k) -  \phitrue(z_k, \tau_k) \leq - \tfrac14 \alphabar_k \eta_\phi \Delta  \mtrue_k(\dtrue_k, \tautrue_k) + (4 + \zeta) \vareps_k \leq -\zeta \vareps_k
  \end{equation*}
  However, with Remark~\ref{remark.slackreset} this implies that $\phitrue(z_{k+1}, \tau_k) - \phitrue(z_k, \tau_k) < - \zeta \vareps_k$, which contradicts Assumption~\ref{assumption.s_bounded}, specifically the fact that $\{\phitrue(z_k,\tau_k)\}$ is bounded below.  Consequently, \eqref{eq.stat_cond_0_v3} must hold for some $k \in \NN$, as desired.
\end{proof}

The following observations explain the strong consequences of Theorem~\ref{thm.main1}.

\begin{itemize}
  \item The first conclusion of the theorem, expressed in~\eqref{eq.stat_cond_0}, shows a preliminary upper bound on the stationarity measure $\Delta \mtrue(\dtrue_k, \tautrue_k)$ in terms of error quantities in the given iteration $k \in \NN$.  If $\Ekmone$, $\Ekmtwo$, and $\Ekm$ are each relatively small and the step size threshold $\alphahat_k$ (recall \eqref{eq.hat_alpha_k}) is not exceedingly small, then \eqref{eq.stat_cond_0} describes that the algorithm has reached an iteration in which $\Delta \mtrue(\dtrue_k, \tautrue_k)$ is small, as is desirable; recall \S\ref{sec.discussion}.  One should of course note that $\alphahat_k$ itself depends on $\Delta \mtrue(\dtrue_k, \tautrue_k)$, $\|\dtrue_k + \epskd\|^2$, and $\alpha_k^{\max}$, and in particular $\alphahat_k$ not being exceedingly small means that $\|\dtrue_k + \epskd\|^2$ is not exceedingly large relative to $\Delta \mtrue(\dtrue_k, \tautrue_k)$ and that $\alpha_k^{\max}$ is not exceedingly small. It is not straightforward to guarantee that such a situation occurs without considering the overall behavior of the algorithm over the entire sequence of iterations, rather than in the particular iteration $k \in \NN$ in which \eqref{eq.stat_cond_0} holds.  Therefore, let us turn to the second desired conclusion wherein we have effectively shown an iteration-independent bound on the threshold reached by the stationary measure $\Delta \mtrue(\dtrue_k, \tautrue_k)$.
  \item In the second conclusion of the theorem under the additional Assumptions~\ref{assumption.sigma_J} and \ref{ass.noise_level}, the bound $\Delta \mtrue(\dtrue_k, \tautrue_k) \leq \max\{\rho_{k,1}, \rho_{k,2}\}$ would imply that
  \begin{equation*}
    \Delta \mtrue(\dtrue_k, \tautrue_k) \leq c_1 + c_2 \Delta \mtrue(\dtrue_k, \tautrue_k)^{1/2}
  \end{equation*}
  for some iteration-independent constants $c_1 \in (0,\infty)$ and $c_2 \in (0,\infty)$, which in turn would bound $\Delta \mtrue(\dtrue_k, \tautrue_k)^{1/2}$ in terms of these constants.  Recalling Corollary~\ref{cor.m1_m2_bound}, one finds that these constants would all vanish if the noise bounds were to vanish, except for the term derived from $\emTwoTwo := \tau_0 g_{\sup}$, which depends on the initial merit parameter value.  This directs us back to Remark~\ref{rem.ugly_but_it_is_what_it_is}, where we first observed that poor behavior can be exhibited by the algorithm if the noisy and noiseless merit parameter sequences diverge substantially.  Specifically, the worst case is when the noisy merit parameter values does not decrease sufficiently, in which case the algorithm does not put enough weight on reducing constraint violation.  One way in which this poor behavior can be mitigated is by choosing $\tau_0$ to a conservative, small value, although this might not always yield the best performance in real-world practice.
  \item In the second conclusion, the bound $\Delta \mtrue(\dtrue_k, \tautrue_k) \leq \rho_{k,3}$ would imply that
  \begin{equation*}
    \Delta \mtrue(\dtrue_k, \tautrue_k) \leq \frac{4(4 + 2\zeta) \vareps_k}{\eta_\phi} \max \left \{ \frac{4 \xizero{\tautrue} (\Delta \mtrue_k(\dtrue_k, \tautrue) + \xi_d\eps_d^2)}{(1 - \eta_\phi) \xi_d \Delta \mtrue_k(\dtrue_k, \tautrue)}, \frac{1}{\alpha_k^{\max}} \right\}.
  \end{equation*}
  Recall that by Corollary~\ref{lemma.ub_epskd} it follows that $\|d_k - \dtrue_k\| \leq \eps_d$.  One can also show under our assumptions that $\{d_k\}$ is bounded in norm; see, e.g., Lemma 3.7 in~\cite{CurtScheWaec10}.  Consequently, following Lemma 3.11 in \cite{CurtScheWaec10}, one can show that there exists a constant $d_{\sup}^s$ such that $(\alpha_k^{\max})^{-1} \leq \eta_s^{-1} \|d_k^s\| \leq \eta_s^{-1} d_{\sup}^s$ for all $k \in \NN$.  In fact, this is quite a pessimistic bound, since in fact one should expect $\|d_k^s\|$ to be small at points at which $\Delta \mtrue(\dtrue_k, \tautrue_k)$ is small.  In any case, overall one can essentially conclude that $\Delta \mtrue(\dtrue_k, \tautrue_k) \leq \rho_{k,3}$ implies that $\Delta \mtrue(\dtrue_k, \tautrue_k)$ is bounded by constants that vanish with the noise.
\end{itemize}

\section{Numerical Experiments}\label{sec.numerical_exp}

The purpose of our numerical experiments is to demonstrate the practical performance of our proposed algorithm on a large and diverse set of test problems.  Our aim is to show that, in practice, our algorithm is able to drive stationarity measures below levels that are proportional to the noise in the function and derivative values.  Similarly as for our theoretical results in the preceding section, our experiments consider runs with a fixed barrier parameter.  We discuss a proposed heuristic for reducing the barrier parameter in Section~\ref{sec.conclusion_4}.

We implemented our algorithm in Matlab.  We chose, in the following manner, test problems from CUTE (Constrained and Unconstrained Testing Environment) \cite{bongartz1995cute} for which AMPL models \cite{gay1997hooking} were available.\footnote{\href{https://vanderbei.princeton.edu/ampl/nlmodels/index.html}{https://vanderbei.princeton.edu/ampl/nlmodels/index.html}}  First, of all available CUTE/AMPL problems, we restricted attention to those involving at least one (potentially nonlinear) constraint given by CUTE/AMPL in their standard form \texttt{l <= c(x) <= u}.  (Problems with only bound constraints on an individual variable or variables were ignored.)  Second, of these problems, we restricted attention to those with \texttt{l < u}, componentwise.  Third, we removed from consideration all problems with more than 100 primal or dual variables.  This resulted in a set of 137 test problems for our numerical experiments.

The following initialization strategy and algorithmic parameters were fixed for all runs.  (Values not mentioned here were chosen differently for different experiments, as discussed later on.)  For the initial primal iterate $x_0$, we took the initial point offered by CUTE/AMPL and projected it inside any present bound constraints using the strategy proposed in Section~3.6 of~\cite{wachter2006implementation}.  The initial slack variable $s_0$ was set to a vector of all ones, then a slack reset was performed; recall \eqref{eq.slack_reset}. The initial merit parameter was set to $\tau_{-1} \gets 10^{-1}$, the normal-step trust-region-radius parameter was set to $\omega \gets 10^3$, the merit parameter update parameters were set to $\sigma \gets 10^{-1}$ and $\delta_{\tau} \gets 10^{-4}$, and the relaxed Armijo line-search parameters were set to $\eta_{\phi} \gets 10^{-8}$ and $\zeta \gets 10^{-1}$.

Important implementation details were as follows.  First, as is typical for software for solving continuous optimization problems, each problem was scaled before our algorithm was applied.  Specifically, if the noiseless gradient of the objective function evaluated at the initial point had an $\ell_\infty$ norm greater than~10, then the objective function was scaled by $10/\|\gtrue_0\|_\infty$.  Similarly, if the noiseless gradient of any constraint function evaluated at the initial point had an $\ell_\infty$ norm greater than 10, then the constraint function was scaled similarly.  All of our subsequent numerical results refer to solving the resulting scaled problems.  Second, during the run of our proposed algorithm to compute the normal step $v_k$ for all $k \in \NN$, our code employs an implementation of the well-known Mor\'e-Sorensen method~\cite{more1983computing} for solving the trust-region subproblem~\eqref{eq.normal_tr_step}.  Third, for all $k \in \NN$, to compute the tangential step $u_k$, we initialize $W_k$ as in \eqref{eq.W} with a noisy Hessian value (see next paragraph) and $\Sigma_k \gets S_kY_k$.  Then, a Hessian modification subroutine is performed to ensure that it is sufficiently positive definite over $\mathrm{null}(J_k)$.  This is done by using Matlab's built-in \texttt{ldl} method to compute a block LDL factorization of the matrix in \eqref{eq.perturbed_newton_sys}, then checking its inertia \cite{nocedal1999numerical}.  If the block diagonal matrix in this factorization has $n$ positive eigenvalues with value at least $10^{-10}$, then no modification is performed.  Otherwise, sequentially larger multiples of the identity matrix are added to $W_k$ until the block diagonal matrix in its LDL factorization has $n$ positive eigenvalues with value at least $10^{-10}$.  Once the Hessian modification strategy is completed, the tangential step $u_k$ and multipliers $y_{k+1}$ are computed using Matlab's built-in \texttt{mldivide} routine for solving linear systems of equations.

We present the results of experiments with two barrier parameter values and two noise levels.  For the barrier parameter, we considered $\mu \gets 10^{-1}$ and $\mu \gets 10^{-4}$.  In each case, the initial dual variable $y_0$ was set to $\mu S_0^{-1} e$ and the fraction-to-the-boundary parameter was set to $\eta_s \gets \max \{0.99, 1 - \mu \}$.  For the noise levels, we considered $\eps_f = \eps_c \gets 10^{-2}$ and $\eps_f = \eps_c \gets 10^{-6}$.  In each case, we chose $\eps_g = \eps_J = \eps_H \gets \sqrt{\eps_f} = \sqrt{\eps_c}$.  Given these constants, noise was added to each CUTE/AMPL function and derivative value for each $k \in \NN$ as follows.  For the objective and constraint functions, we set $f_0(x_k) \gets \fbar_0(x_k) + \vareps_{f_0,k}$ and $c_I(x_k) \gets \ctrue_I(x_k) + \vareps_{c_I,k}$, where $\vareps_{f_0,k}$ was drawn from a uniform distribution over $[-\eps_f,\eps_f]$ and $\vareps_{c_I,k}$ was drawn from a uniform distribution over an $\ell_2$-norm ball with radius $\eps_c$.  For the objective and constraint first-order derivatives, we set $g_0(x_k) \gets \gtrue_0(x_k) + \vareps_{g_0,k}$ and $J_I(x_k) \gets \Jtrue(x_k) + \vareps_{J_I,k}$, where $\vareps_{g_0,k}$ was drawn from a uniform distribution over an $\ell_2$-norm ball with radius $\eps_g$, $\vareps_{J_I,k}$ had the same sparsity pattern as $\Jtrue(x_k)$, and the nonzero components of each row of $\vareps_{J_I,k}$ were drawn from a uniform distribution over an $\ell_2$-norm ball with radius $\eps_J/\sqrt{q}$.  Finally, with $\Htrue_k$ defined as the Hessian of the Lagrangian evaluated at $(x_k,y_k)$, we set $H_k \gets \Htrue_k + \vareps_{H,k}$, where $\vareps_{H,k}$ is a diagonal matrix with each diagonal element drawn from a uniform distribution over $[-\eps_H,\eps_H]$.

For each test problem, barrier parameter, and noise level, we ran our code until 2000 iterations were conducted or a time limit of 1 hour was reached.  In actual practice, it would be more reasonable to run Algorithm~\ref{alg.exact_alg} with termination conditions corresponding to those stated in the algorithm itself, i.e., for prescribed thresholds, the algorithm may terminate when---according to noisy quantities---an approximate infeasible stationary point or approximate first-order stationary point has been reached.  However, for our purposes here, we chose to disable these termination conditions in order to be confident that the algorithm has reached the highest quality solution that it can in each run.

In Figure~\ref{fig:mu_1e-1_noisy}, we show the quality of the solutions found with $\mu = 10^{-1}$ in terms of \emph{noisy} stationarity measures, both with respect to $\eps_f = 10^{-2}$ and $\eps_f = 10^{-6}$.  (Recall from above that the value for $\eps_f$ determines all of the values for $\eps_c$, $\eps_g$, $\eps_J$, and $\eps_H$ as well.)  Specifically, for varying thresholds on the stationarity measure for solving the barrier subproblem (left) and a stationarity measure for minimizing infeasibility (right), we show the percentages of problems for which the geometric average of the last 10 iterates satisfied the given threshold for each measure.  Figure~\ref{fig:mu_1e-1_det} provides the results for these same runs of the algorithm (using noisy quantities), but considers values of the measures with respect to \emph{noiseless} function and derivative values.  These are included to show that the algorithm is reaching points that can be considered nearly stationary for noiseless values, even though it only has access to noisy ones.

\begin{figure}[ht]
\centering
\begin{subfigure}{.5\textwidth}
  \centering
  \includegraphics[width=\linewidth]{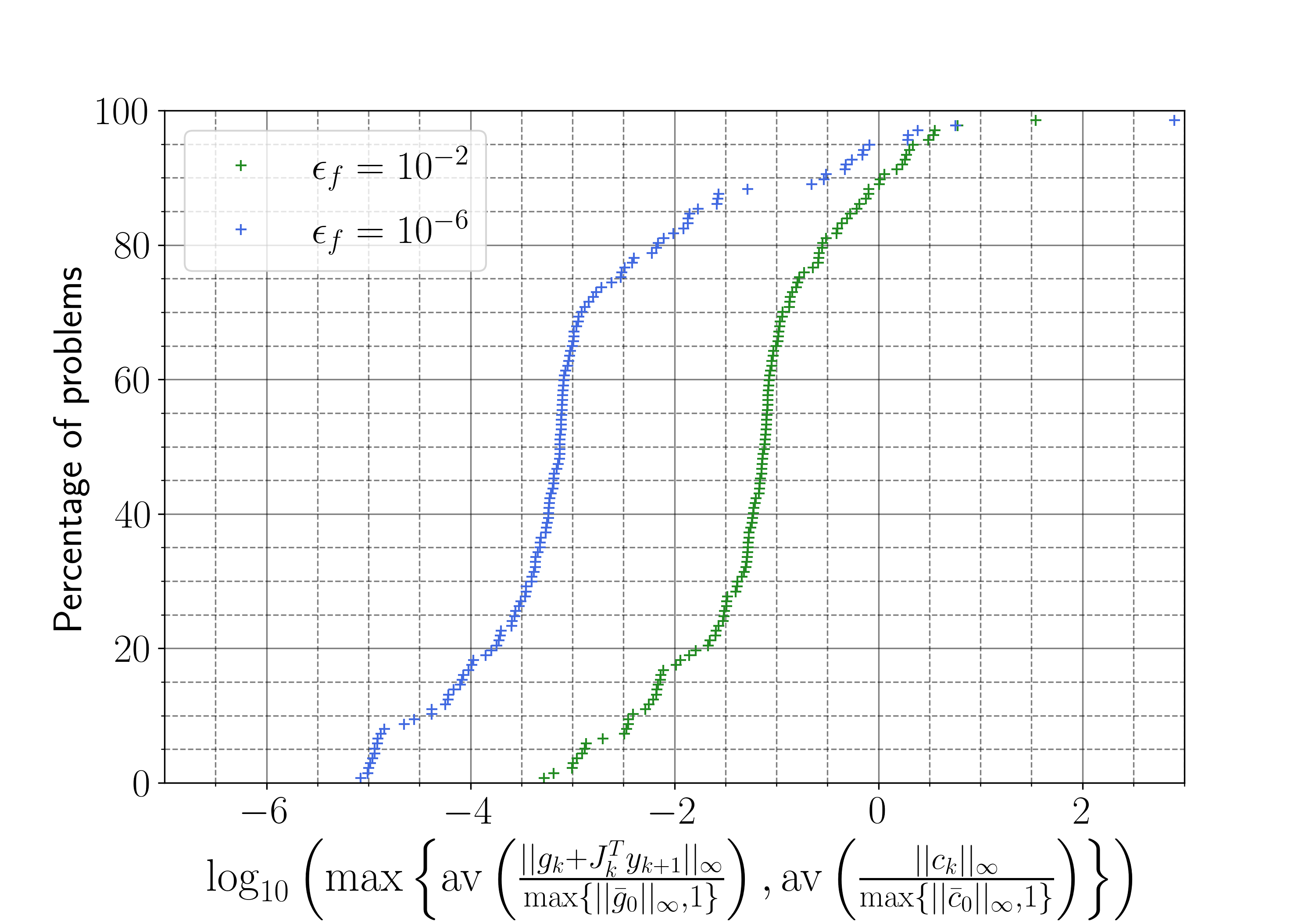}
  \caption{Stationarity w.r.t.~barrier subproblem}
  \label{fig:opt_mu_1e-1_noisy}
\end{subfigure}%
\begin{subfigure}{.5\textwidth}
  \centering
  \includegraphics[width=\linewidth]{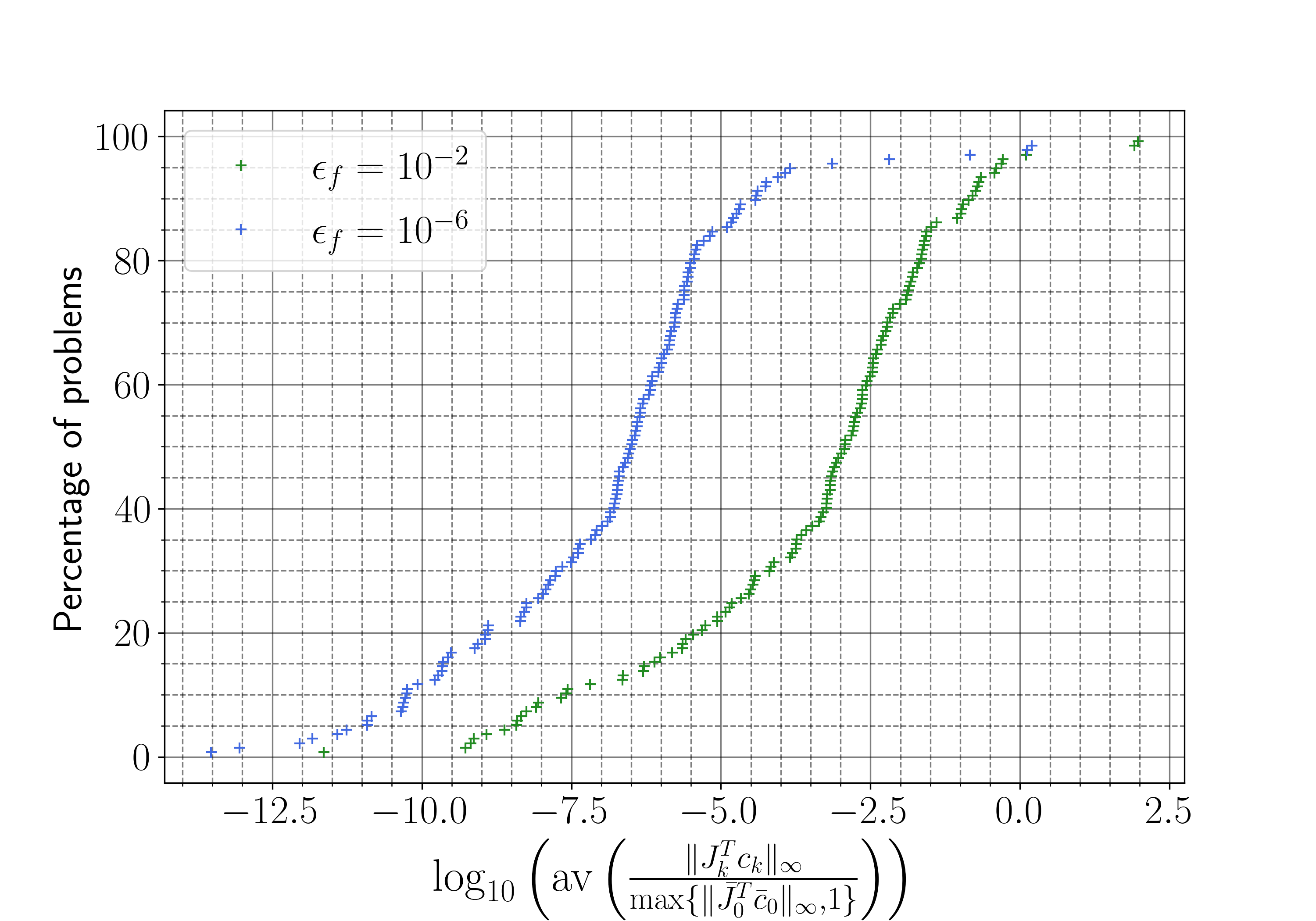}
  \caption{Stationarity w.r.t.~minimizing infeasibility}
  \label{fig:infeas_mu_1e-1_noisy}
\end{subfigure}
\caption{$\mu=10^{-1}$, measured with noisy function and derivative values}
\label{fig:mu_1e-1_noisy}
\end{figure} 

\begin{figure}[ht]
\centering
\begin{subfigure}{.5\textwidth}
  \centering
  \includegraphics[width=\linewidth]{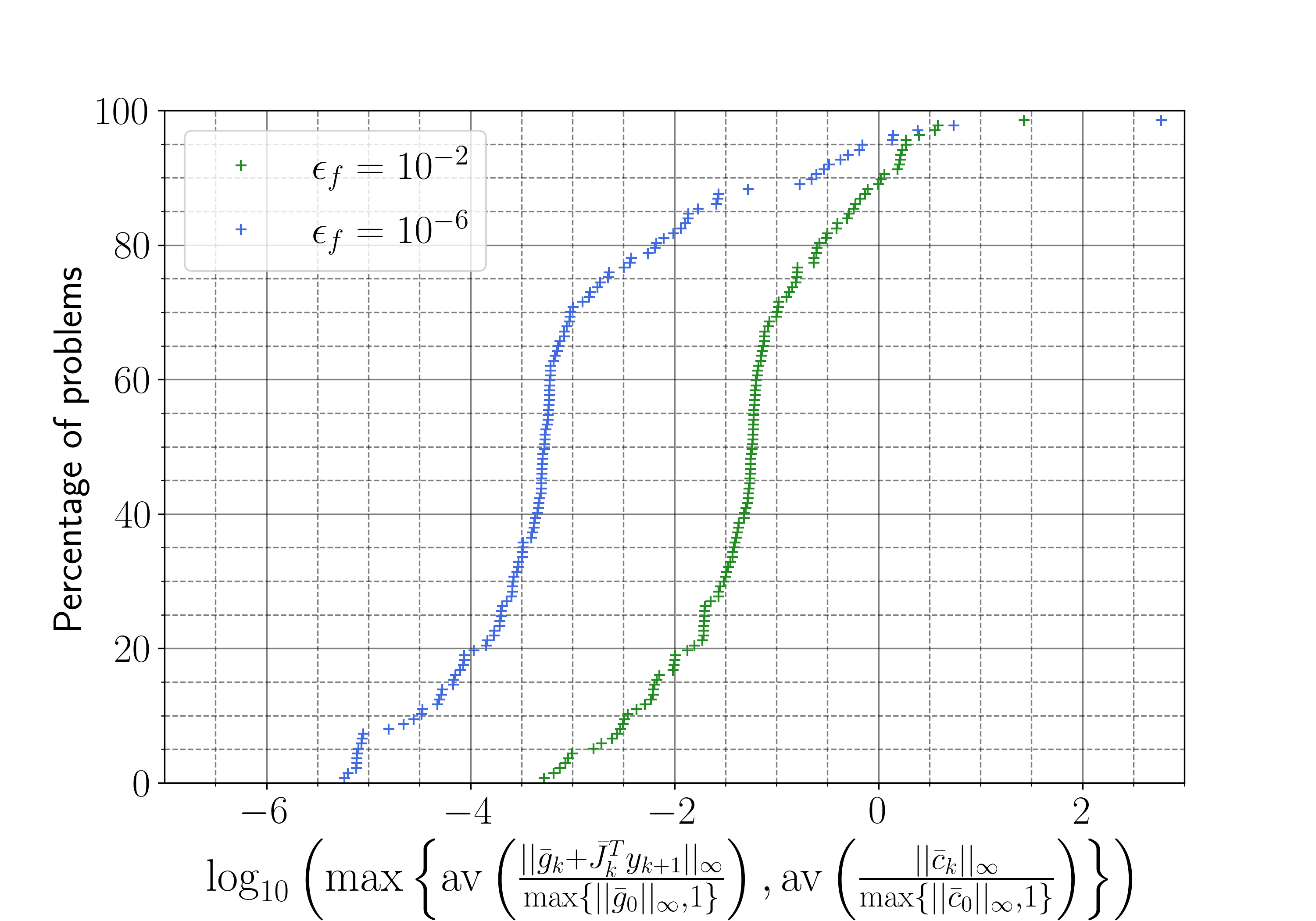}
  \caption{Stationarity w.r.t.~barrier subproblem}
  \label{fig:opt_mu_1e-1_det}
\end{subfigure}%
\begin{subfigure}{.5\textwidth}
  \centering
  \includegraphics[width=\linewidth]{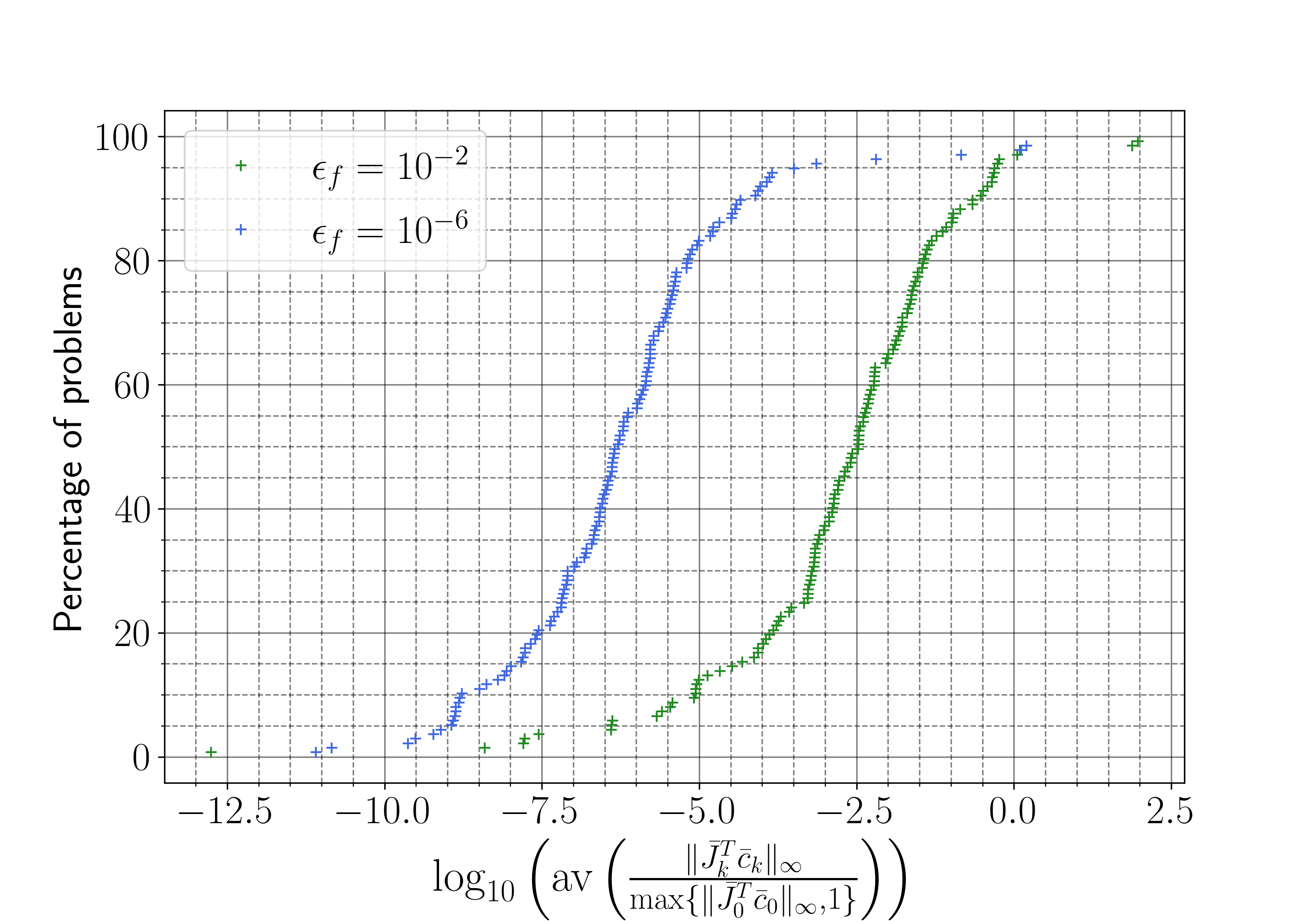}
  \caption{Stationarity w.r.t.~minimizing infeasibility}
  \label{fig:infeas_mu_1e-1_det}
\end{subfigure}
\caption{$\mu=10^{-1}$, measured with noiseless function and derivative values}
\label{fig:mu_1e-1_det}
\end{figure}

Figures~\ref{fig:mu_1e-4_noisy} and \ref{fig:mu_1e-4_det}, respectively, show results in terms of \emph{noisy} and \emph{noiseless} stationarity measures, this time when the barrier parameter is $\mu = 10^{-4}$.

\begin{figure}[ht]
\centering
\begin{subfigure}{.5\textwidth}
  \centering
  \includegraphics[width=\linewidth]{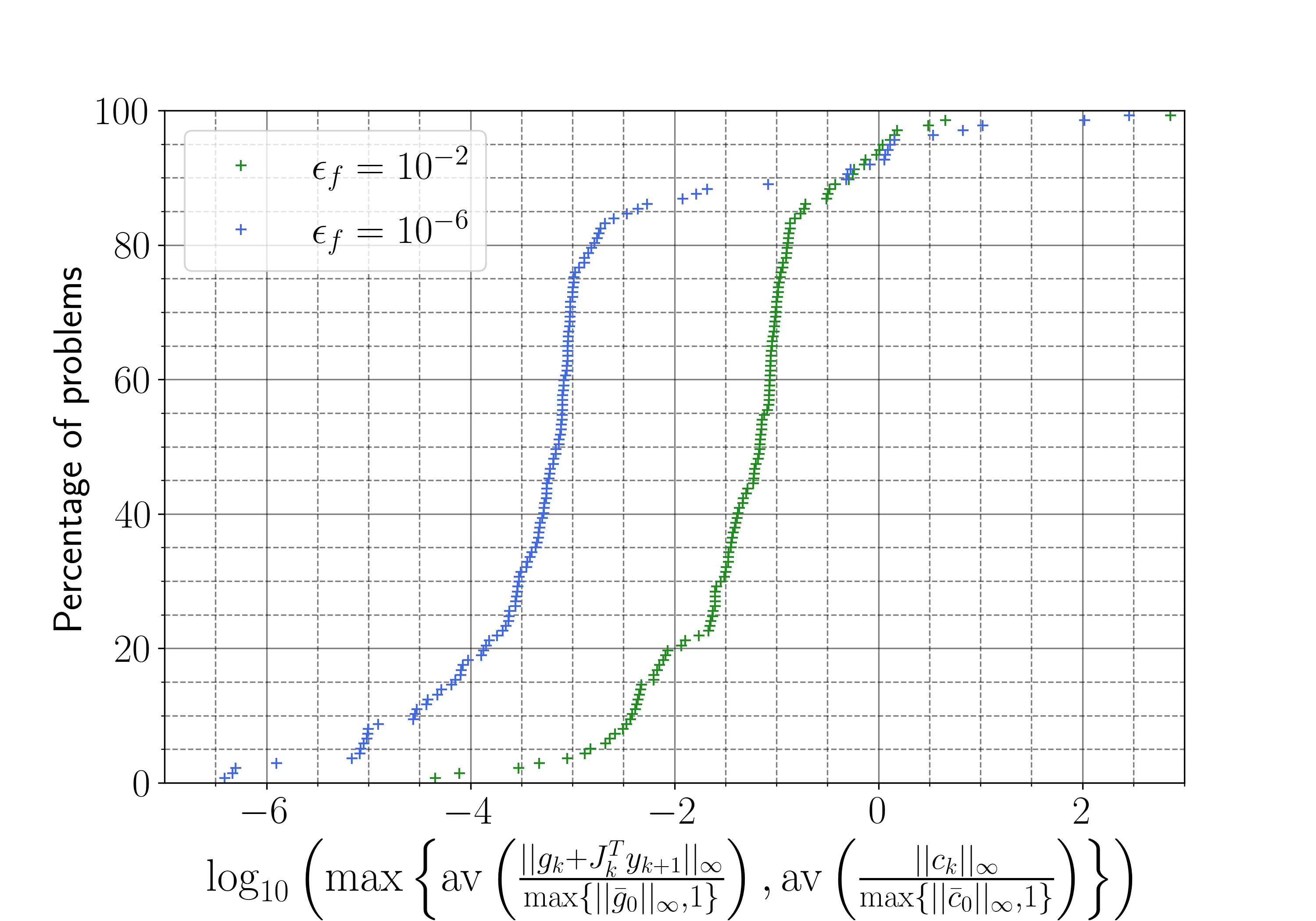}
  \caption{Stationarity w.r.t.~barrier subproblem}
  \label{fig:opt_mu_1e-4_noisy}
\end{subfigure}%
\begin{subfigure}{.5\textwidth}
  \centering
  \includegraphics[width=\linewidth]{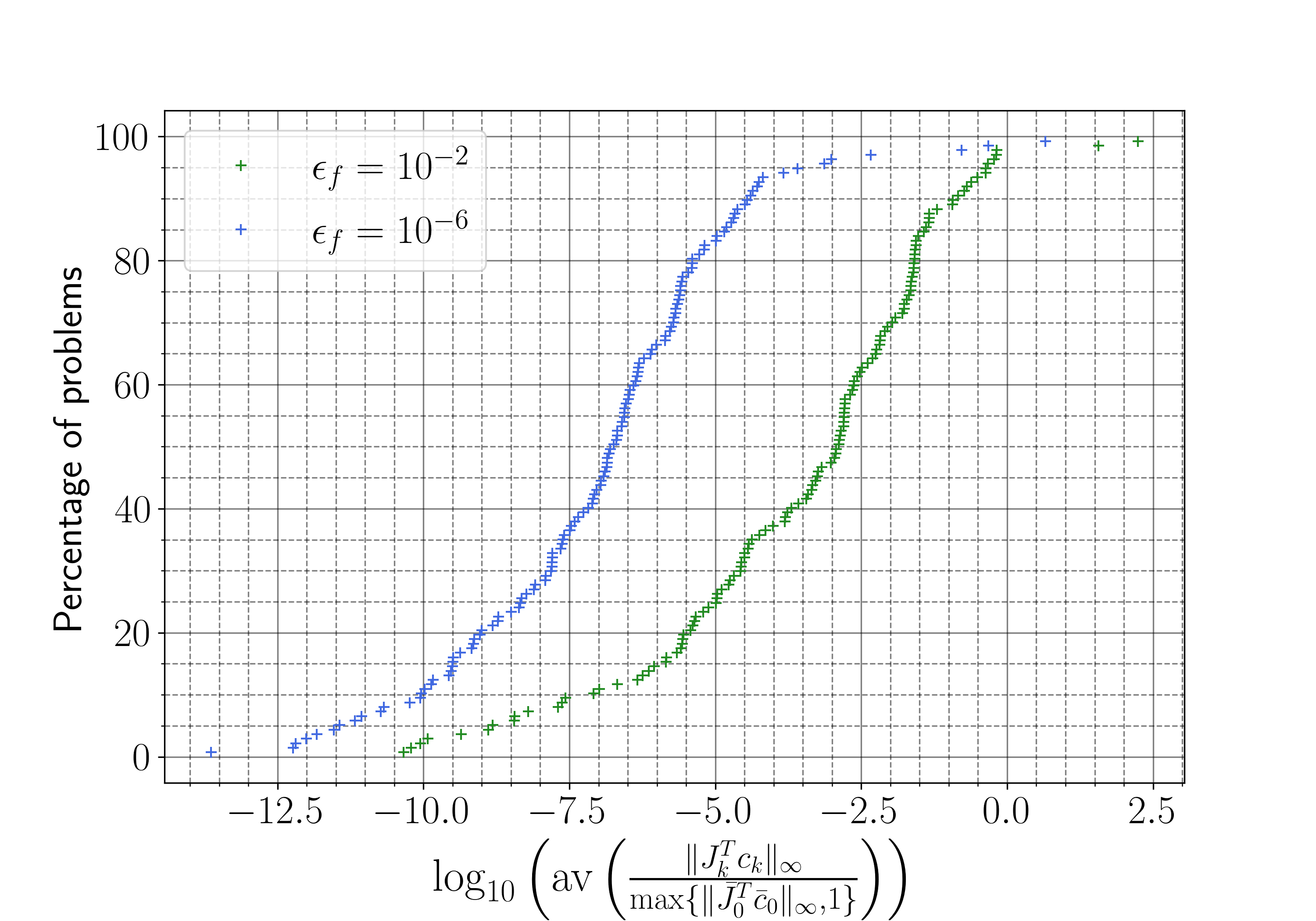}
  \caption{Stationarity w.r.t.~minimizing infeasibility}
  \label{fig:infeas_mu_1e-4_noisy}
\end{subfigure}
\caption{$\mu=10^{-4}$, measured with noisy function and derivative values}
\label{fig:mu_1e-4_noisy}
\end{figure} 

\begin{figure}[ht]
\centering
\begin{subfigure}{.5\textwidth}
  \centering
  \includegraphics[width=\linewidth]{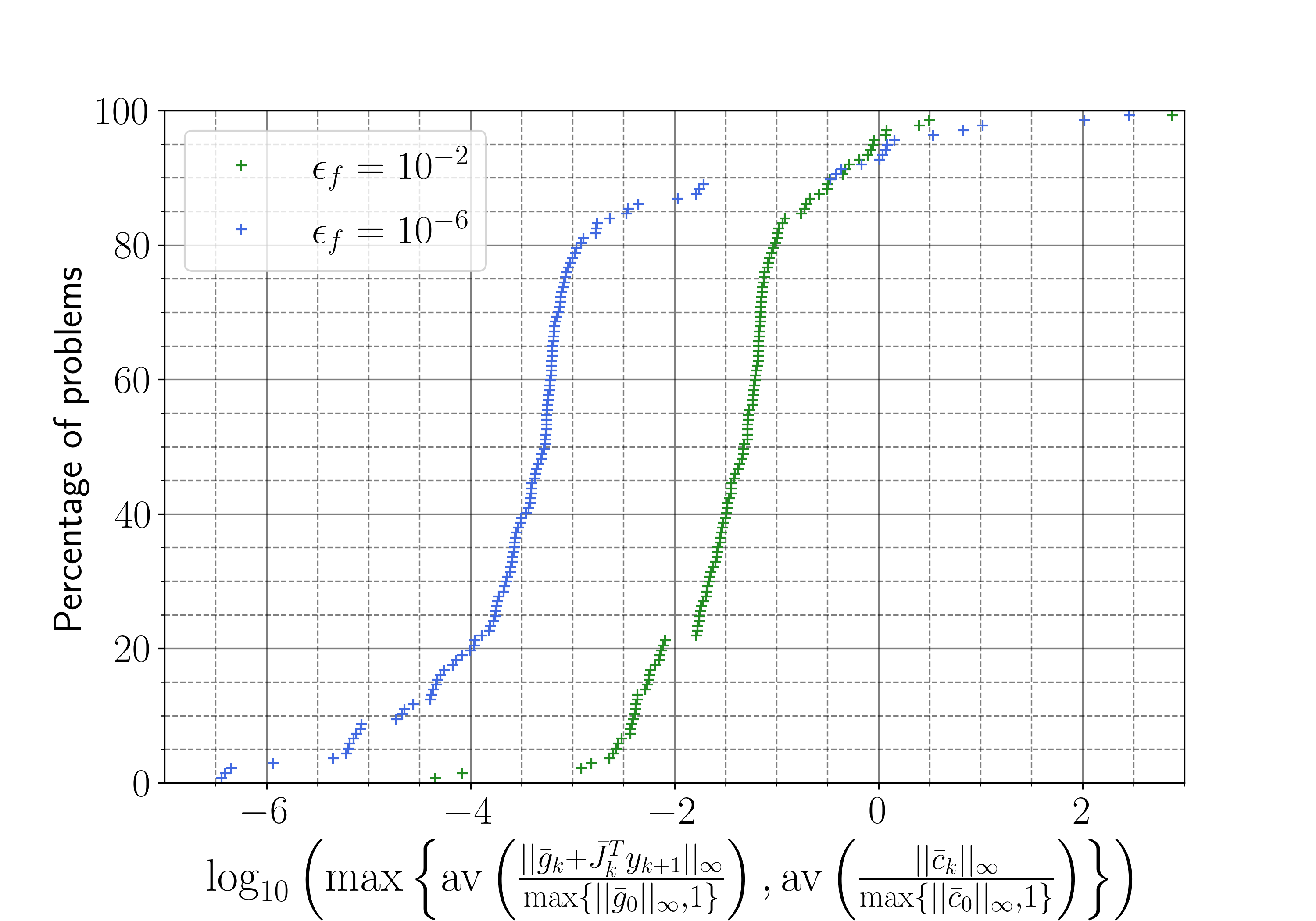}
  \caption{Stationarity w.r.t.~barrier subproblem}
  \label{fig:opt_mu_1e-4_det}
\end{subfigure}%
\begin{subfigure}{.5\textwidth}
  \centering
  \includegraphics[width=\linewidth]{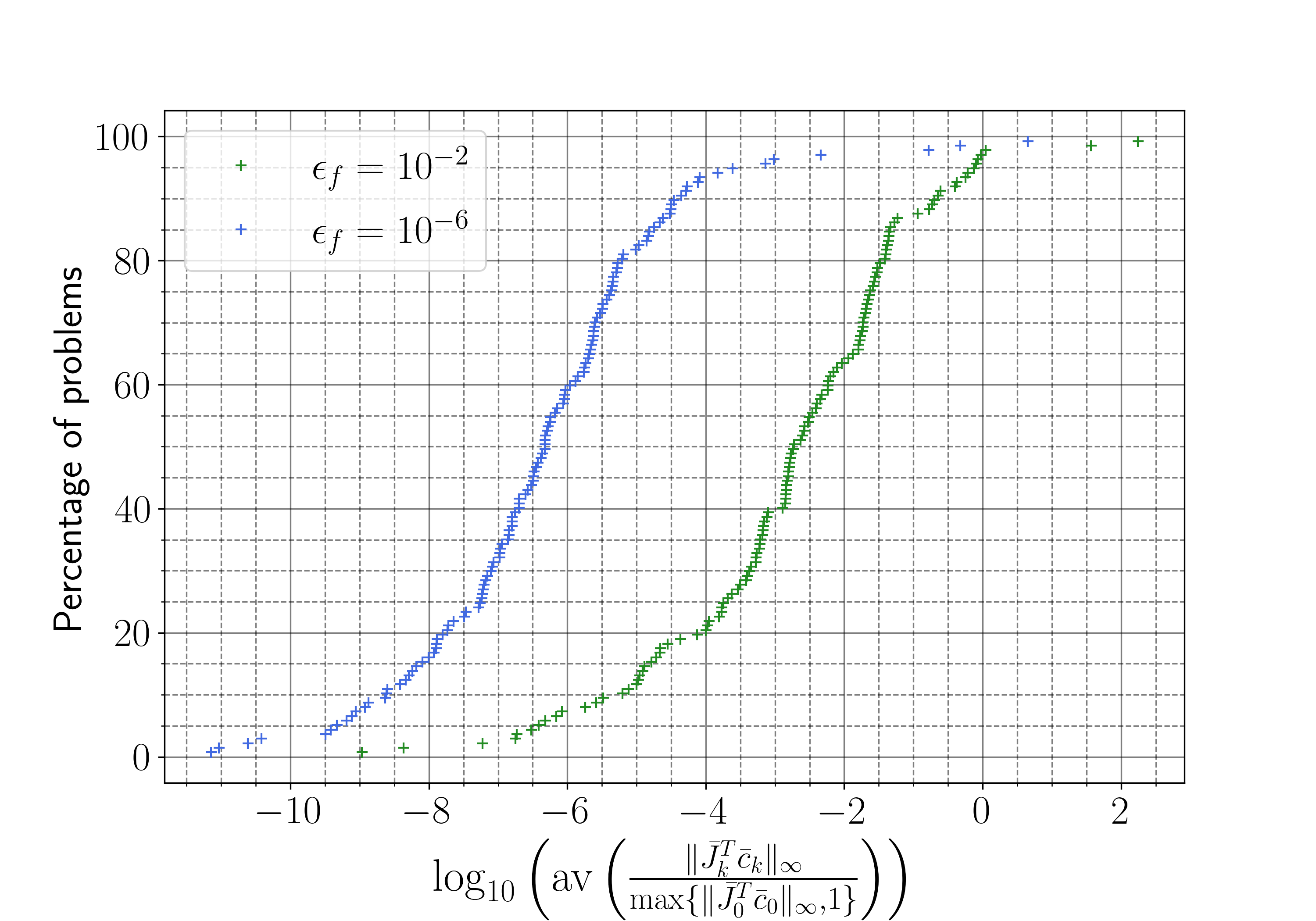}
  \caption{Stationarity w.r.t.~minimizing infeasibility}
  \label{fig:infeas_mu_1e-4_det}
\end{subfigure}
\caption{$\mu=10^{-4}$, measured with noiseless function and derivative values}
\label{fig:mu_1e-4_det}
\end{figure}

The results for $\mu = 10^{-4}$ are comparable to those with $\mu = 10^{-1}$.  On one hand, these results show that the performance of our algorithm is robust with respect to the barrier parameter value, which is good news.  At the same time, the results show that with a smaller barrier parameter, one should not necessarily expect to be able to obtain significantly better final solution estimates.  This can be attributed to the fact, which we mentioned at the end of \S\ref{sec.contributions}, that regardless of the barrier parameter value, one may be inhibited from obtaining high-accuracy solutions due to the presence of noise.


\section{Conclusion and Future Work}\label{sec.conclusion_4}

We have presented, analyzed, and tested the practical performance of an interior-point algorithm for solving continuous inequality-constrained optimization problems when function and derivative values are corrupted by noise.  Our algorithm focuses on solving a single barrier subproblem.  We have shown that the algorithm either terminates finitely due to one of two reasonable termination criteria, or it generates an infinite sequence of iterates, and in this latter case it is guaranteed under certain assumptions to reach a point at which a stationarity measure is below a threshold that depends on the noise level.

As previously mentioned in the paper and shown in our numerical experiments, in the noisy setting there is not necessarily a benefit of reducing the barrier parameter too small, since in any case the noise in the function and derivative values may inhibit the ability to acquire a highly accurate solution (to the underlying noiseless problem).  That being said, it is very reasonable to expect that, in some cases, one may obtain better solutions by solving a sequence of barrier subproblems for a diminishing sequence of barrier parameters.  We propose that this may be done by incorporating a practical stopping condition for a given barrier subproblem that builds off of the practical condition proposed in Theorem~3.9 in \cite{dezfulian2024convergence}.  This would essentially involve employing an iteration-dependent variant of \eqref{eq.stat_cond_0_v3} where the unknown quantities are estimated.  Since $(\eps_f,\eps_g,\eps_c,\eps_J)$ are presumed to be known, and both $g_{\sup}$ and $J_{\sup}$ are easily estimated during a run of the algorithm using \eqref{eq.J_g_bnd}, the main quantity that requires some attention is $\xi_d$.  By its definition in Lemma~\ref{lemma.d_bounded_by_delta_m}, one finds that this value is affected by the choice of $\omega$ in \eqref{eq.normal_tr_step}.  If $\omega$ is chosen to be a relatively large value, then $\xi_d$ becomes relatively small, which in turn can cause $\{\rho_{k,1},\rho_{k,2},\rho_{k,3}\}$ to become relatively large.  In some sense, this makes \eqref{eq.stat_cond_0_v3} easier to satisfy, but at the expense of only having a looser bound on $\Delta \mtrue_k (\dtrue_k,\tautrue)$.  By contrast, one may obtain a tighter, computable bound on $\Delta \mtrue_k (\dtrue_k,\tautrue)$ by replacing $\omega$ in the definition of $\eps_d$ with $\|J_k^T c_k\|/\|v_k\|$.  Passing this value through the definition of $\eps_d$, one can use the resulting right-hand side of \eqref{eq.stat_cond_0_v3} as a threshold for terminating the solve for a given barrier subproblem.  (Specifically, the solve may terminate if/when $\Delta m(d_k,\tau_k)$ is below this threshold.)  If/when this bound holds, the barrier parameter can be reduced using a typical update rule; see, e.g., \cite{wachter2006implementation}.
\bibliographystyle{plain}
\bibliography{references}
\appendix
\section{Convergence of $\|\Jtrue_k^T\ctrue_k\|^2$}\label{app.app}
Under looser assumptions than for the first part of Theorem~\ref{thm.main1}---specifically, without Assumptions~\ref{assumption.sigma_J} and \ref{ass.noise_level}---we prove in this appendix that Algorithm~\ref{alg.exact_alg} is guaranteed to generate a sequence of iterates such that, for some $k \in \NN$, one finds that $\|\Jtrue_{k }^T \ctrue_k\|^2$ is below a threshold that depends on noisy quantities.  The statement is similar to the first part of Theorem~\ref{thm.main1}.  As in Section~\ref{sec.convergence}, the result presumes that the algorithm does not terminate finitely.

\begin{theorem}\label{th.Jc}
  Suppose that Assumptions~\ref{assumption.err0}, \ref{assumption.bounded}, and \ref{assumption.s_bounded} hold, and for all $k \in \NN$, with $\xizero{\tau_k}$ defined in Lemma~\ref{lemma.dec_phi_delta_m}, $(\xithree,\xifour)$ defined in Lemma~\ref{lemma.delta_m_1}, $\epskd$ defined in~\eqref{eq.errors_k}, and $\alpha_k^{\max}$ computed to satisfy the fraction-to-the-boundary rule, let
  \begin{align}\label{eq.alpha_check}
    \check \alpha_k:= \min \left\{ \frac{1}{4 \xizero{\tau_k}} \left( \frac{(1 - \eta_\phi)\xithree \|\Jtrue_k^T \ctrue_k\|^2}{\xifour (\|\utrue_k\|^2 + \|\Jtrue_k^T \ctrue_k \|^2) + \|\epskd\|^2} \right), \alpha_k^{\max}\right\}.
  \end{align}
  Then, with $\Ekmone$ defined in Lemma~\ref{lemma.m1}, $\Ekmtwo$ defined in Lemma~\ref{lemma.m2}, and $\Ekm$ defined in Corollary~\ref{cor.m}, there exists $k \in \NN$ such that
  \begin{align}\label{eq.stat_cond_2}
    \|\Jtrue_{k }^T \ctrue_{ k }\|^2 \leq \max\left\{\frac{2 (\eta_\phi \Ekmone - (1 - \eta_\phi) \Ekmtwo)}{ (1 - \eta_\phi) \xithree }, - \frac{2\Ekm}{\xithree}, \frac{8(2 + \zeta) \vareps_k}{\check \alpha_k \eta_\phi \xithree} \right\}.
  \end{align}
\end{theorem}
\begin{proof}
  Consider arbitrary $k \in \NN$.  Our first aim is to show that if
  \begin{equation}\label{eq.stat_cond_2_Jc}
    \eta_\phi \Ekmone -(1 - \eta_\phi) \Ekmtwo \leq \tfrac{1}{2} (1 - \eta_\phi) \xithree \|\Jtrue_k^T \ctrue_k\|^2,
  \end{equation}
  then the relaxed Armijo condition~\eqref{eq.armijo} is satisfied for all $\alpha \leq \check \alpha_k$.  Toward this end, first observe that by \eqref{eq.errors2}, Lemma~\ref{lemma.m2}, and \ref{lemma.dec_phi_delta_m}, one finds for all $\alpha \in (0, \alpha_k^{\max}]$ that
  \begin{align*}
    \phi(z_k + \alpha \dhat_k, \tau_k) -  \phi(z_k, \tau_k)
      \leq&\ \phitrue(z_k + \alpha \dhat_k, \tau_k) - \phitrue(z_k, \tau_k) + 2 \vareps_k \\
      \leq&\ - \alpha \Delta \mtrue_k(\dtrue_k, \tautrue_k) - \alpha \Ekmtwo + \xizero{\tau_k} \alpha^2 \|d_k\|^2 + 2 \vareps_k \\ 
      =&\ - \alpha \eta_{\phi} \Delta \mtrue_k(\dtrue_k, \tautrue_k) - \alpha (1 - \eta_{\phi}) \Delta \mtrue_k(\dtrue_k, \tautrue_k) \\
      &\ - \alpha \Ekmtwo + \xizero{\tau_k} \alpha^2 \|d_k\|^2 + 2 \vareps_k.
    \end{align*}
    Consequently, by Corollary~\ref{cor.m} and $\|d_k\|^2 \leq 2 (\|\dtrue_k\|^2 + \|\epskd\|^2)$, one finds 
    \begin{align*}
      &\ \phi(z_k + \alpha \dhat_k, \tau_k) - \phi(z_k, \tau_k) \\
      \leq&\ - \alpha \eta_{\phi} \Delta m_k(d_k, \tau_k) + \alpha \eta_{\phi} \Ekm + 2 \vareps_k \\
      &\ - \alpha (1 - \eta_{\phi}) \Delta \mtrue_k(\dtrue_k, \tautrue_k) - \alpha \Ekmtwo + 2 \xizero{\tau_k} \alpha^2 (\|\dtrue_k\|^2 + \|\epskd\|^2) \\
      =&\ - \alpha \eta_{\phi} \Delta m_k(d_k, \tau_k)+ 2 \vareps_k \\
      &\ - \alpha (1 - \eta_{\phi}) \Delta \mtrue_k(\dtrue_k, \tautrue_k) +  2 \xizero{\tau_k} \alpha^2 (\|\dtrue_k\|^2 + \|\epskd\|^2) \\
      &\ + \alpha \eta_{\phi} \Ekmone - \alpha (1 - \eta_{\phi}) \Ekmtwo.
    \end{align*}
    Thus, under \eqref{eq.stat_cond_2_Jc}, one finds from Lemma \ref{lemma.delta_m_1} that
    \begin{align*}
      \phi(z_k + \alpha \dhat_k, \tau_k) -  \phi(z_k, \tau_k)
        \leq&\ - \alpha \eta_{\phi} \Delta m_k(d_k, \tau_k)+ 2 \vareps_k \\
        &\ -  \alpha (1 - \eta_{\phi}) \xithree (\tautrue_k \|\utrue_k\|^2 + \|\Jtrue_k^T \ctrue_k\|^2 ) \\
        &\ + 2 \xizero{\tau_k} \xifour \alpha^2 (\|\utrue_k\|^2 + \|\Jtrue_k^T \ctrue_k\|^2 ) + 2 \xizero{\tau_k} \alpha^2 \|\epskd\|^2 \\
        &\ + \tfrac{1}{2} \alpha (1 - \eta_\phi) \xithree \|\Jtrue_k^T \ctrue_k\|^2 \\
        \leq&\ - \alpha \eta_{\phi} \Delta m_k(d_k,\tau_k) + 2 \vareps_k  \\
        &\ -  \tfrac12 \alpha (1 - \eta_{\phi}) \xithree  \|\Jtrue_k^T \ctrue_k\|^2  \\
        &\ + 2 \xizero{\tau_k} \xifour \alpha^2 (\|\utrue_k\|^2 + \|\Jtrue_k^T \ctrue_k\|^2 ) + 2 \xizero{\tau_k} \alpha^2 \|\epskd\|^2,
    \end{align*}
    which along with \eqref{eq.alpha_check} implies that, for all $\alpha \leq \check \alpha_k$, \eqref{eq.armijo} holds, as claimed.
   
    Now, to derive a contradiction to the desired conclusion, suppose that \eqref{eq.stat_cond_2} does not hold for all $k \in \NN$.  Under this supposition, consider arbitrary $k \in \NN$.  Since \eqref{eq.stat_cond_2} does not hold, it follows that \eqref{eq.stat_cond_2_Jc} holds, which in turn means that~\eqref{eq.armijo} holds for some step size $\alpha \geq \check\alpha_k/2$.  Thus, with Corollary~\ref{cor.m}, one has
    \begin{align*}
      \phitrue(z_k + \alpha \dhat_k, \tau_k) - \phitrue(z_k, \tau_k)
      \leq&\ \phi(z_k + \alpha \dhat_k, \tau_k) - \phi(z_k, \tau_k) + 2 \vareps_k \\
      \leq&\ - \tfrac{1}{2} \check \alpha_k \eta_\phi \Delta m_k(d_k, \tau_k) + 2 \vareps_k + (2 + \zeta) \vareps_k \\
         = & - \tfrac{1}{2} \check \alpha_k \eta_\phi \Delta \mtrue_k(\dtrue_k, \tautrue_k) - \tfrac{1}{2} \check \alpha_k \eta_\phi \Ekm + 2 \vareps_k + (2 + \zeta) \vareps_k.
    \end{align*}
    Now by Lemma~\ref{lemma.delta_m_1}, $\|\Jtrue_k^T \ctrue_k\|^2 > -\tfrac{2\Ekm}{\xithree}$, and $\|\Jtrue_k^T \ctrue_k \|^2 > \frac{8(2 + \zeta) \vareps_k}{\check \alpha_k \eta_\phi \xithree}$, one finds
    \begin{align*}
      \phitrue(z_k + \alpha \dhat_k, \tau_k) - \phitrue(z_k, \tau_k)
      \leq&\ - \tfrac{1}{2} \check \alpha_k \eta_\phi \xithree \|\Jtrue_k^T \ctrue_k\|^2 -\tfrac{1}{2} \check \alpha_k \eta_\phi \Ekm + 2 \vareps_k + (2 + \zeta) \vareps_k \\
      <&\ - \tfrac{1}{4} \check \alpha_k \eta_\phi \xithree \| \Jtrue_k^T \ctrue_k \|^2 + 2 \vareps_k + (2 + \zeta) \vareps_k < - \zeta \vareps_k.
    \end{align*}
    However, with Remark~\ref{remark.slackreset} this implies that $\phitrue(z_{k+1}, \tau_k) -  \phitrue(z_k, \tau_k) < -\zeta \vareps_k$ for all $k \in \NN$, which contradicts Assumption~\ref{assumption.s_bounded}, specifically the fact that $\{\phitrue(z_k,\tau_k)\}$ is bounded below.  Consequently, \eqref{eq.stat_cond_2} must hold for some $k \in \NN$, as desired.
  \qed
\end{proof}

\end{document}